\documentclass[12pt,reqno]{amsart}

\usepackage{amssymb, graphicx}
\usepackage{hyperref,tikz,tabularx}
\usepackage[all]{xypic}
\usepackage{caption}
\usepackage{subcaption}
\usepackage{relsize}
\usepackage[normalem]{ulem}
\allowdisplaybreaks

\usetikzlibrary{arrows,decorations.markings}

\tikzset{
	>=angle 90, 
	shorten <=-1.5pt,
	shorten >=-1.5pt
}

\newcommand{\Z}{\mathbb{Z}}

\newcommand{\comment}[1]{}

\newtheorem*{maintheorem}{Main Theorem}
\newtheorem{symassumption}{Symmetry Assumption}
\newtheorem*{theorem*}{Theorem}
\newtheorem*{lemma*}{Lemma}
\newtheorem{proposition}{Proposition}[section]
\newtheorem{theorem}[proposition]{Theorem}
\newtheorem{corollary}[proposition]{Corollary}
\newtheorem{lemma}[proposition]{Lemma}

\theoremstyle{definition}

\newtheorem*{definition*}{Definition}
\newtheorem*{remark*}{Remark}

\begin{document}

\title{The 6-vertex model and deformations of the Weyl character formula}


\author{Ben Brubaker}
\address{Department of Mathematics, 127 Vincent Hall, 206 Church St. SE, Minneapolis, MN 55455}
\email{brubaker@math.umn.edu}

\author[Andrew Schultz]{Andrew Schultz}
\address{Department of Mathematics, Wellesley College, 106 Central Street, Wellesley, MA 02482}
\email{andrew.c.schultz@gmail.com}

\begin{abstract} We use statistical mechanics -- variants of the six-vertex model in the plane studied by
means of the Yang-Baxter equation -- to give new deformations of Weyl's character formula for classical groups of 
Cartan type $B, C,$ and $D$, and a character formula of Proctor for type $BC$. In each case, the corresponding Boltzmann weights are associated to the free fermion
point of the six-vertex model. These deformations add to the earlier known examples in types $A$ and $C$ by
Tokuyama and Hamel-King, respectively. A special case for classical types recovers deformations of the Weyl denominator
formula due to Okada. 
\end{abstract}

\date{\today}

\maketitle

\parskip=10pt plus 2pt minus 2pt
 
\section{Introduction}

Two of the most common explicit realizations of the character of a highest weight representation of general linear
groups are the Weyl character formula and as a generating function over a combinatorial basis (e.g., semi-standard
Young tableaux or Gelfand-Tsetlin patterns). 
In \cite{tokuyama}, to any dominant weight $\mu$ of $\text{GL}(n,\mathbb{C})$, 
Tokuyama gave a generating function identity which simultaneously generalized these two expressions 
for the character. Hamel and King subsequently
gave an analogous result for $Sp(2n)$ in \cite{hk-symplectic} and noted in \cite{hk-bijective} that, in both cases, the generating function may be expressed
as the partition function of a square ice (or 6-vertex) model in the plane whose boundary conditions are determined by the
highest weight $\mu$. In either case, given complex parameters $\boldsymbol{x} = (x_1, \ldots, x_n)$ and $t$, 
and writing $\mathcal{Z}^{\mu+\rho}(\boldsymbol{x}; t)$ for the generating function in Cartan types $A$ or $C$, then the identity
can be written
\begin{equation} \mathcal{Z}^{\mu+\rho}(\boldsymbol{x}; t) = \boldsymbol{x}^\rho \prod_{\alpha \in \Phi^+} (1 + t \boldsymbol{x}^{-\alpha}) s_\mu(\boldsymbol{x}). \label{tokuyamaform} \end{equation}
Here $s_\mu(\boldsymbol{x})$ denotes the character with highest weight $\mu$. Setting $t = -1$ and dividing by the Weyl denominator on both sides gives the Weyl character formula. Setting $t=0$, the states of the model that remain
on the left-hand side of (\ref{tokuyamaform}) are in bijection with Gelfand-Tsetlin patterns with top row corresponding to $\mu$. In the symplectic case, these patterns are originally due to Zhelobenko~\cite{zhelobenko}.

Moreover, the right-hand side of (\ref{tokuyamaform}) appears in the representation theory of reductive groups $G$ over a local field.
More precisely, let $q$ be the size of the residue field for the local field and let $\boldsymbol{x}$ be the Langlands parameters
for an unramified principal series representation of $G$. Setting $t = -1/q$, then (\ref{tokuyamaform}) agrees with the 
Casselman-Shalika formula for the spherical Whittaker function at a torus element corresponding
to $\mu$. This formula plays an outsized role in the theory of automorphic forms, as it is central to both the Langlands-Shahidi
and Rankin-Selberg methods. Better still, a special case of \cite{mcnamara} shows that the individual contributions of the generating function $\mathcal{Z}^{\mu+\rho}$
in type $A$ are natural for computing the Whittaker function. (An analogous result is expected to follow similarly in type $C$.) Theorem 7.5 and Corollary 8.7 of \cite{mcnamara}, together
with \cite[p. 1091]{bbf-annals} imply roughly that the individual terms in the generating function are the contribution to the Whittaker function coming from the local field analogue
of MV cycles. The MV cycles are subsets of the affine Grassmannian defined by Mirkovi{\'c} and Vilonen \cite{mv-satake} in the proof of the geometric Satake correspondence.

The first author, with Bump and Friedberg, showed in \cite{bbf-ice} that the generating function identity of Tokuyama may be proved using techniques from statistical mechanics,
particularly the (parametrized) Yang-Baxter equation, in the spirit of Baxter's book \cite{baxter}. Following Baxter, they show that ice models with boundary conditions determined by a 
dominant weight $\mu$ admit a Yang-Baxter equation whenever a single algebraic expression $\Delta$ in the Boltzmann weights is equal to 0 --- the so-called ``free-fermion point'' of the model \cite{fan-wu}. 
Thus the Boltzmann weights used by Tokuyama are but one choice among a large family of weights satisfying the Yang-Baxter equation and leading to similar identities. The symplectic
case was handled similarly in \cite{ivanov}. Definitions and results relating to lattice models will be reviewed in more detail in the next section. 

This paper presents deformations of the Weyl character formula in the spirit of Tokuyama's formula for classical groups of Cartan type $B, C,$ and $D$ and for type $BC$  using variants of the six-vertex model. See \cite{proctor} for the connection between character formulae for $BC$ and odd symplectic groups.
The models used for each type are rectangular grids in the plane with certain pairs of horizontal boundary edges identified; the boundary conditions and precise form are 
presented in the next section, as each Cartan type has a slightly particular definition. In type $A$, the corresponding models are simply rectangular grids; the models
for other classical groups and their boundary edge identifications are meant to reflect embeddings of the respective group (or dual group) into $\text{GL}(n)$. The Boltzmann
weights associated to each model are ``spectrally dependent;'' that is, they are allowed to vary based on the row of the grid in which they occur.  We will refer to these row numbers as the spectral indices.

Our inspiration for these models, apart from the aforementioned antecedents, comes from two sources. Okada \cite{okada} (and later, Simpson \cite{simpson}) explored deformations of the Weyl {\it denominator} in the
spirit of Tokuyama's formula for other classical groups. This corresponds to the case $\mu = 0$ in (\ref{tokuyamaform}). Okada's generating functions were indexed by symmetry classes of alternating
sign matrices. Second, Kuperberg \cite{kuperberg} used six-vertex models with similar boundary conditions and the Yang-Baxter equation to enumerate such symmetry classes of alternating sign matrices (though the corresponding
Boltzmann weights were not at the free-fermion point; the $\Delta$ alluded to above is a cube root of unity). Thus we suspected that Boltzmann weights at the free-fermion point might simultaneously explain Okada's
ad-hoc denominator deformations and generalize to deformations of the Weyl character formula of Tokuyama type. This is exactly what we will show in the subsequent sections, and
at the same time we'll produce additional ice type models beyond the scope of \cite{kuperberg} that lead to interesting, previously unknown deformations. 

The definitions of these models are rather subtle; connections to symmetry classes of alternating sign matrices suggest a rough form, but the precise description of the models required to apply Yang-Baxter-type methods is, a priori, difficult to pin down. It would be very desirable to have a general recipe for producing the exact models, say from natural functors on highest weight representations of the corresponding group, but we know of no such satisfactory description for all cases presented here.

Let $\mathfrak{M}^\lambda$ denote any of the ice models in the next section corresponding to one of classical types $B,C$ or $D$ and a strictly dominant weight $\lambda$ for the respective group. 
Thus $\lambda - \rho_\star$ is dominant, where $\rho_\star$ denotes the respective Weyl vector for $\star \in \{ B, C, D \}$, the sum of fundamental dominant weights. Let $\mathcal{Z}(\mathfrak{M}^\lambda)$ be the corresponding partition function for the model, as defined in (\ref{partitiondef}).

\begin{maintheorem} There exists a family of spectrally dependent, free-fermionic Boltzmann weights so that
\begin{enumerate}
\item $\mathcal{Z}(\mathfrak{M}^\rho)$ divides $\mathcal{Z}(\mathfrak{M}^\lambda)$ (as a polynomial expression in the Boltzmann weights). Moreover, the resulting quotient is symmetric under the Weyl group action on spectral indices. (Theorem~\ref{prop:divisibility.for.partition.functions} and Proposition \ref{prop:partition.functions.for.rho})
\item The partition function $\mathcal{Z}(\mathfrak{M}^\rho)$ factors as a product of homogeneous polynomials of degree at most 2 in the Boltzmann weights. The number of such factors
depends on the number of positive roots of the corresponding Lie type. (Proposition~\ref{prop:partition.functions.for.rho})
\item A particular choice of Boltzmann weights in the family, which we refer to as ``deformation weights'' (see Section~\ref{okadamatch}), has two further specializations:
\begin{enumerate}
\item One for which $\mathcal{Z}(\mathfrak{M}^{\rho})$ recovers Okada's deformed denominator identities \cite{okada} under a weight-preserving bijection between ice in $\mathfrak{M}^\rho$ and symmetry classes of alternating sign matrices. (Proposition~\ref{prop:okada.bijection})
\item And another such that $\mathcal{Z}(\mathfrak{M}^{\lambda})$ reduces to the Weyl character formula for the representation of highest weight $\lambda - \rho_\star$ on the Lie group corresponding to~$\star$. (Theorem~\ref{weylchartheorem})
\end{enumerate}
\end{enumerate}
The ice model for type $BC$ also satisfies the above properties, but following \cite{proctor}, the resulting character formula uses the Weyl group for type $D$ and the Weyl vector $\rho_B$ of type $B$.
\end{maintheorem}

A special case of the above results in type $B$, essentially using deformation weights, appeared as part of the Ph.D. thesis of Sawyer Tabony~\cite{tabony}.

Despite knowing all of these properties of the partition function $\mathcal{Z}(\mathfrak{M}^\lambda)$, we are not able to evaluate it explicitly for arbitrary strictly dominant weights. Specializing to the ``deformation weights,'' which are functions of parameters $\boldsymbol{x} = (x_1, \ldots, x_n)$ and $\boldsymbol{t} = (t_1, \ldots, t_n)$, then parts (1) and (2) of the main theorem show $\mathcal{Z}(\mathfrak{M}^\lambda)$ are divisible by a deformed Weyl denominator, just as in Tokuyama's result. However, unlike Tokuyama's identity, the resulting quotient is not independent of $\boldsymbol{t}$. So in particular, it is not equal to a highest weight character in $\boldsymbol{x}$ as in the earlier type $A$ and $C$ results (and hence not matching the Casselman-Shalika formula). Our methods are limited in that the Yang-Baxter equation only shows that the resulting partition function is symmetric and one needs additional methods to pin down the explicit form. Tabony~\cite{tabony} used additional combinatorial methods to prove deformed versions of Pieri's rule and Clebsch-Gordan coefficients for the type $B$ partition function with ``deformation weights'' which seem suggestive that these partition functions might have a natural representation theoretic characterization, e.g., as the characters of modules for a deformation of the corresponding Lie algebra.

During a 2010 workshop at BIRS in Banff, Canada, Brubaker discussed this problem with Hamel and asked if their tableaux methods might yield conjectures and/or proofs of these deformations. In forthcoming independent work, Hamel and King have arrived at similar deformations and made precise conjectures for each model and even proved the conjectures in select cases as of this writing. Their generating functions essentially use a generalized version of the deformation weights in the above theorem, though not the entire family at the free-fermion point $\Delta = 0$. It would be interesting to try other methods from statistical mechanics to obtain further information about the partition function (e.g., the quantum inverse scattering method as in~\cite{korepinetal}).

The authors thank Hamel and King for helpful conversations and their willingness to openly share portions of their work in progress. The computer algebra software Sage~\cite{sage} was used extensively to perform supporting computations. This work was partially supported by NSF grant DMS-1258675.

\section{Ice Models for Classical Groups \label{classyice}}

Let $M$ be one of the classical, complex matrix groups $GL(n)$, $Sp(2n)$, $SO(2n+1)$, $SO(2n)$, or the odd symplectic group $Sp(2n-1)$ as defined in \cite{proctor}. Given a strictly dominant weight $\lambda$ for the corresponding group $M$, we describe associated lattice models -- families $\mathfrak{M}^\lambda$ of planar graphs with fixed boundary depending on $\lambda$. In the language of statistical mechanics, the graphs in this family are the ``admissible states'' of the model $\mathfrak{M}^\lambda$. 

We begin by reviewing the model for $GL(n)$ studied in \cite{bbf-ice}, though the notation here is slightly altered.  Under the usual identification of the weight lattice with $\mathbb{Z}^n$, strictly dominant weights are expressible as $n$-tuples $\lambda = [\lambda_1,\cdots,\lambda_n]$ with $\lambda_1 > \cdots > \lambda_n$. For convenience,  we assume $\lambda_n \geq 1$. To each such $\lambda$,  we associate a rectangular grid of $n$ rows and $\lambda_1$ columns.  Starting from the top and proceeding to bottom, the rows are labelled $1$ through $n$; counting from right to left, the columns are numbered $1$ through $\lambda_1$.  The set $\mathfrak{A}^\lambda$ is then the collection of all directed graphs on this grid such that every vertex has in-degree equal to out-degree, and such that the following boundary conditions hold:
\begin{itemize}
\item the far-left and far-right edges of each row  point inward;
\item the bottom edge of every column  points outward;
\item for any column whose index equals a part of $\lambda$, the top edge  points outward;
\item for any column whose index doesn't equal a part of $\lambda$, the top edge  points inward.
\end{itemize}  Figure \ref{fig:sample.A3} gives an example of a state in $\mathfrak{A}^{[5,4,2]}$.

\begin{figure}[!ht]
\begin{tikzpicture}

\node [label=left:$1$] at (0,2) {};
\node [label=left:$2$] at (0,1) {};
\node [label=left:$3$] at (0,0) {};

\node [label=above:$5$] at (0.5,2.5) {};
\node [label=above:$4$] at (1.5,2.5) {};
\node [label=above:$3$] at (2.5,2.5) {};
\node [label=above:$2$] at (3.5,2.5) {};
\node [label=above:$1$] at (4.5,2.5) {};

\draw [>-] (0,2) -- (1,2);
\draw [>-] (0,1) -- (1,1);
\draw [>-] (0,0) -- (1,0);

\draw [<-] (0.5,2.5) -- (0.5,1.5);
\draw [<-<] (0.5,1.5) -- (0.5,.5);
\draw [->] (0.5,0.5) -- (0.5,-0.5);

\draw [>-] (1,2) -- (2,2);
\draw [>-] (1,1) -- (2,1);
\draw [<-] (1,0) -- (2,0);

\draw [<-] (1.5,2.5) -- (1.5,1.5);
\draw [>->] (1.5,1.5) -- (1.5,.5);
\draw [->] (1.5,0.5) -- (1.5,-0.5);

\draw [<-] (2,2) -- (3,2);
\draw [>-] (2,1) -- (3,1);
\draw [<-] (2,0) -- (3,0);

\draw [>-] (2.5,2.5) -- (2.5,1.5);
\draw [<->] (2.5,1.5) -- (2.5,.5);
\draw [->] (2.5,0.5) -- (2.5,-0.5);

\draw [>-] (3,2) -- (4,2);
\draw [<-] (3,1) -- (4,1);
\draw [<-] (3,0) -- (4,0);

\draw [<-] (3.5,2.5) -- (3.5,1.5);
\draw [>->] (3.5,1.5) -- (3.5,.5);
\draw [->] (3.5,0.5) -- (3.5,-0.5);

\draw [<-<] (4,2) -- (5,2);
\draw [<-<] (4,1) -- (5,1);
\draw [<-<] (4,0) -- (5,0);

\draw [>-] (4.5,2.5) -- (4.5,1.5);
\draw [>->] (4.5,1.5) -- (4.5,.5);
\draw [->] (4.5,0.5) -- (4.5,-0.5);

\end{tikzpicture}
\caption{An admissible state in $\mathfrak{A}^{[5,4,2]}$.}
\label{fig:sample.A3}
\end{figure}
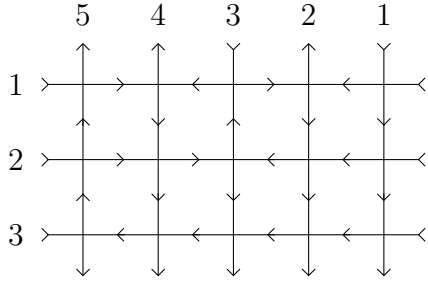

Combining this idea from \cite{bbf-ice} of boundary conditions corresponding to a dominant weight with the lattice models from \cite{kuperberg} for symmetry classes of alternating sign matrices, we can produce new families of models for each of the other matrix groups listed above. Each new family of graphs consists of rectangular grids with a certain ``u-turn" boundary, a collection of nested u-turn bends along the right-hand side.  For this reason, these models will collectively be referred to as the ``bent ice" diagrams. For the sake of uniformity, these other models will continue to be associated to integer $n$-tuples $\lambda = [\lambda_1,\cdots,\lambda_n]$ with $\lambda_1 > \cdots > \lambda_n \geq 1$; hence $n$ will always refer to the number of parts of $\lambda$, and $\lambda_1$ is always the greatest part of $\lambda$. Similarly, continue to denote $\rho = [n, n-1, \ldots, 1]$ in all cases. Thus $\lambda$ and $\rho$ are not the usual coordinatization of strictly dominant weights for all such Cartan types, but provide a convenient and uniform numbering scheme for columns. In each case, as with $\mathfrak{A}^\lambda$, the admissible states of the model are all directed graphs satisfying certain boundary conditions corresponding to $\lambda$, and for which the in-degree and out-degree at each vertex are equal. Thus it suffices to just describe the shape and boundary conditions for each model.

The first such family will be denoted $\mathfrak{B}^\lambda$.  The underlying graph has $2n$ rows and $\lambda_1$ columns.  The top $n$ rows are labelled $1$ through $n$, and the bottom rows are labelled $\overline{n}$ through $\overline{1}$.  There is an edge connecting the far right side of row $j$ with row $\overline{j}$, and we place a vertex on each bend; these bent edges are nested so that the resultant graph is planar.  Columns are numbered $1$ through $\lambda_1$ as we move from right to left.  
See Figure \ref{fig:type.B.configs} for an example of an admissible state.  

A related family is denoted $\mathfrak{B}_*^\lambda$; it shares the same boundary conditions as $\mathfrak{B}^\lambda$, but includes an additional central row (labeled 0) which does not have a bend attached to it; the far left arrow on this central row points inward, and the far right arrow on this central row points outward. Row labeling from top to bottom is $1$ through $n$, then $0$, then $\overline{n}$ through $\overline{1}$. Figure \ref{fig:type.B.configs} includes an example from this family.


\begin{figure}
\centering
\begin{subfigure}[!ht]{.45\textwidth}
\centering
\begin{tikzpicture}[scale=.75]
\node [label=left:$1$] at (0,2) {};
\node [label=left:$2$] at (0,1) {};
\node [label=left:$\overline{2}$] at (0,0) {};
\node [label=left:$\overline{1}$] at (0,-1) {};

\node [label=above:$4$] at (0.5,2.5) {};
\node [label=above:$3$] at (1.5,2.5) {};
\node [label=above:$2$] at (2.5,2.5) {};
\node [label=above:$1$] at (3.5,2.5) {};

\draw [>-] (0,2) -- (1,2);
\draw [>-] (0,1) -- (1,1);
\draw [>-] (0,0) -- (1,0);
\draw [>-] (0,-1) -- (1,-1);

\draw [<-] (0.5,2.5) -- (0.5,1.5);
\draw [<-] (0.5,1.5) -- (0.5,.5);
\draw [<-] (0.5,0.5) -- (0.5,-0.5);
\draw [<->] (0.5,-0.5) -- (0.5,-1.5);

\draw [>-] (1,2) -- (2,2);
\draw [>-] (1,1) -- (2,1);
\draw [>-] (1,0) -- (2,0);
\draw [<-] (1,-1) -- (2,-1);

\draw [>-] (1.5,2.5) -- (1.5,1.5);
\draw [>-] (1.5,1.5) -- (1.5,.5);
\draw [>-] (1.5,0.5) -- (1.5,-0.5);
\draw [>->] (1.5,-0.5) -- (1.5,-1.5);

\draw [>-] (2,2) -- (3,2);
\draw [>-] (2,1) -- (3,1);
\draw [>-] (2,0) -- (3,0);
\draw [<-] (2,-1) -- (3,-1);

\draw [<-] (2.5,2.5) -- (2.5,1.5);
\draw [<-] (2.5,1.5) -- (2.5,.5);
\draw [>-] (2.5,0.5) -- (2.5,-0.5);
\draw [>->] (2.5,-0.5) -- (2.5,-1.5);

\draw [>->] (3,2) -- (4,2);
\draw [<-<] (3,1) -- (4,1);
\draw [>->] (3,0) -- (4,0);
\draw [<-<] (3,-1) -- (4,-1);

\draw [>-] (3.5,2.5) -- (3.5,1.5);
\draw [>-] (3.5,1.5) -- (3.5,.5);
\draw [>-] (3.5,0.5) -- (3.5,-0.5);
\draw [>->] (3.5,-0.5) -- (3.5,-1.5);

\draw [>-]
	(4,2) arc (90:0:1.5);
\draw [*->]
	(5.5,0.5) arc (0:-90:1.5);

\draw [<-]
	(4,1) arc (90:0:.5);
\draw [*-<]
	(4.5,0.5) arc (0:-90:.5);
\end{tikzpicture}
\end{subfigure}
\begin{subfigure}[h]{.45\textwidth}
  \centering
\begin{tikzpicture}[scale=.65]

\node [label=left:$1$] at (1,2) {};
\node [label=left:$2$] at (1,1) {};
\node [label=left:$0$] at (1,0) {};
\node [label=left:$\overline{2}$] at (1,-1) {};
\node [label=left:$\overline{1}$] at (1,-2) {};

\node [label=above:$4$] at (1.5,2.5) {};
\node [label=above:$3$] at (2.5,2.5) {};
\node [label=above:$2$] at (3.5,2.5) {};
\node [label=above:$1$] at (4.5,2.5) {};

\draw [>-] (1,2) -- (2,2);
\draw [>-] (1,1) -- (2,1);
\draw [>-] (1,0) -- (2,0);
\draw [>-] (1,-1) -- (2,-1);
\draw [>-] (1,-2) -- (2,-2);

\draw [<-] (1.5,2.5) -- (1.5,1.5);
\draw [>-] (1.5,1.5) -- (1.5,.5);
\draw [>-] (1.5,0.5) -- (1.5,-0.5);
\draw [>-] (1.5,-0.5) -- (1.5,-1.5);
\draw [>->] (1.5,-1.5) -- (1.5,-2.5);

\draw [<-] (2,2) -- (3,2);
\draw [>-] (2,1) -- (3,1);
\draw [>-] (2,0) -- (3,0);
\draw [>-] (2,-1) -- (3,-1);
\draw [>-] (2,-2) -- (3,-2);

\draw [>-] (2.5,2.5) -- (2.5,1.5);
\draw [<-] (2.5,1.5) -- (2.5,.5);
\draw [>-] (2.5,0.5) -- (2.5,-0.5);
\draw [>-] (2.5,-0.5) -- (2.5,-1.5);
\draw [>->] (2.5,-1.5) -- (2.5,-2.5);

\draw [>-] (3,2) -- (4,2);
\draw [<-] (3,1) -- (4,1);
\draw [>-] (3,0) -- (4,0);
\draw [>-] (3,-1) -- (4,-1);
\draw [>-] (3,-2) -- (4,-2);

\draw [<-] (3.5,2.5) -- (3.5,1.5);
\draw [<-] (3.5,1.5) -- (3.5,.5);
\draw [<-] (3.5,0.5) -- (3.5,-0.5);
\draw [>-] (3.5,-0.5) -- (3.5,-1.5);
\draw [>->] (3.5,-1.5) -- (3.5,-2.5);

\draw [>->] (4,2) -- (5,2);
\draw [<-<] (4,1) -- (5,1);
\draw [<->] (4,0) -- (5,0);
\draw [>->] (4,-1) -- (5,-1);
\draw [>-<] (4,-2) -- (5,-2);

\draw [>-] (4.5,2.5) -- (4.5,1.5);
\draw [>-] (4.5,1.5) -- (4.5,.5);
\draw [>-] (4.5,0.5) -- (4.5,-0.5);
\draw [<-] (4.5,-0.5) -- (4.5,-1.5);
\draw [<->] (4.5,-1.5) -- (4.5,-2.5);


\draw [>-]
	(5,2) arc (90:0:2);
\draw [*->]
	(7,0) arc (0:-90:2);

\draw [<-]
	(5,1) arc (90:0:1);
\draw [*-<]
	(6,0) arc (0:-90:1);

\end{tikzpicture}
 
\end{subfigure}

\caption{Admissible states of $\mathfrak{B}^{[4,2]}$ (left) and  $\mathfrak{B}_*^{[4,2]}$ (right).}
\label{fig:type.B.configs}
\end{figure}
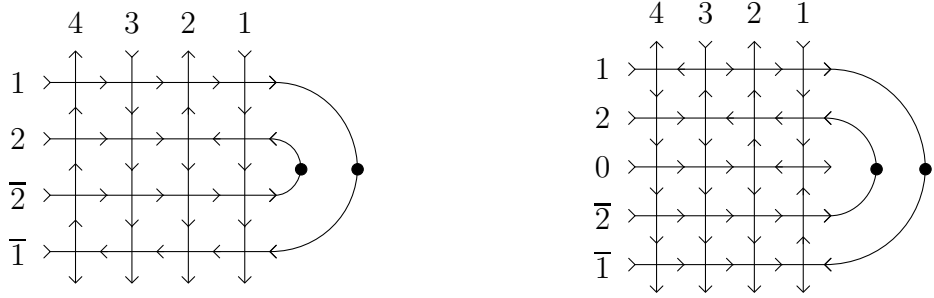

The next family --- denoted $\mathfrak{C}^\lambda$ --- also shares the boundary conditions of the $\mathfrak{B}^\lambda$ model, but has both an additional central row and an additional half column to the right of all the full columns and passing through the lower half of rows. This row and half column are connected by a vertex in the center of the diagram, and each is labelled $0$.  The two adjacent edges to the new vertex are internal edges of the model. The left-most edge of row 0 points inward as before, and the bottom-most edge of the half-column 0 has an outward pointing arrow.  
We show an example in Figure \ref{fig:type.C.configs}.

The family $\mathfrak{C}_*^\lambda$ has a structure similar to that of the $\mathfrak{B}^\lambda$ family, but has an additional half column to the right of the last full column in the diagram.  This half column (labeled 0) only intersects the bottom rows labeled $\overline{n}$ to $\overline{1}$.  The boundary edges are decorated with arrows as in the $\mathfrak{B}^\lambda$ model, and the top edge of the half column is always decorated with an inward pointing arrow, regardless of $\lambda$.  Figure \ref{fig:type.C.configs} also includes an example of an admissible state of $\mathfrak{C}_*^{\lambda}$.

\begin{figure}
\centering
\begin{subfigure}[!ht]{.45\textwidth}
\centering
  \begin{tikzpicture}[scale=.75]

\node [label=left:$1$] at (1,2) {};
\node [label=left:$2$] at (1,1) {};
\node [label=left:$0$] at (1,0) {};
\node [label=left:$\overline{2}$] at (1,-1) {};
\node [label=left:$\overline{1}$] at (1,-2) {};

\node [label=above:$3$] at (1.5,2.5) {};
\node [label=above:$2$] at (2.5,2.5) {};
\node [label=above:$1$] at (3.5,2.5) {};
\node [label=above:$0$] at (4.5,2.5) {};

\draw [>-] (1,2) -- (2,2);
\draw [>-] (1,1) -- (2,1);
\draw [>-] (1,0) -- (2,0);
\draw [>-] (1,-1) -- (2,-1);
\draw [>-] (1,-2) -- (2,-2);

\draw [<-] (1.5,2.5) -- (1.5,1.5);
\draw [<-] (1.5,1.5) -- (1.5,.5);
\draw [<-] (1.5,0.5) -- (1.5,-0.5);
\draw [>-] (1.5,-0.5) -- (1.5,-1.5);
\draw [>->] (1.5,-1.5) -- (1.5,-2.5);

\draw [>-] (2,2) -- (3,2);
\draw [>-] (2,1) -- (3,1);
\draw [<-] (2,0) -- (3,0);
\draw [>-] (2,-1) -- (3,-1);
\draw [>-] (2,-2) -- (3,-2);

\draw [>-] (2.5,2.5) -- (2.5,1.5);
\draw [>-] (2.5,1.5) -- (2.5,.5);
\draw [>-] (2.5,0.5) -- (2.5,-0.5);
\draw [>-] (2.5,-0.5) -- (2.5,-1.5);
\draw [>->] (2.5,-1.5) -- (2.5,-2.5);

\draw [>-<] (3,2) -- (4,2);
\draw [>->] (3,1) -- (4,1);
\draw [<-] (3,0) -- (4,0);
\draw [>-] (3,-1) -- (4,-1);
\draw [>-] (3,-2) -- (4,-2);

\draw [<-] (3.5,2.5) -- (3.5,1.5);
\draw [>-] (3.5,1.5) -- (3.5,.5);
\draw [>-] (3.5,0.5) -- (3.5,-0.5);
\draw [<-] (3.5,-0.5) -- (3.5,-1.5);
\draw [>->] (3.5,-1.5) -- (3.5,-2.5);

\draw [-] (4,2) -- (5,2);
\draw [-] (4,1) -- (5,1);
\draw [<-] (4,-1) -- (5,-1);
\draw [>-] (4,-2) -- (5,-2);


\draw [-]
	(5,2) arc (90:0:2);
\draw [*-<]
	(7,0) arc (0:-90:2);

\draw [-]
	(5,1) arc (90:0:1);
\draw [*->]
	(6,0) arc (0:-90:1);


	\draw [>-,decoration={markings, mark=at position 0.58 with {\arrow{*}}},postaction={decorate}]	
		(4,0) -- (4.5,0) -- (4.5,-.5);
	\draw [>-]	(4.5,-.5) -- (4.5,-1.5);
	\draw [>->]	(4.5,-1.5) -- (4.5,-2.5);

\end{tikzpicture}
\end{subfigure}
\begin{subfigure}{.45\textwidth}
\centering
  \begin{tikzpicture}[scale=.75]

\node [label=left:$1$] at (1,2) {};
\node [label=left:$2$] at (1,1) {};
\node [label=left:$\overline{2}$] at (1,0) {};
\node [label=left:$\overline{1}$] at (1,-1) {};

\node [label=above:$3$] at (1.5,2.5) {};
\node [label=above:$2$] at (2.5,2.5) {};
\node [label=above:$1$] at (3.5,2.5) {};
\node [label=above:$0$] at (4.5,2.5) {};

\draw [>-] (1,2) -- (2,2);
\draw [>-] (1,1) -- (2,1);
\draw [>-] (1,0) -- (2,0);
\draw [>-] (1,-1) -- (2,-1);

\draw [<-] (1.5,2.5) -- (1.5,1.5);
\draw [>-] (1.5,1.5) -- (1.5,.5);
\draw [>-] (1.5,0.5) -- (1.5,-.5);
\draw [>->] (1.5,-0.5) -- (1.5,-1.5);

\draw [<-] (2,2) -- (3,2);
\draw [>-] (2,1) -- (3,1);
\draw [>-] (2,0) -- (3,0);
\draw [>-] (2,-1) -- (3,-1);

\draw [<-] (2.5,2.5) -- (2.5,1.5);
\draw [<-] (2.5,1.5) -- (2.5,.5);
\draw [>-] (2.5,0.5) -- (2.5,-.5);
\draw [>->] (2.5,-.5) -- (2.5,-1.5);

\draw [<-] (3,2) -- (4,2);
\draw [<-] (3,1) -- (4,1);
\draw [>-] (3,0) -- (4,0);
\draw [>-] (3,-1) -- (4,-1);

\draw [>-] (3.5,2.5) -- (3.5,1.5);
\draw [<-] (3.5,1.5) -- (3.5,.5);
\draw [<-] (3.5,0.5) -- (3.5,-.5);
\draw [>->] (3.5,-.5) -- (3.5,-1.5);

\draw [>-] (4,2) -- (5,2);
\draw [<-] (4,1) -- (5,1);
\draw [<-] (4,0) -- (5,0);
\draw [>-] (4,-1) -- (5,-1);

\draw [-]
	(5,2) arc (90:0:1.5);
\draw [*->]
	(6.5,0.5) arc (0:-90:1.5);

\draw [-]
	(5,1) arc (90:0:.5);
\draw [*-<]
	(5.5,0.5) arc (0:-90:.5);

	\draw [>-]	(4.5,.5) -- (4.5,-.5);
	\draw [<->]	(4.5,-.5) -- (4.5,-1.5);

\end{tikzpicture}
\end{subfigure}
\caption{Elements of $\mathfrak{C}^{[3,1]}$ (left) and $\mathfrak{C}_*^{[3,2]}$ (right).}
\label{fig:type.C.configs}

\end{figure}
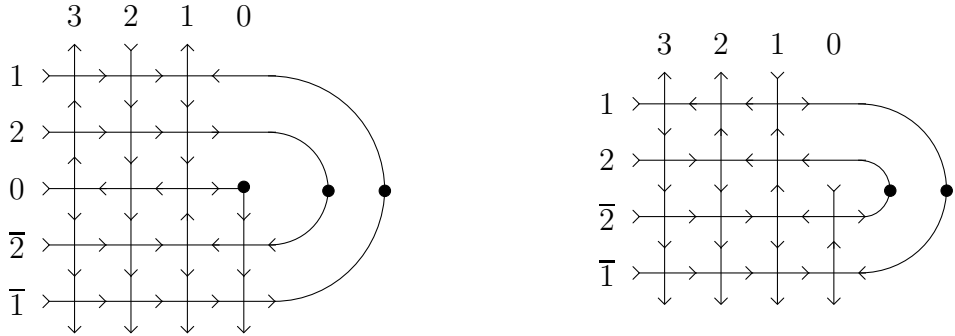

Our next family is denoted $\mathfrak{D}^{\lambda}$; it is much like $\mathfrak{C}_*^{\lambda}$, but now the half column is labeled $1$, so the first full column to its left is labeled 2, etc. Moreover, when assigning arrows along the top of each column according to $\lambda$, we \emph{do} include the half column in our count.

Lastly, there's a family we call $\mathfrak{BC}^{\lambda}$.  Its $n$th row is central and has no bend attached to it, much like the $\mathfrak{B}_*^\lambda$ model (though this central row is labeled $n$ in this model instead of $0$).  Hence the total number of pairs of rows connected by a u-turn bend is one less than the number of terms in the partition.  The boundary conditions on the $n$-th row are set so that both ends point inward.  Examples of elements of $\mathfrak{D}^\lambda$ and $\mathfrak{BC}^\lambda$ can be found in Figure \ref{fig:type.D.configs}. 

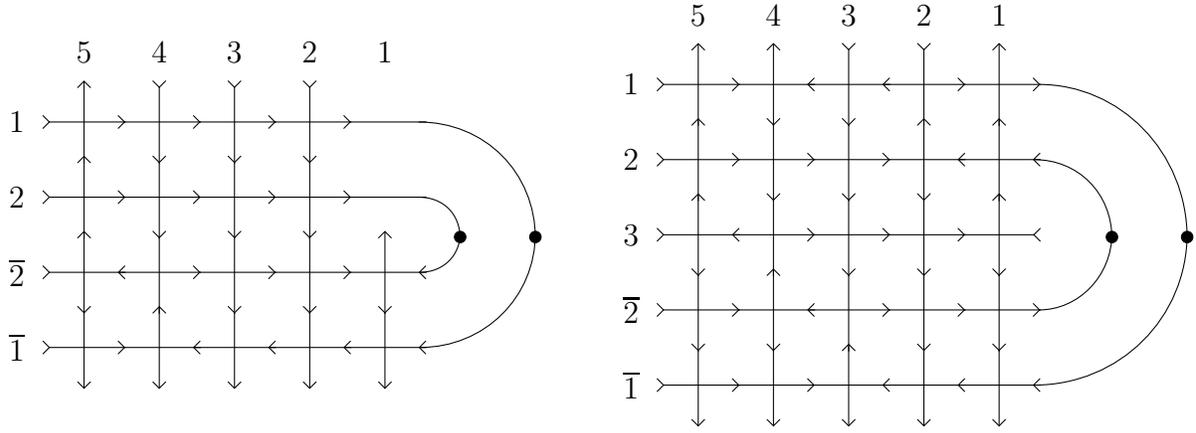
\begin{figure}
\centering

\begin{subfigure}[!ht]{.49\textwidth}
  \centering
  \begin{tikzpicture}

\node [label=left:$1$] at (2,2) {};
\node [label=left:$2$] at (2,1) {};
\node [label=left:$\overline{2}$] at (2,0) {};
\node [label=left:$\overline{1}$] at (2,-1) {};

\node [label=above:$5$] at (2.5,2.5) {};
\node [label=above:$4$] at (3.5,2.5) {};
\node [label=above:$3$] at (4.5,2.5) {};
\node [label=above:$2$] at (5.5,2.5) {};
\node [label=above:$1$] at (6.5,2.5) {};

\draw [>-] (2,2) -- (3,2);
\draw [>-] (2,1) -- (3,1);
\draw [>-] (2,0) -- (3,0);
\draw [>-] (2,-1) -- (3,-1);

\draw [<-] (2.5,2.5) -- (2.5,1.5);
\draw [<-] (2.5,1.5) -- (2.5,.5);
\draw [<-] (2.5,0.5) -- (2.5,-.5);
\draw [>->] (2.5,-.5) -- (2.5,-1.5);

\draw [>-] (3,2) -- (4,2);
\draw [>-] (3,1) -- (4,1);
\draw [<-] (3,0) -- (4,0);
\draw [>-] (3,-1) -- (4,-1);

\draw [>-] (3.5,2.5) -- (3.5,1.5);
\draw [>-] (3.5,1.5) -- (3.5,.5);
\draw [>-] (3.5,0.5) -- (3.5,-.5);
\draw [<->] (3.5,-.5) -- (3.5,-1.5);

\draw [>-] (4,2) -- (5,2);
\draw [>-] (4,1) -- (5,1);
\draw [>-] (4,0) -- (5,0);
\draw [<-] (4,-1) -- (5,-1);

\draw [>-] (4.5,2.5) -- (4.5,1.5);
\draw [>-] (4.5,1.5) -- (4.5,.5);
\draw [>-] (4.5,0.5) -- (4.5,-.5);
\draw [>->] (4.5,-.5) -- (4.5,-1.5);

\draw [>-] (5,2) -- (6,2);
\draw [>-] (5,1) -- (6,1);
\draw [>-] (5,0) -- (6,0);
\draw [<-] (5,-1) -- (6,-1);

\draw [>-] (5.5,2.5) -- (5.5,1.5);
\draw [>-] (5.5,1.5) -- (5.5,.5);
\draw [>-] (5.5,0.5) -- (5.5,-.5);
\draw [>->] (5.5,-.5) -- (5.5,-1.5);

\draw [>-] (6,2) -- (7,2);
\draw [>-] (6,1) -- (7,1);
\draw [>-] (6,0) -- (7,0);
\draw [<-] (6,-1) -- (7,-1);

\draw [-]
	(7,2) arc (90:0:1.5);
\draw [*->]
	(8.5,0.5) arc (0:-90:1.5);

\draw [-]
	(7,1) arc (90:0:.5);
\draw [*->]
	(7.5,0.5) arc (0:-90:.5);

\draw [<-]	(6.5,.5) -- (6.5,-.5);
\draw [>->]	(6.5,-.5) -- (6.5,-1.5);

\end{tikzpicture}
\end{subfigure}
\begin{subfigure}[!ht]{.5\textwidth}
\centering
\begin{tikzpicture}[scale=1]

\node [label=left:$1$] at (1,2) {};
\node [label=left:$2$] at (1,1) {};
\node [label=left:$3$] at (1,0) {};
\node [label=left:$\overline{2}$] at (1,-1) {};
\node [label=left:$\overline{1}$] at (1,-2) {};

\node [label=above:$5$] at (1.5,2.5) {};
\node [label=above:$4$] at (2.5,2.5) {};
\node [label=above:$3$] at (3.5,2.5) {};
\node [label=above:$2$] at (4.5,2.5) {};
\node [label=above:$1$] at (5.5,2.5) {};

\draw [>-] (1,2) -- (2,2);
\draw [>-] (1,1) -- (2,1);
\draw [>-] (1,0) -- (2,0);
\draw [>-] (1,-1) -- (2,-1);
\draw [>-] (1,-2) -- (2,-2);

\draw [<-] (1.5,2.5) -- (1.5,1.5);
\draw [<-] (1.5,1.5) -- (1.5,.5);
\draw [<-] (1.5,0.5) -- (1.5,-0.5);
\draw [>-] (1.5,-0.5) -- (1.5,-1.5);
\draw [>->] (1.5,-1.5) -- (1.5,-2.5);

\draw [>-] (2,2) -- (3,2);
\draw [>-] (2,1) -- (3,1);
\draw [<-] (2,0) -- (3,0);
\draw [>-] (2,-1) -- (3,-1);
\draw [>-] (2,-2) -- (3,-2);

\draw [<-] (2.5,2.5) -- (2.5,1.5);
\draw [>-] (2.5,1.5) -- (2.5,.5);
\draw [>-] (2.5,0.5) -- (2.5,-0.5);
\draw [<-] (2.5,-0.5) -- (2.5,-1.5);
\draw [>->] (2.5,-1.5) -- (2.5,-2.5);

\draw [<-] (3,2) -- (4,2);
\draw [>-] (3,1) -- (4,1);
\draw [>-] (3,0) -- (4,0);
\draw [<-] (3,-1) -- (4,-1);
\draw [>-] (3,-2) -- (4,-2);

\draw [>-] (3.5,2.5) -- (3.5,1.5);
\draw [>-] (3.5,1.5) -- (3.5,.5);
\draw [>-] (3.5,0.5) -- (3.5,-0.5);
\draw [>-] (3.5,-0.5) -- (3.5,-1.5);
\draw [<->] (3.5,-1.5) -- (3.5,-2.5);

\draw [<-] (4,2) -- (5,2);
\draw [>-] (4,1) -- (5,1);
\draw [>-] (4,0) -- (5,0);
\draw [>-] (4,-1) -- (5,-1);
\draw [<-] (4,-2) -- (5,-2);

\draw [>-] (4.5,2.5) -- (4.5,1.5);
\draw [<-] (4.5,1.5) -- (4.5,.5);
\draw [>-] (4.5,0.5) -- (4.5,-0.5);
\draw [>-] (4.5,-0.5) -- (4.5,-1.5);
\draw [>->] (4.5,-1.5) -- (4.5,-2.5);

\draw [>-] (5,2) -- (6,2);
\draw [<-] (5,1) -- (6,1);
\draw [>-<] (5,0) -- (6,0);
\draw [>-] (5,-1) -- (6,-1);
\draw [<-] (5,-2) -- (6,-2);

\draw [<-] (5.5,2.5) -- (5.5,1.5);
\draw [<-] (5.5,1.5) -- (5.5,.5);
\draw [<-] (5.5,0.5) -- (5.5,-0.5);
\draw [>-] (5.5,-0.5) -- (5.5,-1.5);
\draw [>->] (5.5,-1.5) -- (5.5,-2.5);

\draw [>-]
	(6,2) arc (90:0:2);
\draw [*->]
	(8,0) arc (0:-90:2);

\draw [<-]
	(6,1) arc (90:0:1);
\draw [*-<]
	(7,0) arc (0:-90:1);

\end{tikzpicture}
 

  \end{subfigure}
\caption{An element of $\mathfrak{D}^{[5,1]}$ on the left, and $\mathfrak{BC}^{[5,4,1]}$ on the right.}
\label{fig:type.D.configs}

\end{figure}

In the sequel we will often make use of bijections between our ice models above and certain classes of matrices that are akin to (if not precisely) alternating sign matrices (ASMs).  Recall that an ASM is a square matrix whose entries come from $\{-1,0,1\}$, whose nonzero entries alternate along any given row or column, and for which the net sum along any given row or column is $1$. 

To create a matrix associated to an admissible state $A\in \mathfrak{A}^\lambda$, one fills in the entries of an $n \times \lambda_1$ matrix $\widehat{A}$ according to the following dictionary: if the vertex in row $j$, column $k$ of $A$ has inward pointing arrows only along the horizontal axis, then $\hat a_{jk} = 1$; if the vertex has inward pointing arrows only along the vertical axis, then $\hat a_{jk} = -1$; otherwise, $\hat a_{jk} = 0$.  The procedure for associating a matrix to an admissible state of one of the bent models above is similar, though in these cases the associated matrix $\widehat{A}$ has only the left half of its entries determined by the ice configuration, with the remaining values determined by half-turn symmetry.  So, for instance, the matrix associated to a state in $\mathfrak{B}^\lambda$ is $2n \times 2\lambda_1$, and the matrix associated to a state in $\mathfrak{C}^\lambda$ is $2n +1 \times 2\lambda_1 + 1$.  We give two examples in Figure \ref{fig:matrix.bijection}.

\begin{figure}[!ht]
\begin{subfigure}{.4\textwidth}
$$\left(\begin{array}{rrrrrrrr}1&-1&0&0&1&0&0&0\\0&1&0&0&0&0&0&0\\0&0&1&-1&-1&1&0&0\\0&0&0&0&0&0&1&0\\0&0&0&1&0&0&-1&1\end{array}\right)$$
\end{subfigure}
\begin{subfigure}{.4\textwidth}
$$\left(\begin{array}{rrrrrrr}1&0&-1&1&0&0&0\\0&1&0&-1&1&0&0\\0&0&1&-1&0&1&0\\0&0&0&1&-1&0&1\end{array}\right)$$
\end{subfigure}

\caption{The matrices associated to the states of $\mathfrak{B}_*^{[4,2]}$ and $\mathfrak{C}_*^{[3,2]}$ from Figures \ref{fig:type.B.configs} and \ref{fig:type.C.configs}.}\label{fig:matrix.bijection}
\end{figure}

Note that the set of admissible states $\mathfrak{A}^\rho$ is in bijection with the collection of $n \times n$ alternating sign matrices, the set $\mathfrak{B}^\rho$ is in bijection with half-turn symmetric $2n \times 2n$ alternating sign matrices (see, e.g., Figure~7 of \cite{kuperberg}), and the set $\mathfrak{C}^\rho$ is in bijection with the set of half-turn symmetric $2n+1 \times 2n+1$ alternating sign matrices (see Figure 10 in \cite{razumovstroganov}).  We discuss further connections of our various models to alternating sign matrices in Section \ref{okadamatch}.

Finally, our models differ from the type $C$ model of \cite{hk-symplectic} in the identification of pairs of rows along the right-hand boundary. The bends in the models above are nested, linking rows equidistant from the central row(s), while the model of Hamel and King indentifies the right-most edges of consecutive rows. There appears to be a real dichotomy between models of the two types. As we have remarked, nested models are more closely tied to symmetry classes of alternating sign matrices, while the non-nested type $C$ model has physical implications (see \cite{yamamoto-tsuchiya}) and is more closely connected to branching rules for the symplectic group, as manifested in the Gelfand-Tsetlin patterns of Zhelobenko. It would be interesting to develop a theory of non-nested models for other types.

\section{Boltzmann weights and the Yang-Baxter equation}
Each tetravalent vertex in an admissible state is required to have two inward pointing arrows and two outward pointing arrows. Boltzmann weights are assigned to each of the six possible adjacent edge decorations. These weights are allowed to depend on the row in which the vertex sits, and so we call the row numbers the ``spectral indices." For instance, if the inward-pointing arrows for a vertex decoration in the $j$th row from the top are in the north and west positions, we call this weight $a_1^{(j)}$.  The labeling of the rows indicates that the same vertex decoration in the $j$th row from the bottom is denoted $a_1^{(\overline{j})}$.

\begin{figure}[!ht]
\begin{tabular}
{|c|c|c|c|c|c|}
\hline
$\begin{tikzpicture}
\node [label=left:$j$] at (0.5,1) {};
\node [label=right:$ $] at (1.5,1) {};
\node [label=above:$ $] at (1,1.5) {};
\node [label=below:$ $] at (1,0.5) {};
\draw [>->] (0.5,1) -- (1.5,1);
\draw [>->] (1,1.5) -- (1,0.5);
\end{tikzpicture}$
&
$\begin{tikzpicture}
\node [label=left:$j$] at (0.5,1) {};
\node [label=right:$ $] at (1.5,1) {};
\node [label=above:$ $] at (1,1.5) {};
\node [label=below:$ $] at (1,0.5) {};
\draw [<-<] (0.5,1) -- (1.5,1);
\draw [<-<] (1,1.5) -- (1,0.5);
\end{tikzpicture}$
&
$\begin{tikzpicture}
\node [label=left:$j$] at (0.5,1) {};
\node [label=right:$ $] at (1.5,1) {};
\node [label=above:$ $] at (1,1.5) {};
\node [label=below:$ $] at (1,0.5) {};
\draw [>->] (0.5,1) -- (1.5,1);
\draw [<-<] (1,1.5) -- (1,0.5);
\end{tikzpicture}$
&
$\begin{tikzpicture}
\node [label=left:$j$] at (0.5,1) {};
\node [label=right:$ $] at (1.5,1) {};
\node [label=above:$ $] at (1,1.5) {};
\node [label=below:$ $] at (1,0.5) {};
\draw [<-<] (0.5,1) -- (1.5,1);
\draw [>->] (1,1.5) -- (1,0.5);
\end{tikzpicture}$
&
$\begin{tikzpicture}
\node [label=left:$j$] at (0.5,1) {};
\node [label=right:$ $] at (1.5,1) {};
\node [label=above:$ $] at (1,1.5) {};
\node [label=below:$ $] at (1,0.5) {};
\draw [<->] (0.5,1) -- (1.5,1);
\draw [>-<] (1,1.5) -- (1,0.5);
\end{tikzpicture}$
&
$\begin{tikzpicture}
\node [label=left:$j$] at (0.5,1) {};
\node [label=right:$ $] at (1.5,1) {};
\node [label=above:$ $] at (1,1.5) {};
\node [label=below:$ $] at (1,0.5) {};
\draw [>-<] (0.5,1) -- (1.5,1);
\draw [<->] (1,1.5) -- (1,0.5);
\end{tikzpicture}$
\\
\hline
$a_1^{(j)} = NW$&$a_2^{(j)}=SE$&$b_1^{(j)}=SW$&$b_2^{(j)}=NE$&$c_1^{(j)}=NS$&$c_2^{(j)}=EW$\\ \hline
\end{tabular}
\caption{Boltzmann weights for each of the six possible vertex decorations}\label{fig:six.vertex.model}
\end{figure}

Weights are also assigned to the two possible adjacent edge decorations of the bend vertices and of the vertex at the right-angle in row 0 of the $\mathfrak{C}$ model.  Recall that in each case, a valid vertex decoration consists of one inward- and one outward-pointing arrow.  These additional vertex configurations are given in Figure \ref{fig:bonus.vertex.configurations}.
\begin{figure}[!ht]
\begin{tabular}{|c|c|c|c|}
\hline
\begin{tikzpicture}[scale=0.5]
\node [label=left:$j$] at (0,1) {};
\node [label=left:$\overline{j}$] at (0,-1) {};
\node [label=right:$ $] at (1,0) {};

\draw [>-]
	(0,1) arc (90:0:1);
\draw [*->]
	(1,0) arc (0:-90:1);
\end{tikzpicture}
&
\begin{tikzpicture}[scale=0.5]
\node [label=left:$j$] at (0,1) {};
\node [label=left:$\overline{j}$] at (0,-1) {};
\node [label=right:$ $] at (1,0) {};

\draw [<-]
	(0,1) arc (90:0:1);
\draw [*-<]
	(1,0) arc (0:-90:1);
\end{tikzpicture}
&
\begin{tikzpicture}
\node [label=left:$0$] at (0,1) {};
\node [label=below:$ $] at (1,0) {};
\node [label=above:$ $] at (1,1) {};

\draw [>-] (0,1)--(1,1);
\draw [*->] (1,1)--(1,0);
\end{tikzpicture}
&
\begin{tikzpicture}
\node [label=left:$0$] at (0,1) {};
\node [label=below:$ $] at (1,0) {};
\node [label=above:$ $] at (1,1) {};

\draw [<-] (0,1)--(1,1);
\draw [*-<] (1,1)--(1,0);
\end{tikzpicture}
\\\hline
$D^{(j)}$&$U^{(j)}$&$R^{(0)} = R$&$L^{(0)}=L$\\\hline
\end{tabular}
\caption{Boltzmann weights for special vertices.}
\label{fig:bonus.vertex.configurations}
\end{figure}

To each admissible state $A$ in a given family $\mathfrak{M}^\lambda$, define the Boltzmann weight of the state, $\text{wt}(A)$, to be the product of the Boltzmann weights over all vertices in $A$.  The partition function for $\mathfrak{M}^\lambda$ is then 
\begin{equation} \mathcal{Z}(\mathfrak{M}^\lambda) := \sum_{A \in \mathfrak{M}^\lambda} \text{wt}(A). \label{partitiondef} \end{equation}
Our primary goal is to study the partition function for the various families we have constructed above, as a polynomial in the Boltzmann weights. Our main tool is the Yang-Baxter equation as in \cite{baxter} and as applied to ice models in \cite{kuperberg}, whose definition we recall momentarily. First we place an additional restriction on the Boltzmann weights which will aid in the application of the Yang-Baxter equation.

\begin{symassumption}
The symmetry conditions assumed for Boltzmann weights for bent ice diagrams are
$$a_1^{(j)} = a_2^{(\overline{j})}, a_1^{(\overline{j})} = a_2^{(j)}, b_1^{(j)} = b_2^{(\overline{j})}, b_1^{(\overline{j})} = b_2^{(j)}, c_1^{(j)} = c_1^{(\overline{j})}, c_2^{(j)} = c_2^{(\overline{j})}.$$
\end{symassumption}

An additional motivation for these restrictions comes from their connection to symmetry classes of alternating sign matrices. Half-turn symmetric alternating sign matrices may be represented by ice models in one of two ways --- either as a subset of the admissible states in the $\mathfrak{A}^\rho$ model with correlated vertices or as a bent ice model. From the point of view of statistical mechanics, the latter is to be preferred since the alternating sign matrices are in bijection with all admissible states. However if we take the former approach, then any vertex decoration in the top-right quartile of the ice determines the vertex decoration in the corresponding entry of the bottom-left quartile. One determines the vertex decoration in the bottom-left quartile by simply rotating the decoration in the upper-right quartile by $180^\circ$. For instance, if the vertex decoration in row $1$, column $2n$ is $a_1$, then the vertex decoration in row $2n$, column $1$ will be $a_2$.  Hence from this point of view, it is natural insist $a_1^{(1)} = a_2^{(\overline{1})}$.  

Notice that in the $\mathfrak{B}_*,\mathfrak{C}$ and $\mathfrak{BC}$ models, the central row has the property that its vertical mirror image is itself, and so this gives us a degeneracy of the spectral index: $0 = \overline{0}$ in the first two families, and $n = \overline{n}$ in the third family.  The symmetry conditions on the Boltzmann weights therefore give
\begin{align*}
a_1^{(0)} &= a_2^{(\overline{0})} = a_2^{(0)} = a_1^{(\overline{0})}\\
b_1^{(0)} &= b_2^{(\overline{0})} = b_2^{(0)} = b_1^{(\overline{0})}\end{align*}
in the $\mathfrak{B}_*$ and $\mathfrak{C}$ families.  Hence in these families we write $a^{(0)}$ and $b^{(0)}$ for these two quantities.  Similar results follow in the $\mathfrak{BC}$ family, with the parameter $0$ replaced by $n$.

To analyze the partition functions for each model, we will apply the Yang-Baxter equation (YBE) in the spirit of \cite{baxter}. Baxter used slightly different boundary conditions in the six-vertex model which leads to correlations among the six vertices and a simplification of the Boltzmann weights (the so-called ``field-free case''), so we use the more general formulation in \cite{bbf-ice}. To each spectral index $j$, define
\begin{equation} \Delta^{(j)} = a_1^{(j)}a_2^{(j)}+b_1^{(j)}b_2^{(j)} - c_1^{(j)}c_2^{(j)}. \label{deltadefined} \end{equation}

\begin{proposition}[Yang-Baxter Equation {\cite[Th.~3]{bbf-ice}}]\label{prop:ybe}
Let $j$ and $k$ be a pair of spectral indices whose Boltzmann weights satisfy that $\Delta^{(j)} = \Delta^{(k)} = 0$. Then
there is a third set of Boltzmann weights for the six-vertex model given by
$$
\scalebox{0.85}{
\makebox[0pt]{
\begin{tabular}
{|c|c|c|c|c|c|}
\hline
$\begin{tikzpicture}[scale=0.5]
\node [label=left:$j$] at (0,1) {};
\node [label=left:$k$] at (0,0) {};
\node [label=right:$ $] at (2,1) {};
\node [label=right:$ $] at (2,0) {};
\node [label=above:$ $] at (0,1) {};
\node [label=below:$ $] at (0,0) {};

\draw [>->] (0,1) .. controls (1,1) and (1,0) .. (2,0);
\draw [>->] (0,0) .. controls (1,0) and (1,1) .. (2,1);

\end{tikzpicture}$
&
$\begin{tikzpicture}[scale=0.5]
\node [label=left:$j$] at (0,1) {};
\node [label=left:$k$] at (0,0) {};
\node [label=right:$ $] at (2,1) {};
\node [label=right:$ $] at (2,0) {};
\node [label=above:$ $] at (0,1) {};
\node [label=below:$ $] at (0,0) {};

\draw [<-<] (0,1) .. controls (1,1) and (1,0) .. (2,0);
\draw [<-<] (0,0) .. controls (1,0) and (1,1) .. (2,1);

\end{tikzpicture}$
&$\begin{tikzpicture}[scale=0.5]
\node [label=left:$j$] at (0,1) {};
\node [label=left:$k$] at (0,0) {};
\node [label=right:$ $] at (2,1) {};
\node [label=right:$ $] at (2,0) {};
\node [label=above:$ $] at (0,1) {};
\node [label=below:$ $] at (0,0) {};

\draw [<->] (0,1) .. controls (1,1) and (1,0) .. (2,0);
\draw [>-<] (0,0) .. controls (1,0) and (1,1) .. (2,1);

\end{tikzpicture}$
&$\begin{tikzpicture}[scale=0.5]
\node [label=left:$j$] at (0,1) {};
\node [label=left:$k$] at (0,0) {};
\node [label=right:$ $] at (2,1) {};
\node [label=right:$ $] at (2,0) {};
\node [label=above:$ $] at (0,1) {};
\node [label=below:$ $] at (0,0) {};

\draw [>-<] (0,1) .. controls (1,1) and (1,0) .. (2,0);
\draw [<->] (0,0) .. controls (1,0) and (1,1) .. (2,1);

\end{tikzpicture}$
&$\begin{tikzpicture}[scale=0.5]
\node [label=left:$j$] at (0,1) {};
\node [label=left:$k$] at (0,0) {};
\node [label=right:$ $] at (2,1) {};
\node [label=right:$ $] at (2,0) {};
\node [label=above:$ $] at (0,1) {};
\node [label=below:$ $] at (0,0) {};

\draw [>->] (0,1) .. controls (1,1) and (1,0) .. (2,0);
\draw [<-<] (0,0) .. controls (1,0) and (1,1) .. (2,1);

\end{tikzpicture}$
&
$\begin{tikzpicture}[scale=0.5]
\node [label=left:$j$] at (0,1) {};
\node [label=left:$k$] at (0,0) {};
\node [label=right:$ $] at (2,1) {};
\node [label=right:$ $] at (2,0) {};
\node [label=above:$ $] at (0,1) {};
\node [label=below:$ $] at (0,0) {};

\draw [<-<] (0,1) .. controls (1,1) and (1,0) .. (2,0);
\draw [>->] (0,0) .. controls (1,0) and (1,1) .. (2,1);

\end{tikzpicture}$

\\ \hline
$a_1^{(k)}a_2^{(j)}+b_1^{(j)}b_2^{(k)}$ & $a_1^{(j)}a_2^{(k)}+b_1^{(k)}b_2^{(j)}$ & $c_1^{(j)}c_2^{(k)}$ & $c_1^{(k)}c_2^{(j)}$ & $a_1^{(j)}b_2^{(k)}-a_1^{(k)}b_2^{(j)}$ & $a_2^{(j)}b_1^{(k)}-a_2^{(k)}b_1^{(j)}$ \\ 
\hline
\end{tabular}
}}
$$
such that the following identity of partition functions holds:
$$
\mathcal{Z} \left(
 \begin{minipage}{1.75in}
\begin{tikzpicture}
\node [label=left:$j$] at (-0.5,1) {};
\node [label=left:$k$] at (-0.5,0) {};

\node [circle,draw,scale=.6] at (-.25,1) {$\alpha$};
\node [circle,draw,scale=0.55] at (-.25,0) {$\beta$};
\node [circle,draw,scale=.55] at (2,-.75) {$\gamma$};
\node [circle,draw,scale=.55] at (2.75,0) {$\delta$};
\node [circle,draw,scale=.6] at (2.75,1) {$\epsilon$};
\node [circle,draw,scale=.55] at (2,1.75) {$\phi$};

\draw [-] (0,1) .. controls (1,1) and (1,0) .. (2,0);
\draw [-] (0,0) .. controls (1,0) and (1,1) .. (2,1);
\draw [-] (2,0) -- (2.5,0);
\draw [-] (2,1) -- (2.5,1);
\draw [-] (2,1.5) -- (2,-.5);

\end{tikzpicture}
\end{minipage}
\right)
\quad =\quad 
\mathcal{Z} \left(
\begin{minipage}{1.75in}
\begin{tikzpicture}
\node [label=left:$j$] at (-0.5,1) {};
\node [label=left:$k$] at (-0.5,0) {};

\node [circle,draw,scale=.6] at (-.25,1) {$\alpha$};
\node [circle,draw,scale=0.55] at (-.25,0) {$\beta$};
\node [circle,draw,scale=.55] at (.5,-.75) {$\gamma$};
\node [circle,draw,scale=.55] at (2.75,0) {$\delta$};
\node [circle,draw,scale=.6] at (2.75,1) {$\epsilon$};
\node [circle,draw,scale=.55] at (.5,1.75) {$\phi$};

\draw [-] (0.5,1) .. controls (1.5,1) and (1.5,0) .. (2.5,0);
\draw [-] (0.5,0) .. controls (1.5,0) and (1.5,1) .. (2.5,1);
\draw [-] (0,0) -- (0.5,0);
\draw [-] (0,1) -- (0.5,1);
\draw [-] (.5,1.5) -- (.5,-.5);
\end{tikzpicture}
\end{minipage}
\right)
$$
for any choice of boundary arrow decorations $\alpha,\beta,\gamma,\delta,\epsilon,\phi$.
\end{proposition}

\begin{remark*}
The condition $\Delta^{(j)}=0$ is known as the free-fermion condition on the Boltzmann weights.  Without loss of generality, we may thus assume $c_2^{(j)}=1$ for all spectral indices $j$. Note that in the case when the spectral index is $0$ in the $\mathfrak{B}_*$ and $\mathfrak{C}$ families, the fact that $0 = \overline{0}$ means that the symmetry conditions on the Boltzmann weights give $$\left(a^{(0)}\right)^2 + \left(b^{(0)}\right)^2 = c_1^{(0)}.$$  A similar relation involving $a^{(n)}, b^{(n)}$ and $c_1^{(n)}$ follows in the $\mathfrak{BC}$ family because $n = \bar n$ in these models.
\end{remark*}

Since our diagrams are not merely rectangular (having bends on the right side, symmetric row pairs, etc.), we will need additional local identities of partition functions, similar to the YBE.  These variations are given in the next few lemmas. Identities of this type appeared, for example, in Section~4 of \cite{kuperberg}, and we mimic his zoological nomenclature. At this point, the reader may 
want to skip to Section~\ref{usingYBE} of this paper to see how the Yang-Baxter equation and these additional rules are applied to study the partition function.

\begin{lemma}[YBE along the bend]\label{le:bend.ybe}
Assume that $U^{(j)},D^{(j)}, c_1^{(j)},U^{(k)},D^{(k)}, c_1^{(k)} \neq 0$.  Then a necessary and sufficient condition for 
$$
\text{wt}
\left(
\begin{minipage}{1.4in}
\begin{centering}
\begin{tikzpicture}
\node [label=left:$j$] at (-0.5,1) {};
\node [label=left:$k$] at (-0.5,0.5) {};
\node [label=left:$\overline{j}$] at (-0.5,0) {};
\node [label=left:$\overline{k}$] at (-0.5,-.5) {};
\node [circle,draw,scale=.6] at (-.25,1) {$\alpha$};
\node [circle,draw,scale=0.55] at (-.25,0.5) {$\beta$};
\node [circle,draw,scale=0.55] at (-.25,0) {$\gamma$};
\node [circle,draw,scale=0.6] at (-.25,-0.5) {$\delta$};

\node [draw,circle,scale=0.6] at (1.25,1) {$\delta$};
\node [draw,circle,scale=0.55] at (1.25,0.5) {$\gamma$};

\draw [-] (0,1) .. controls (0.5,1) and (.5,.5) .. (1,.5);
\draw [-] (0,.5) .. controls (0.5,.5) and (.5,1) .. (1,1);
\draw [-] (0,0) -- (1.5,0);
\draw [-] (0,-.5) -- (1.5,-.5);

\draw [-]
	(1.5,1) arc (90:0:.75);
\draw [*-]
	(2.25,.25) arc (0:-90:.75);
\draw [-]
	(1.5,.5) arc (90:0:.25);
\draw [*-]
	(1.75,.25) arc (0:-90:.25);
\end{tikzpicture}
\end{centering}
\end{minipage}
\right)
\quad = \quad
\text{wt}
\left( 
\begin{minipage}{1.4in}
\begin{centering}
\begin{tikzpicture}
\node [label=left:$j$] at (-0.5,1) {};
\node [label=left:$k$] at (-0.5,0.5) {};
\node [label=left:$\overline{j}$] at (-0.5,0) {};
\node [label=left:$\overline{k}$] at (-0.5,-.5) {};
\node [circle,draw,scale=.6] at (-.25,1) {$\alpha$};
\node [circle,draw,scale=0.55] at (-.25,0.5) {$\beta$};
\node [circle,draw,scale=0.55] at (-.25,0) {$\gamma$};
\node [circle,draw,scale=0.60] at (-.25,-0.5) {$\delta$};

\node [draw,circle,scale=0.55] at (1.25,0) {$\beta$};
\node [draw,circle,scale=.6] at (1.25,-0.5) {$\alpha$};

\draw [-] (0,0) .. controls (0.5,0) and (.5,-.5) .. (1,-.5);
\draw [-] (0,-.5) .. controls (0.5,-.5) and (.5,0) .. (1,0);
\draw [-] (0,1) -- (1.5,1);
\draw [-] (0,.5) -- (1.5,.5);

\draw [-]
	(1.5,1) arc (90:0:.75);
\draw [*-]
	(2.25,.25) arc (0:-90:.75);
\draw [-]
	(1.5,.5) arc (90:0:.25);
\draw [*-]
	(1.75,.25) arc (0:-90:.25);
\end{tikzpicture}
\end{centering}
\end{minipage}
\right)
$$
regardless of the choice of inward and outward pointing arrows $\alpha,\beta,\gamma,\delta$ is that $D^{(j)}/U^{(j)} = D^{(k)}/U^{(k)}$.  If $c_1^{(j)} = c_1^{(k)}=0$, then this equation is always valid.
\end{lemma}

We have used the notation ``wt'' in the above lemma to stress that both sides of the identity are the weight of a single admissible ice configuration (rather than a sum of several admissible configurations, for which we reserve the notation $\mathcal{Z}$). This convention will be repeated in the lemmas that follow. 

\begin{proof}
One simply examines all possible assignments of boundary arrows and applies the Boltzmann symmetry conditions; we have listed the possible configurations and their corresponding weights in Figure \ref{fig:turn}.  
\begin{figure}[!ht]
\scalebox{0.95}{
\makebox[0pt]{
\begin{tabular}{|ccc|ccc|}\hline
%
%

\begin{tikzpicture}[scale=0.75]

\node [label=left:$j$] at (1,2) {};
\node [label=left:$k$] at (1,1) {};
\node [label=left:$\overline{j}$] at (1,0) {};
\node [label=left:$\overline{k}$] at (1,-1) {};

\draw [>->] (1,2) .. controls (1.5,2) and (1.5,1) .. (2,1);
\draw [>->] (1,1) .. controls (1.5,1) and (1.5,2) .. (2,2);
\draw [<-] (1,0) -- (2,0);
\draw [<-] (1,-1) -- (2,-1);

\draw [-]
	(2,2) arc (90:0:1.5);
\draw [*-]
	(3.5,0.5) arc (0:-90:1.5);

\draw [-]
	(2,1) arc (90:0:.5);
\draw [*-]
	(2.5,0.5) arc (0:-90:.5);
\end{tikzpicture}
&

&
\begin{tikzpicture}[scale=0.75]

\node [label=left:$j$] at (1,2) {};
\node [label=left:$k$] at (1,1) {};
\node [label=left:$\overline{j}$] at (1,0) {};
\node [label=left:$\overline{k}$] at (1,-1) {};

\draw [>-] (1,2) -- (2,2);
\draw [>-] (1,1) -- (2,1);
\draw [<-<] (1,0) .. controls (1.5,0) and (1.5,-1) .. (2,-1);
\draw [<-<] (1,-1) .. controls (1.5,-1) and (1.5,0) .. (2,0);

\draw [-]
	(2,2) arc (90:0:1.5);
\draw [*-]
	(3.5,0.5) arc (0:-90:1.5);

\draw [-]
	(2,1) arc (90:0:.5);
\draw [*-]
	(2.5,0.5) arc (0:-90:.5);
\end{tikzpicture}
&
\begin{tikzpicture}[scale=0.75]
\node [label=left:$j$] at (1,2) {};
\node [label=left:$k$] at (1,1) {};
\node [label=left:$\overline{j}$] at (1,0) {};
\node [label=left:$\overline{k}$] at (1,-1) {};

\draw [>-<] (1,2) .. controls (1.5,2) and (1.5,1) .. (2,1);
\draw [<->] (1,1) .. controls (1.5,1) and (1.5,2) .. (2,2);
\draw [>-] (1,0) -- (2,0);
\draw [<-] (1,-1) -- (2,-1);

\draw [-]
	(2,2) arc (90:0:1.5);
\draw [*-]
	(3.5,0.5) arc (0:-90:1.5);

\draw [-]
	(2,1) arc (90:0:.5);
\draw [*-]
	(2.5,0.5) arc (0:-90:.5);
\end{tikzpicture}
&

&

\begin{tikzpicture}[scale=0.75]

\node [label=left:$j$] at (1,2) {};
\node [label=left:$k$] at (1,1) {};
\node [label=left:$\overline{j}$] at (1,0) {};
\node [label=left:$\overline{k}$] at (1,-1) {};

\draw [>-] (1,2) -- (2,2);
\draw [<-] (1,1) -- (2,1);
\draw [>-<] (1,0) .. controls (1.5,0) and (1.5,-1) .. (2,-1);
\draw [<->] (1,-1) .. controls (1.5,-1) and (1.5,0) .. (2,0);

\draw [-]
	(2,2) arc (90:0:1.5);
\draw [*-]
	(3.5,0.5) arc (0:-90:1.5);

\draw [-]
	(2,1) arc (90:0:.5);
\draw [*-]
	(2.5,0.5) arc (0:-90:.5);
\end{tikzpicture}

\\[10pt]
$\left(a_1^{(k)}a_2^{(j)}+b_1^{(j)}b_2^{(k)}\right)$
&
$=$
&

$\left(a_1^{(\overline{j})}a_2^{(\overline{k})}+b_1^{(\overline{k})}b_2^{(\overline{j})}\right)$
&
$D^{(j)}U^{(k)}\left(c_1^{(k)}c_2^{(j)}\right)$
&
$=$
&

$U^{(j)}D^{(k)}\left(c_1^{(\overline{k})}c_2^{(\overline{j})}\right)$

\\[10pt] \hline

%
%

\begin{tikzpicture}[scale=0.75]

\node [label=left:$j$] at (1,2) {};
\node [label=left:$k$] at (1,1) {};
\node [label=left:$\overline{j}$] at (1,0) {};
\node [label=left:$\overline{k}$] at (1,-1) {};

\draw [>->] (1,2) .. controls (1.5,2) and (1.5,1) .. (2,1);
\draw [<-<] (1,1) .. controls (1.5,1) and (1.5,2) .. (2,2);
\draw [<-] (1,0) -- (2,0);
\draw [>-] (1,-1) -- (2,-1);

\draw [-]
	(2,2) arc (90:0:1.5);
\draw [*-]
	(3.5,0.5) arc (0:-90:1.5);

\draw [-]
	(2,1) arc (90:0:.5);
\draw [*-]
	(2.5,0.5) arc (0:-90:.5);
\end{tikzpicture}
&

&

\begin{tikzpicture}[scale=0.75]

\node [label=left:$j$] at (1,2) {};
\node [label=left:$k$] at (1,1) {};
\node [label=left:$\overline{j}$] at (1,0) {};
\node [label=left:$\overline{k}$] at (1,-1) {};

\draw [>-] (1,2) -- (2,2);
\draw [<-] (1,1) -- (2,1);
\draw [<-<] (1,0) .. controls (1.5,0) and (1.5,-1) .. (2,-1);
\draw [>->] (1,-1) .. controls (1.5,-1) and (1.5,0) .. (2,0);

\draw [-]
	(2,2) arc (90:0:1.5);
\draw [*-]
	(3.5,0.5) arc (0:-90:1.5);

\draw [-]
	(2,1) arc (90:0:.5);
\draw [*-]
	(2.5,0.5) arc (0:-90:.5);
\end{tikzpicture}
&

\begin{tikzpicture}[scale=0.75]
\node [label=left:$j$] at (1,2) {};
\node [label=left:$k$] at (1,1) {};
\node [label=left:$\overline{j}$] at (1,0) {};
\node [label=left:$\overline{k}$] at (1,-1) {};

\draw [<-<] (1,2) .. controls (1.5,2) and (1.5,1) .. (2,1);
\draw [>->] (1,1) .. controls (1.5,1) and (1.5,2) .. (2,2);
\draw [>-] (1,0) -- (2,0);
\draw [<-] (1,-1) -- (2,-1);

\draw [-]
	(2,2) arc (90:0:1.5);
\draw [*-]
	(3.5,0.5) arc (0:-90:1.5);

\draw [-]
	(2,1) arc (90:0:.5);
\draw [*-]
	(2.5,0.5) arc (0:-90:.5);
\end{tikzpicture}
&

&

\begin{tikzpicture}[scale=0.75]

\node [label=left:$j$] at (1,2) {};
\node [label=left:$k$] at (1,1) {};
\node [label=left:$\overline{j}$] at (1,0) {};
\node [label=left:$\overline{k}$] at (1,-1) {};

\draw [<-] (1,2) -- (2,2);
\draw [>-] (1,1) -- (2,1);
\draw [>->] (1,0) .. controls (1.5,0) and (1.5,-1) .. (2,-1);
\draw [<-<] (1,-1) .. controls (1.5,-1) and (1.5,0) .. (2,0);

\draw [-]
	(2,2) arc (90:0:1.5);
\draw [*-]
	(3.5,0.5) arc (0:-90:1.5);

\draw [-]
	(2,1) arc (90:0:.5);
\draw [*-]
	(2.5,0.5) arc (0:-90:.5);
\end{tikzpicture}

\\[10pt]
$\left(a_1^{(j)}b_2^{(k)}-a_1^{(k)}b_2^{(j)}\right)$
&
$=$
&

$\left(a_2^{(\overline{j})}b_1^{(\overline{k})}-a_2^{(\overline{k})}b_1^{(\overline{j})}\right)$
&
$\left(a_2^{(j)}b_1^{(k)}-a_2^{(k)}b_1^{(j)}\right)$
&
$=$
&
$\left(a_1^{(\overline{j})}b_2^{(\overline{k})}-a_1^{(\overline{k})}b_2^{(\overline{j})}\right)$

\\[10pt] \hline
%
%

\begin{tikzpicture}[scale=0.75]

\node [label=left:$j$] at (1,2) {};
\node [label=left:$k$] at (1,1) {};
\node [label=left:$\overline{j}$] at (1,0) {};
\node [label=left:$\overline{k}$] at (1,-1) {};

\draw [<->] (1,2) .. controls (1.5,2) and (1.5,1) .. (2,1);
\draw [>-<] (1,1) .. controls (1.5,1) and (1.5,2) .. (2,2);
\draw [<-] (1,0) -- (2,0);
\draw [>-] (1,-1) -- (2,-1);

\draw [-]
	(2,2) arc (90:0:1.5);
\draw [*-]
	(3.5,0.5) arc (0:-90:1.5);

\draw [-]
	(2,1) arc (90:0:.5);
\draw [*-]
	(2.5,0.5) arc (0:-90:.5);
\end{tikzpicture}
&

&

\begin{tikzpicture}[scale=0.75]

\node [label=left:$j$] at (1,2) {};
\node [label=left:$k$] at (1,1) {};
\node [label=left:$\overline{j}$] at (1,0) {};
\node [label=left:$\overline{k}$] at (1,-1) {};

\draw [<-] (1,2) -- (2,2);
\draw [>-] (1,1) -- (2,1);
\draw [<->] (1,0) .. controls (1.5,0) and (1.5,-1) .. (2,-1);
\draw [>-<] (1,-1) .. controls (1.5,-1) and (1.5,0) .. (2,0);

\draw [-]
	(2,2) arc (90:0:1.5);
\draw [*-]
	(3.5,0.5) arc (0:-90:1.5);

\draw [-]
	(2,1) arc (90:0:.5);
\draw [*-]
	(2.5,0.5) arc (0:-90:.5);
\end{tikzpicture}
&
\begin{tikzpicture}[scale=0.75]
\node [label=left:$j$] at (1,2) {};
\node [label=left:$k$] at (1,1) {};
\node [label=left:$\overline{j}$] at (1,0) {};
\node [label=left:$\overline{k}$] at (1,-1) {};

\draw [<-<] (1,2) .. controls (1.5,2) and (1.5,1) .. (2,1);
\draw [<-<] (1,1) .. controls (1.5,1) and (1.5,2) .. (2,2);
\draw [>-] (1,0) -- (2,0);
\draw [>-] (1,-1) -- (2,-1);

\draw [-]
	(2,2) arc (90:0:1.5);
\draw [*-]
	(3.5,0.5) arc (0:-90:1.5);

\draw [-]
	(2,1) arc (90:0:.5);
\draw [*-]
	(2.5,0.5) arc (0:-90:.5);
\end{tikzpicture}
&

&

\begin{tikzpicture}[scale=0.75]

\node [label=left:$j$] at (1,2) {};
\node [label=left:$k$] at (1,1) {};
\node [label=left:$\overline{j}$] at (1,0) {};
\node [label=left:$\overline{k}$] at (1,-1) {};

\draw [<-] (1,2) -- (2,2);
\draw [<-] (1,1) -- (2,1);
\draw [>->] (1,0) .. controls (1.5,0) and (1.5,-1) .. (2,-1);
\draw [>->] (1,-1) .. controls (1.5,-1) and (1.5,0) .. (2,0);

\draw [-]
	(2,2) arc (90:0:1.5);
\draw [*-]
	(3.5,0.5) arc (0:-90:1.5);

\draw [-]
	(2,1) arc (90:0:.5);
\draw [*-]
	(2.5,0.5) arc (0:-90:.5);
\end{tikzpicture}

\\[10pt]
$U^{(j)}D^{(k)}\left(c_1^{(j)}c_2^{(k)}\right)$
&
$=$
&

$D^{(j)}U^{(k)}\left(c_1^{(\overline{j})}c_2^{(\overline{k})}\right)$
&
$\left(a_1^{(j)}a_2^{(k)} + b_1^{(k)}b_2^{(j)}\right) $
&
$=$
&

$\left(a_1^{(\overline{k})}a_2^{(\overline{j})} + b_1^{(\overline{j})}b_2^{(\overline{k})}\right)$

\\ \hline
\end{tabular} }}

\caption{Verifying that weights do not change when passed through the bend; all possible values for $\alpha,\beta,\gamma,\delta$ are accounted for.}
\label{fig:turn}
\end{figure}
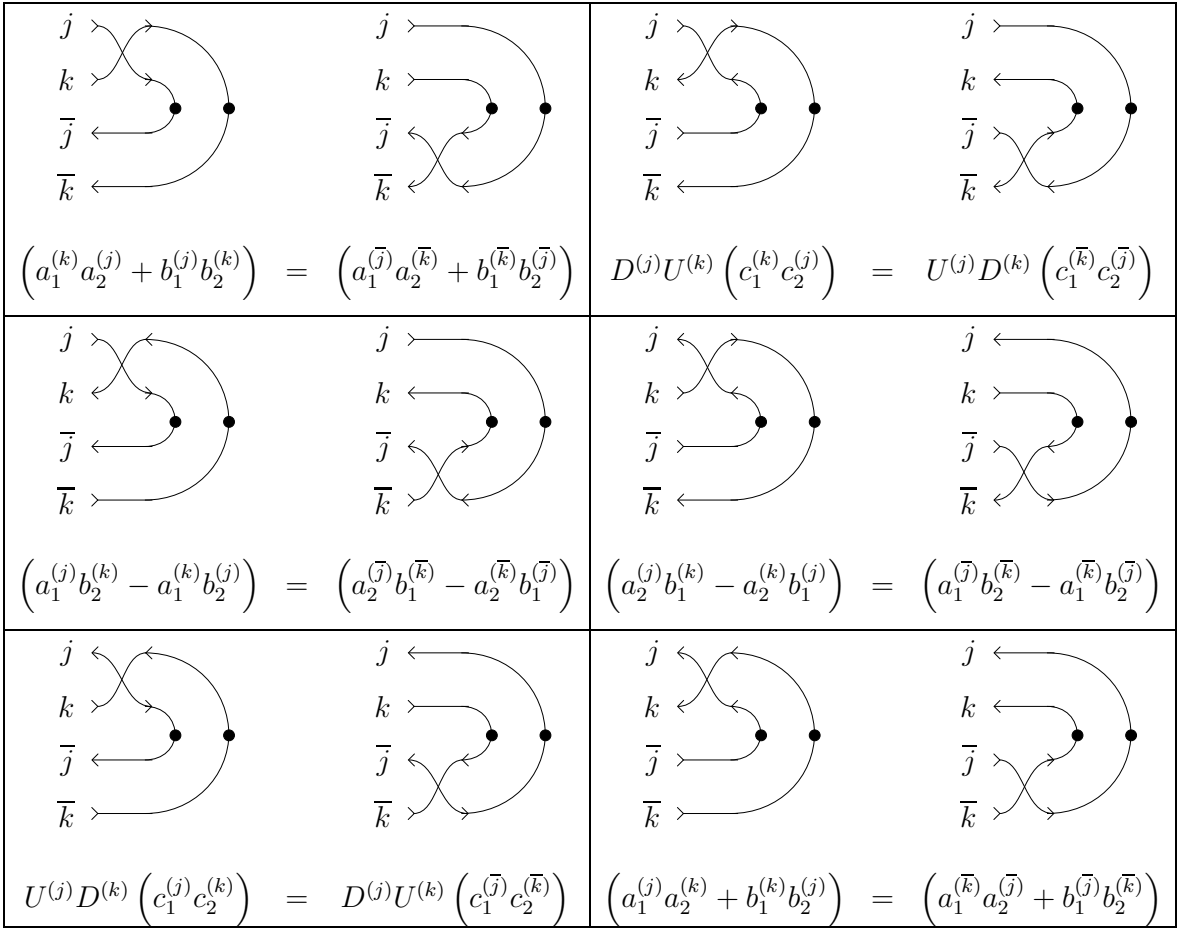
\end{proof}

\begin{lemma}[Caduceus relation]\label{le:caduceus}
Let $\alpha,\beta,\gamma,\delta,\epsilon,\phi,\kappa,\lambda$ be given orientations, and let $\star = 0$ in the $\mathfrak{B}_*$ and $\mathfrak{C}$ families and $\star = n$ in the $\mathfrak{BC}$ family.  Then
$$
\mathcal{Z}\left(
\begin{minipage}{2.3in}
\begin{tikzpicture}
\node [label=left:$\overline{j}$] at (-0.5,1) {};
\node [label=left:$\star$] at (-0.5,0) {};
\node [label=left:$j$] at (-0.5,-1) {};

\node [circle,draw,scale=.6] at (-.25,1) {$\alpha$};
\node [circle,draw,scale=0.55] at (-.25,0) {$\beta$};
\node [circle,draw,scale=.55] at (-.25,-1) {$\gamma$};
\node [circle,draw,scale=.55] at (3.5,-1.75) {$\delta$};
\node [circle,draw,scale=.6] at (4.25,-1) {$\epsilon$};
\node [circle,draw,scale=.55] at (4.25,0) {$\phi$};
\node [circle,draw,scale=.55] at (4.25,1) {$\kappa$};
\node [circle,draw,scale=.55] at (3.5,1.75) {$\lambda$};

\draw [-] (0,1) .. controls (.5,1) and (.5,0) .. (1,0);
\draw [-] (0,0) .. controls (.5,0) and (.5,1) .. (1,1);
\draw [-] (0,-1) -- (1,-1);
\draw [-] (1,-1) .. controls (1.5,-1) and (1.5,0) .. (2,0);
\draw [-] (1,0) .. controls (1.5,0) and (1.5,-1) .. (2,-1);
\draw [-] (1,1) -- (2,1);
\draw [-] (2,1) .. controls (2.5,1) and (2.5,0) .. (3,0);
\draw [-] (2,0) .. controls (2.5,0) and (2.5,1) .. (3,1);
\draw [-] (2,-1) -- (4,-1);
\draw [-] (3,0) -- (4,0);
\draw [-] (3,1) -- (4,1);
\draw [-] (3.5,1.5) -- (3.5,-1.5);
\end{tikzpicture}
\end{minipage}
\right) \quad 
= \quad 
\mathcal{Z}
\left(
\begin{minipage}{2.3in}
\begin{tikzpicture}
\node [label=left:$\overline{j}$] at (-0.5,1) {};
\node [label=left:$\star$] at (-0.5,0) {};
\node [label=left:$j$] at (-0.5,-1) {};

\node [circle,draw,scale=.6] at (-.25,1) {$\alpha$};
\node [circle,draw,scale=0.55] at (-.25,0) {$\beta$};
\node [circle,draw,scale=.55] at (-.25,-1) {$\gamma$};
\node [circle,draw,scale=.55] at (.5,-1.75) {$\delta$};
\node [circle,draw,scale=.6] at (4.25,-1) {$\epsilon$};
\node [circle,draw,scale=.55] at (4.25,0) {$\phi$};
\node [circle,draw,scale=.55] at (4.25,1) {$\kappa$};
\node [circle,draw,scale=.55] at (.5,1.75) {$\lambda$};

\draw [-] (1,1) .. controls (1.5,1) and (1.5,0) .. (2,0);
\draw [-] (1,0) .. controls (1.5,0) and (1.5,1) .. (2,1);
\draw [-] (0,-1) -- (2,-1);
\draw [-] (2,-1) .. controls (2.5,-1) and (2.5,0) .. (3,0);
\draw [-] (2,0) .. controls (2.5,0) and (2.5,-1) .. (3,-1);
\draw [-] (2,1) -- (3,1);
\draw [-] (3,1) .. controls (3.5,1) and (3.5,0) .. (4,0);
\draw [-] (3,0) .. controls (3.5,0) and (3.5,1) .. (4,1);
\draw [-] (3,-1) -- (4,-1);
\draw [-] (0,0) -- (1,0);
\draw [-] (0,1) -- (1,1);
\draw [-] (.5,1.5) -- (.5,-1.5);
\end{tikzpicture}
\end{minipage}
\right)
$$
\end{lemma}

\begin{proof}
This follows from three applications of the Yang-Baxter equation.
\end{proof}

Before proceeding to the next result, we introduce an additional symmetry assumption on bend vertices. Let $D^{(\overline{j})}$ and $U^{(\overline{j})}$ denote the bend vertices with
down and up arrows, respectively, where the spectral indices are now reversed: $\overline{j}$ on top and $j$ on bottom as shown in Figure~\ref{fig:more.bonus}.

\begin{symassumption}
The symmetry conditions assumed for Boltzmann weights on the bends are 
$$ U^{(j)} = U^{(\overline{j})} \quad \text{and} \quad D^{(j)} = D^{(\overline{j})}. $$
\end{symassumption}

\begin{figure}[!ht]
\begin{tabular}{|c|c|c|c|}
\hline
\begin{tikzpicture}[scale=0.5]
\node [label=left:$j$] at (0,1) {};
\node [label=left:$\overline{j}$] at (0,-1) {};
\node [label=right:$ $] at (1,0) {};

\draw [>-]
	(0,1) arc (90:0:1);
\draw [*->]
	(1,0) arc (0:-90:1);
\end{tikzpicture}
&
\begin{tikzpicture}[scale=0.5]
\node [label=left:$j$] at (0,1) {};
\node [label=left:$\overline{j}$] at (0,-1) {};
\node [label=right:$ $] at (1,0) {};

\draw [<-]
	(0,1) arc (90:0:1);
\draw [*-<]
	(1,0) arc (0:-90:1);
\end{tikzpicture}
&
\begin{tikzpicture}[scale=0.5]
\node [label=left:$\overline{j}$] at (0,1) {};
\node [label=left:$j$] at (0,-1) {};
\node [label=right:$ $] at (1,0) {};

\draw [>-]
	(0,1) arc (90:0:1);
\draw [*->]
	(1,0) arc (0:-90:1);
\end{tikzpicture}
&
\begin{tikzpicture}[scale=0.5]
\node [label=left:$\overline{j}$] at (0,1) {};
\node [label=left:$j$] at (0,-1) {};
\node [label=right:$ $] at (1,0) {};

\draw [<-]
	(0,1) arc (90:0:1);
\draw [*-<]
	(1,0) arc (0:-90:1);
\end{tikzpicture}
\\\hline
$D^{(j)}$&$U^{(j)}$&$D^{(\overline{j})}$&$U^{(\overline{j})}$ \\
\hline
\end{tabular}
\caption{Notation for bend weights with spectral parameters $j$ and $\overline{j}$.}\label{fig:more.bonus}
\end{figure}
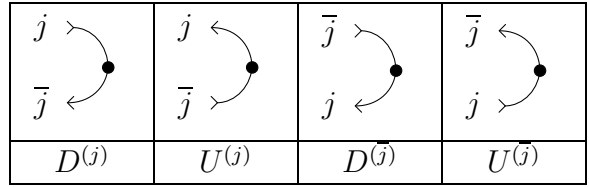

\begin{lemma}[Fish relation, type $\mathfrak{B}$]\label{le:fish.B}
Assume $U^{(j)}, D^{(j)} \neq 0$.  A necessary and sufficient condition for the ratio of 
$$
\mathcal{Z}
\left(
\begin{minipage}{1.1in}
\begin{tikzpicture}
\node [label=left:$\overline{j}$] at (-0.5,1) {};
\node [label=left:$j$] at (-0.5,0) {};
\node [circle,draw,scale=.6] at (-.25,1) {$\alpha$};
\node [circle,draw,scale=0.55] at (-.25,0) {$\beta$};


\draw [-] (0,0) .. controls (0.5,0) and (.5,1) .. (1,1);
\draw [-] (0,1) .. controls (0.5,1) and (.5,0) .. (1,0);

\draw [-]
	(1,1) arc (90:0:.5);
\draw [*-]
	(1.5,.5) arc (0:-90:.5);
\end{tikzpicture}
\end{minipage}
\right)
\quad \mbox{and} \quad
\text{wt}
\left(
\begin{minipage}{.7in}
\begin{tikzpicture}
\node [label=left:$\overline{j}$] at (-.5,1) {};
\node [label=left:$j$] at (-.5,0) {};
\node [circle,draw,scale=.6] at (-.25,1) {$\alpha$};
\node [circle,draw,scale=0.55] at (-.25,0) {$\beta$};

\draw [-]
	(0,1) arc (90:0:.5);
\draw [*-]
	(.5,0.5) arc (0:-90:.5);
\end{tikzpicture}
\end{minipage}
\right)
$$
to be a constant independent of  $\alpha$ and $\beta$ is that $\left(D^{(j)}/U^{(j)}\right)^2=-1$.  If we have $D^{(j)}/U^{(j)}=i$, then the ratio is $$\left(a_1^{(j)}-i~b_2^{(j)}\right)\left(a_2^{(j)}+i~b_1^{(j)}\right).$$
\end{lemma}

\begin{proof}
There are two choices for the pair $\{\alpha,\beta\}$, and setting the corresponding ratios equal yields
$$\frac{D^{(j)} c_1^{(j)}c_2^{(\overline{j})}+U^{(j)}\left(a_1^{(\overline{j})}b_2^{(j)}-a_1^{(j)}b_2^{(\overline{j})}\right)}{D^{(j)}} = \frac{U^{(j)}c_1^{(j)}c_2^{(\overline{j})}+D^{(j)}\left(a_2^{(\overline{j})}b_1^{(j)}-a_2^{(j)}b_1^{(\overline{j})}\right)}{U^{(j)}}$$  Applying the symmetry conditions and simplifying leaves
$$\left(\frac{D^{(j)}}{U^{(j)}}\right)^2 = \frac{a_1^{(\overline{j})}b_2^{(j)}-a_1^{(j)}b_2^{(\overline{j})}}{a_2^{(\overline{j})}b_1^{(j)}-a_2^{(j)}b_1^{(\overline{j})}} = \frac{a_2^{(j)}b_2^{(j)}-a_1^{(j)}b_1^{(j)}}{a_1^{(j)}b_1^{(j)}-a_2^{(j)}b_2^{(j)}} = -1.$$

Now to compute the ratio in question, we  calculate
\begin{align*}
c_1^{(j)}c_2^{(\overline{j})} -i\left(a_1^{(\overline{j})}b_2^{(j)}-a_1^{(j)}b_2^{(\overline{j})}\right) 
& = c_1^{(j)}c_2^{(j)}-i\left(a_2^{(j)}b_2^{(j)}-a_1^{(j)}b_1^{(j)}\right)\\
&= a_1^{(j)}a_2^{(j)}+b_1^{(j)}b_2^{(j)}-i\left(a_2^{(j)}b_2^{(j)}-a_1^{(j)}b_1^{(j)}\right)\\
&= a_1^{(j)}\left(a_2^{(j)}+i~b_1^{(j)}\right)+b_2^{(j)}\left(b_1^{(j)}-i~a_2^{(j)}\right)\\
&=\left(a_1^{(j)}-i~b_2^{(j)}\right)\left(a_2^{(j)}+i~b_1^{(j)}\right).
\end{align*}
\end{proof}

\begin{lemma}[Fish relation,types $\mathfrak{C}_*$ and $\mathfrak{D}$ for $1 \not\in\lambda$]\label{le:fish.C*}
Assume $U^{(j)}, D^{(j)} \neq 0$.  A necessary and sufficient condition for the ratio between
$$
\mathcal{Z}
\left(
\begin{minipage}{1.5in}
\begin{tikzpicture}
\node [label=left:$\overline{j}$] at (-0.5,1) {};
\node [label=left:$j$] at (-0.5,0) {};
\node [circle,draw,scale=.6] at (-.25,1) {$\alpha$};
\node [circle,draw,scale=0.55] at (-.25,0) {$\beta$};
\node [circle,draw,scale=.55] at (1.5,-.75) {$\gamma$};

\draw [>-] (1.5,.5) -- (1.5,-.5);
\draw [-] (0,0) .. controls (0.5,0) and (.5,1) .. (1,1);
\draw [-] (0,1) .. controls (0.5,1) and (.5,0) .. (1,0);
\draw [-] (1,1) -- (2,1);
\draw [-] (1,0) -- (2,0);

\draw [-]
	(2,1) arc (90:0:.5);
\draw [*-]
	(2.5,.5) arc (0:-90:.5);
\end{tikzpicture}
\end{minipage}
\right)
\quad \mbox{and} \quad
\text{wt}\left(
\begin{minipage}{1.1in}
\begin{tikzpicture}
\node [label=left:$\overline{j}$] at (-.5,1) {};
\node [label=left:$j$] at (-.5,0) {};
\node [circle,draw,scale=.6] at (-.25,1) {$\alpha$};
\node [circle,draw,scale=0.55] at (-.25,0) {$\beta$};
\node [circle,draw,scale=.55] at (.5,-.75) {$\gamma$};

\draw [-] (0,1) -- (1,1);
\draw [-] (0,0) -- (1,0);
\draw [>-] (0.5,.5) -- (.5,-.5);
\draw [-]
	(1,1) arc (90:0:.5);
\draw [*-]
	(1.5,0.5) arc (0:-90:.5);
\end{tikzpicture}
\end{minipage}
\right)
$$ to be a constant independent of $\alpha,\beta$ and $\gamma$ is that $D^{(j)} = U^{(j)}$.  Under this assumption, this ratio is equal to $$\left(a_2^{(j)}\right)^2 + \left(b_1^{(j)}\right)^2.$$
\end{lemma}

\begin{proof}
To prove the result, we simply list the possibilities and evaluate the ratio in each case.  The possible fillings are given in Figure \ref{fig:Calt.and.nonzero.containing.D.fish}.  The values of the ratio in each row are:
$$
\frac{U^{(j)}}{D^{(j)}} \left(a_2^{(j)}\right)^2 + \left(b_1^{(j)}\right)^2 + \left(1-\frac{U^{(j)}}{D^{(j)}}\right)\frac{a_1^{(j)}a_2^{(j)}b_1^{(j)}}{b_2^{(j)}}$$
$$ \left(a_2^{(j)}\right)^2 + \frac{D^{(j)}}{U^{(j)}}\left(b_1^{(j)}\right)^2 + \left(1-\frac{D^{(j)}}{U^{(j)}}\right)\frac{a_2^{(j)}b_1^{(j)}b_2^{(j)}}{a_1^{(j)}}$$
$$\left(a_2^{(j)}\right)^2 + \left(b_1^{(j)}\right)^2.
$$
One simply compares coefficients to get the necessary and sufficient condition; the resulting ratio can be read off the third expression.

\begin{figure}[!ht]
\begin{tabular}{|ccc|c|} \hline
%
%

\begin{tikzpicture}

\node [label=left:$\overline{j}$] at (1,2) {};
\node [label=left:$j$] at (1,1) {};

\draw [>->] (1,2) .. controls (1.5,2) and (1.5,1) .. (2,1);
\draw [<-<] (1,1) .. controls (1.5,1) and (1.5,2) .. (2,2);

\draw [-<] (2,2) -- (3,2);
\draw [->] (2,1) -- (3,1);

\draw [>->] (2.5,1.5) -- (2.5,0.5);

\draw [-]
	(3,2) arc (90:0:.5);
\draw [*-]
	(3.5,1.5) arc (0:-90:.5);
\end{tikzpicture}
&

&
\begin{tikzpicture}

\node [label=left:$\overline{j}$] at (1,2) {};
\node [label=left:$j$] at (1,1) {};

\draw [>-<] (1,2) .. controls (1.5,2) and (1.5,1) .. (2,1);
\draw [<->] (1,1) .. controls (1.5,1) and (1.5,2) .. (2,2);

\draw [->] (2,2) -- (3,2);
\draw [-<] (2,1) -- (3,1);

\draw [>->] (2.5,1.5) -- (2.5,0.5);

\draw [-]
	(3,2) arc (90:0:.5);
\draw [*-]
	(3.5,1.5) arc (0:-90:.5);
\end{tikzpicture}
&

\begin{tikzpicture}

\node [label=left:$\overline{j}$] at (1,2) {};
\node [label=left:$j$] at (1,1) {};

\draw [>->] (1,2) -- (2,2);
\draw [<-<] (1,1) -- (2,1);

\draw [>->] (1.5,1.5) -- (1.5,.5);

\draw [-]
	(2,2) arc (90:0:.5);
\draw [*-]
	(2.5,1.5) arc (0:-90:.5);
\end{tikzpicture}

\\[10pt]
$\left(a_1^{(\overline{j})}b_2^{(j)}-a_1^{(j)}b_2^{(\overline{j})}\right)U^{(j)}a_1^{(\overline{j})}$&$+$&$\left(c_1^{(j)}c_2^{(\overline{j})}\right)D^{(j)}b_2^{(\overline{j})}$&$D^{(j)}b_2^{(j)}$
\\[10pt] \hline
%
%

\begin{tikzpicture}

\node [label=left:$\overline{j}$] at (1,2) {};
\node [label=left:$j$] at (1,1) {};

\draw [<->] (1,2) .. controls (1.5,2) and (1.5,1) .. (2,1);
\draw [>-<] (1,1) .. controls (1.5,1) and (1.5,2) .. (2,2);

\draw [-<] (2,2) -- (3,2);
\draw [->] (2,1) -- (3,1);

\draw [>->] (2.5,1.5) -- (2.5,0.5);

\draw [-]
	(3,2) arc (90:0:.5);
\draw [*-]
	(3.5,1.5) arc (0:-90:.5);
\end{tikzpicture}
&

&
\begin{tikzpicture}

\node [label=left:$\overline{j}$] at (1,2) {};
\node [label=left:${j}$] at (1,1) {};

\draw [<-<] (1,2) .. controls (1.5,2) and (1.5,1) .. (2,1);
\draw [>->] (1,1) .. controls (1.5,1) and (1.5,2) .. (2,2);

\draw [->] (2,2) -- (3,2);
\draw [-<] (2,1) -- (3,1);

\draw [>->] (2.5,1.5) -- (2.5,0.5);

\draw [-]
	(3,2) arc (90:0:.5);
\draw [*-]
	(3.5,1.5) arc (0:-90:.5);
\end{tikzpicture}
&

\begin{tikzpicture}

\node [label=left:$\overline{j}$] at (1,2) {};
\node [label=left:${j}$] at (1,1) {};

\draw [<-<] (1,2) -- (2,2);
\draw [>->] (1,1) -- (2,1);

\draw [>->] (1.5,1.5) -- (1.5,.5);

\draw [-]
	(2,2) arc (90:0:.5);
\draw [*-]
	(2.5,1.5) arc (0:-90:.5);
\end{tikzpicture}

\\[10pt]
$\left(c_1^{(\overline{j})}c_2^{({j})}\right)U^{(j)}a_1^{(\overline{j})}$&$+$&$\left(a_2^{(\overline{j})}b_1^{({j})}-a_2^{({j})}b_1^{(\overline{j})}\right)D^{(j)}b_2^{(\overline{j})}$&$U^{(j)}a_1^{({j})}$
\\[10pt] \hline

%
%

\begin{tikzpicture}

\node [label=left:$\overline{j}$] at (1,2) {};
\node [label=left:${j}$] at (1,1) {};

\draw [<-<] (1,2) .. controls (1.5,2) and (1.5,1) .. (2,1);
\draw [<-<] (1,1) .. controls (1.5,1) and (1.5,2) .. (2,2);

\draw [-<] (2,2) -- (3,2);
\draw [->] (2,1) -- (3,1);

\draw [>-<] (2.5,1.5) -- (2.5,0.5);

\draw [-]
	(3,2) arc (90:0:.5);
\draw [*-]
	(3.5,1.5) arc (0:-90:.5);
\end{tikzpicture}
&
&
&

\begin{tikzpicture}

\node [label=left:$\overline{j}$] at (1,2) {};
\node [label=left:${j}$] at (1,1) {};

\draw [<-<] (1,2) -- (2,2);
\draw [<->] (1,1) -- (2,1);

\draw [>-<] (1.5,1.5) -- (1.5,.5);

\draw [-]
	(2,2) arc (90:0:.5);
\draw [*-]
	(2.5,1.5) arc (0:-90:.5);
\end{tikzpicture}

\\[10pt]
$\left(a_1^{(\overline{j})}a_2^{({j})}+b_1^{({j})}b_2^{(\overline{j})}\right)U^{(j)}c_1^{(\overline{j})}$&$ $&$ $&$U^{(j)}c_1^{(j)}$
\\[10pt] \hline 
\end{tabular}
\caption{The various fish configurations in type $\mathfrak{C}_*$, as well as type $\mathfrak{D}$ when $1 \not\in\lambda$.}
\label{fig:Calt.and.nonzero.containing.D.fish}
\end{figure}

\end{proof}

\begin{lemma}[Fish relation, type $\mathfrak{D}$ for $1 \in\lambda$]\label{le:fish.D}
Assume $U^{(j)}, D^{(j)} \neq 0$.  A necessary and sufficient condition for the ratio of
$$
\mathcal{Z}
\left(
\begin{minipage}{1.5in}
\begin{tikzpicture}
\node [label=left:$\overline{j}$] at (-0.5,1) {};
\node [label=left:$j$] at (-0.5,0) {};
\node [circle,draw,scale=.6] at (-.25,1) {$\alpha$};
\node [circle,draw,scale=0.55] at (-.25,0) {$\beta$};
\node [circle,draw,scale=.55] at (1.5,-.75) {$\gamma$};


\draw [<-] (1.5,.5) -- (1.5,-.5);
\draw [-] (0,0) .. controls (0.5,0) and (.5,1) .. (1,1);
\draw [-] (0,1) .. controls (0.5,1) and (.5,0) .. (1,0);
\draw [-] (1,1) -- (2,1);
\draw [-] (1,0) -- (2,0);

\draw [-]
	(2,1) arc (90:0:.5);
\draw [*-]
	(2.5,.5) arc (0:-90:.5);
\end{tikzpicture}
\end{minipage}
\right)
\quad \mbox{and} \quad
\text{wt}\left(
\begin{minipage}{1.1in}
\begin{tikzpicture}
\node [label=left:$\overline{j}$] at (-.5,1) {};
\node [label=left:$j$] at (-.5,0) {};
\node [circle,draw,scale=.6] at (-.25,1) {$\alpha$};
\node [circle,draw,scale=0.55] at (-.25,0) {$\beta$};
\node [circle,draw,scale=.55] at (.5,-.75) {$\gamma$};

\draw [-] (0,1) -- (1,1);
\draw [-] (0,0) -- (1,0);
\draw [<-] (0.5,.5) -- (.5,-.5);
\draw [-]
	(1,1) arc (90:0:.5);
\draw [*-]
	(1.5,0.5) arc (0:-90:.5);
\end{tikzpicture}
\end{minipage}
\right)
$$ to be a constant independent of $\alpha,\beta$ and $\gamma$ is that $D^{(j)} = U^{(j)}$.  Under this assumption, this ratio is equal to $$\left(a_1^{(j)}\right)^2 + \left(b_2^{(j)}\right)^2.$$
\end{lemma}

\begin{proof}
Again, we simply list the possibilities and evaluate the ratio in each case.  The possible fillings are given in Figure \ref{fig:zero.containing.D.fish}.  The values of the ratio in each row are:
$$
\left(a_1^{(j)}\right)^2 +\frac{U^{(j)}}{D^{(j)}} \left(b_2^{(j)}\right)^2 +  \left(1-\frac{U^{(j)}}{D^{(j)}}\right)\frac{a_1^{(j)}b_1^{(j)}b_2^{(j)}}{a_2^{(j)}}$$
$$
 \frac{D^{(j)}}{U^{(j)}}\left(a_1^{(j)}\right)^2 + \left(b_2^{(j)}\right)^2 + \left(1-\frac{D^{(j)}}{U^{(j)}}\right)\frac{a_1^{(j)}a_2^{(j)}b_2^{(j)}}{b_1^{(j)}}$$
$$
\left(a_1^{(j)}\right)^2 + \left(b_2^{(j)}\right)^2.$$
One simply compares coefficients to get the necessary and sufficient condition; the resulting ratio can be read off the third expression.

\begin{figure}[!ht]

\begin{tabular}{|ccc|c|} \hline
%
%

\begin{tikzpicture}

\node [label=left:$\overline{j}$] at (1,2) {};
\node [label=left:${j}$] at (1,1) {};

\draw [>->] (1,2) .. controls (1.5,2) and (1.5,1) .. (2,1);
\draw [<-<] (1,1) .. controls (1.5,1) and (1.5,2) .. (2,2);

\draw [-<] (2,2) -- (3,2);
\draw [->] (2,1) -- (3,1);

\draw [<-<] (2.5,1.5) -- (2.5,0.5);

\draw [-]
	(3,2) arc (90:0:.5);
\draw [*-]
	(3.5,1.5) arc (0:-90:.5);
\end{tikzpicture}
&

&
\begin{tikzpicture}

\node [label=left:$\overline{j}$] at (1,2) {};
\node [label=left:${j}$] at (1,1) {};

\draw [>-<] (1,2) .. controls (1.5,2) and (1.5,1) .. (2,1);
\draw [<->] (1,1) .. controls (1.5,1) and (1.5,2) .. (2,2);

\draw [->] (2,2) -- (3,2);
\draw [-<] (2,1) -- (3,1);

\draw [<-<] (2.5,1.5) -- (2.5,0.5);

\draw [-]
	(3,2) arc (90:0:.5);
\draw [*-]
	(3.5,1.5) arc (0:-90:.5);
\end{tikzpicture}
&

\begin{tikzpicture}

\node [label=left:$\overline{j}$] at (1,2) {};
\node [label=left:${j}$] at (1,1) {};

\draw [>->] (1,2) -- (2,2);
\draw [<-<] (1,1) -- (2,1);

\draw [<-<] (1.5,1.5) -- (1.5,.5);

\draw [-]
	(2,2) arc (90:0:.5);
\draw [*-]
	(2.5,1.5) arc (0:-90:.5);
\end{tikzpicture}

\\[10pt]
$\left(a_1^{(\overline{j})}b_2^{({j})}-a_1^{({j})}b_2^{(\overline{j})}\right)U^{(i)}b_1^{(\overline{j})}$&$+$&$\left(c_1^{({j})}c_2^{(\overline{j})}\right)D^{(j)}a_2^{(\overline{j})}$&$D^{(j)}a_2^{({j})}$
\\[10pt] \hline

%
%

\begin{tikzpicture}

\node [label=left:$\overline{j}$] at (1,2) {};
\node [label=left:${j}$] at (1,1) {};

\draw [<->] (1,2) .. controls (1.5,2) and (1.5,1) .. (2,1);
\draw [>-<] (1,1) .. controls (1.5,1) and (1.5,2) .. (2,2);

\draw [-<] (2,2) -- (3,2);
\draw [->] (2,1) -- (3,1);

\draw [<-<] (2.5,1.5) -- (2.5,0.5);

\draw [-]
	(3,2) arc (90:0:.5);
\draw [*-]
	(3.5,1.5) arc (0:-90:.5);
\end{tikzpicture}
&

&
\begin{tikzpicture}

\node [label=left:$\overline{j}$] at (1,2) {};
\node [label=left:${j}$] at (1,1) {};

\draw [<-<] (1,2) .. controls (1.5,2) and (1.5,1) .. (2,1);
\draw [>->] (1,1) .. controls (1.5,1) and (1.5,2) .. (2,2);

\draw [->] (2,2) -- (3,2);
\draw [-<] (2,1) -- (3,1);

\draw [<-<] (2.5,1.5) -- (2.5,0.5);

\draw [-]
	(3,2) arc (90:0:.5);
\draw [*-]
	(3.5,1.5) arc (0:-90:.5);
\end{tikzpicture}
&

\begin{tikzpicture}

\node [label=left:$\overline{j}$] at (1,2) {};
\node [label=left:${j}$] at (1,1) {};

\draw [<-<] (1,2) -- (2,2);
\draw [>->] (1,1) -- (2,1);

\draw [<-<] (1.5,1.5) -- (1.5,.5);

\draw [-]
	(2,2) arc (90:0:.5);
\draw [*-]
	(2.5,1.5) arc (0:-90:.5);
\end{tikzpicture}

\\[10pt]
$\left(c_1^{(\overline{j})}c_2^{({j})}\right)U^{(j)}b_1^{(\overline{j})}$&$+$&$\left(a_2^{(\overline{j})}b_1^{({j})}-a_2^{({j})}b_1^{(\overline{j})}\right)D^{(j)}a_2^{(\overline{j})}$&$U^{(j)}b_1^{({j})}$
\\[10pt] \hline

%
%

\begin{tikzpicture}

\node [label=left:$\overline{j}$] at (1,2) {};
\node [label=left:${j}$] at (1,1) {};

\draw [>->] (1,2) .. controls (1.5,2) and (1.5,1) .. (2,1);
\draw [>->] (1,1) .. controls (1.5,1) and (1.5,2) .. (2,2);

\draw [->] (2,2) -- (3,2);
\draw [-<] (2,1) -- (3,1);

\draw [<->] (2.5,1.5) -- (2.5,0.5);

\draw [-]
	(3,2) arc (90:0:.5);
\draw [*-]
	(3.5,1.5) arc (0:-90:.5);
\end{tikzpicture}
&

&
&

\begin{tikzpicture}

\node [label=left:$\overline{j}$] at (1,2) {};
\node [label=left:${j}$] at (1,1) {};

\draw [>->] (1,2) -- (2,2);
\draw [>-<] (1,1) -- (2,1);

\draw [<->] (1.5,1.5) -- (1.5,.5);

\draw [-]
	(2,2) arc (90:0:.5);
\draw [*-]
	(2.5,1.5) arc (0:-90:.5);
\end{tikzpicture}

\\[10pt]
$\left(a_1^{({j})}a_2^{(\overline{j})}+b_1^{(\overline{j})}b_2^{({j})}\right)D^{(j)}c_2^{(\overline{j})}$&$ $&$ $&$D^{(j)}c_2^{({j})}$
\\[10pt] \hline
\end{tabular}
\caption{The fish configurations in type $\mathfrak{D}$  when $1 \in \lambda$.}
\label{fig:zero.containing.D.fish}
\end{figure}
\end{proof}

\begin{lemma}[Jellyfish relation, type $\mathfrak{C}$]\label{le:fish.C}
Assume $U^{(j)}, D^{(j)} \neq 0$.  A necessary and sufficient condition for the ratio of
$$
\mathcal{Z}
\left(
\begin{minipage}{2.3in}
\begin{tikzpicture}
	

	\node [label=left:$\overline{j}$] at (-.5,1) {};
	\node [label=left:$0$] at (-.5,0) {};
	\node [label=left:$j$] at (-.5,-1) {};

	\node [circle,draw,scale=.6] at (-.25,1) {$\alpha$};
	\node [circle,draw,scale=0.55] at (-.25,0) {$\beta$};
	\node [circle,draw,scale=.55] at (-.25,-1) {$\gamma$};
	\node [circle,draw,scale=.55] at (3.5,-1.75) {$\delta$};

	\draw [-]	(0,1) .. controls (.5,1) and (.5,0) .. (1,0);
	\draw [-]	(0,0) .. controls (.5,0) and (.5,1) .. (1,1);
	\draw [-]	(0,-1) -- (1,-1);

	\draw [-]	(1,1) -- (2,1);	
	\draw [-]	(1,0) .. controls (1.5,0) and (1.5,-1) .. (2,-1);
	\draw [-]	(1,-1) .. controls (1.5,-1) and (1.5,0) .. (2,0);
	
	\draw [-]	(2,1) .. controls (2.5,1) and (2.5,0) .. (3,0);
	\draw [-]	(2,0) .. controls (2.5,0) and (2.5,1) .. (3,1);
	\draw [-]	(2,-1) -- (3,-1);
	
	\draw [-]	(3,1) -- (3.5,1);
	\draw [-]	(3,0) -- (3.5,0);
	\draw [-]	(3,-1) -- (3.5,-1);
	
	
	\draw [-] (3.5,1) arc (90:0:1);
	\draw [*-] (4.5,0) arc (0:-90:1);
	
	\draw [*-] (3.5,0) -- (3.5,-1.5);
		
\end{tikzpicture}

\end{minipage}
\right)
\quad \mbox{and} \quad
\text{wt}\left(
\begin{minipage}{1.1in}
\begin{tikzpicture}
\node [label=left:$\overline{j}$] at (-.5,1) {};
\node [label=left:$0$] at (-.5,0) {};
\node [label=left:$j$] at (-.5,-1) {};
\node [circle,draw,scale=.6] at (-.25,1) {$\alpha$};
\node [circle,draw,scale=0.55] at (-.25,0) {$\beta$};
\node [circle,draw,scale=.55] at (-.25,-1) {$\gamma$};
\node [circle,draw,scale=.55] at (.5,-1.75) {$\delta$};

\draw [-] (0,1) -- (.5,1);
\draw [-] (0,0) -- (.5,0);
\draw [-] (0,-1) -- (.5,-1);

\draw [*-] (.5,0) -- (.5,-1.5);

\draw [-]
	(.5,1) arc (90:0:1);
\draw [*-]
	(1.5,0) arc (0:-90:1);
\end{tikzpicture}
\end{minipage}
\right)
$$ to be a polynomial in $\Z[i,a^{(0)},b^{(0)},a_1^{(j)},a_2^{(j)},b_1^{(j)},b_2^{(j)}]$ that is independent of $\alpha,\beta,\gamma$ and $\delta$ is that $L/R = a^{(0)}\pm i~b^{(0)}$ and $D^{(j)}/U^{(j)} =\mp i$.  When $L/R = a^{(0)}-ib^{(0)}$ and $D^{(j)}/U^{(j)} = i$, this ratio is equal to $$\left(a_1^{(j)}- i~b_2^{(j)}\right)\left(a_2^{(j)} + i~b_1^{(j)}\right)\left(a_1^{(j)}a^{(0)} + b^{(0)}b_2^{(j)}\right)\left(a^{(0)}a_2^{(j)}+b_1^{(j)}b^{(0)}\right).$$
\end{lemma}

\begin{proof}
Again, the proof is simply a matter of writing out all the possible fillings in each case and evaluating the partition function, though of course this time the calculations are a bit more cumbersome.  Examining two cases in detail will be enough to prove necessity of the stated conditions; sufficiency and the value of the quotient come from six cumbersome calculations that are best implemented by a computer. Figure \ref{fig:C.fish} shows the possible fillings for two choices of $\{\alpha,\beta,\gamma,\delta\}$. 

\begin{figure}[!ht]
\begin{tabular}{|ccccc|c|} \hline
%
%

\begin{tikzpicture}[scale=.5]

	\node [label=left:$\overline{j}$] at (0,1) {};
	\node [label=left:$0$] at (0,0) {};
	\node [label=left:$j$] at (0,-1) {};

	\draw [>-]	(0,1) .. controls (.5,1) and (.5,0) .. (1,0);
	\draw [>-]	(0,0) .. controls (.5,0) and (.5,1) .. (1,1);
	\draw [<-]	(0,-1) -- (1,-1);

	\draw [>-]	(1,1) -- (2,1);	
	\draw [>-]	(1,0) .. controls (1.5,0) and (1.5,-1) .. (2,-1);
	\draw [<-]	(1,-1) .. controls (1.5,-1) and (1.5,0) .. (2,0);
	
	\draw [>-]	(2,1) .. controls (2.5,1) and (2.5,0) .. (3,0);
	\draw [>-]	(2,0) .. controls (2.5,0) and (2.5,1) .. (3,1);
	\draw [<-]	(2,-1) -- (3,-1);
	
	\draw [>-]	(3,1) -- (3.5,1);
	\draw [>-]	(3,0) -- (3.5,0);
	\draw [<-]	(3,-1) -- (3.5,-1);
	
	
	\draw [-] (3.5,1) arc (90:0:1);
	\draw [*-] (4.5,0) arc (0:-90:1);
	
	\draw [*->] (3.5,0) -- (3.5,-1.5);
		
\end{tikzpicture}
&

&
\begin{tikzpicture}[scale=.5]
	
	\node [label=left:$\overline{j}$] at (0,1) {};
	\node [label=left:$0$] at (0,0) {};
	\node [label=left:$j$] at (0,-1) {};

	\draw [>-]	(0,1) .. controls (.5,1) and (.5,0) .. (1,0);
	\draw [>-]	(0,0) .. controls (.5,0) and (.5,1) .. (1,1);
	\draw [<-]	(0,-1) -- (1,-1);

	\draw [>-]	(1,1) -- (2,1);	
	\draw [>-]	(1,0) .. controls (1.5,0) and (1.5,-1) .. (2,-1);
	\draw [<-]	(1,-1) .. controls (1.5,-1) and (1.5,0) .. (2,0);
	
	\draw [>-]	(2,1) .. controls (2.5,1) and (2.5,0) .. (3,0);
	\draw [<-]	(2,0) .. controls (2.5,0) and (2.5,1) .. (3,1);
	\draw [>-]	(2,-1) -- (3,-1);
	
	\draw [>-]	(3,1) -- (3.5,1);
	\draw [<-]	(3,0) -- (3.5,0);
	\draw [>-]	(3,-1) -- (3.5,-1);
	
	
	\draw [-] (3.5,1) arc (90:0:1);
	\draw [*-] (4.5,0) arc (0:-90:1);
	
	\draw [*->] (3.5,0) -- (3.5,-1.5);
		
\end{tikzpicture}
&
&
\begin{tikzpicture}[scale=.5]
	
	\node [label=left:$\overline{j}$] at (0,1) {};
	\node [label=left:$0$] at (0,0) {};
	\node [label=left:$j$] at (0,-1) {};

	\draw [>-]	(0,1) .. controls (.5,1) and (.5,0) .. (1,0);
	\draw [>-]	(0,0) .. controls (.5,0) and (.5,1) .. (1,1);
	\draw [<-]	(0,-1) -- (1,-1);

	\draw [>-]	(1,1) -- (2,1);	
	\draw [>-]	(1,0) .. controls (1.5,0) and (1.5,-1) .. (2,-1);
	\draw [<-]	(1,-1) .. controls (1.5,-1) and (1.5,0) .. (2,0);
	
	\draw [>-]	(2,1) .. controls (2.5,1) and (2.5,0) .. (3,0);
	\draw [<-]	(2,0) .. controls (2.5,0) and (2.5,1) .. (3,1);
	\draw [>-]	(2,-1) -- (3,-1);
	
	\draw [<-]	(3,1) -- (3.5,1);
	\draw [>-]	(3,0) -- (3.5,0);
	\draw [>-]	(3,-1) -- (3.5,-1);
	
	
	\draw [-] (3.5,1) arc (90:0:1);
	\draw [*-] (4.5,0) arc (0:-90:1);
	
	\draw [*->] (3.5,0) -- (3.5,-1.5);
		
\end{tikzpicture}
&
\begin{tikzpicture}[scale=.5]
\node [label=left:$\overline{j}$] at (0,1) {};
\node [label=left:$0$] at (0,0) {};
\node [label=left:$j$] at (0,-1) {};

\draw [>-] (0,1) -- (.5,1);
\draw [>-] (0,0) -- (.5,0);
\draw [<-] (0,-1) -- (.5,-1);

\draw [*->] (.5,0) -- (.5,-1.5);

\draw [-]
	(.5,1) arc (90:0:1);
\draw [*-]
	(1.5,0) arc (0:-90:1);
\end{tikzpicture}
\\[10pt]\hline
%
%

\begin{tikzpicture}[scale=.5]

	\node [label=left:$\overline{j}$] at (0,1) {};
	\node [label=left:$0$] at (0,0) {};
	\node [label=left:$j$] at (0,-1) {};

	\draw [<-]	(0,1) .. controls (.5,1) and (.5,0) .. (1,0);
	\draw [<-]	(0,0) .. controls (.5,0) and (.5,1) .. (1,1);
	\draw [>-]	(0,-1) -- (1,-1);

	\draw [<-]	(1,1) -- (2,1);	
	\draw [<-]	(1,0) .. controls (1.5,0) and (1.5,-1) .. (2,-1);
	\draw [>-]	(1,-1) .. controls (1.5,-1) and (1.5,0) .. (2,0);
	
	\draw [<-]	(2,1) .. controls (2.5,1) and (2.5,0) .. (3,0);
	\draw [>-]	(2,0) .. controls (2.5,0) and (2.5,1) .. (3,1);
	\draw [<-]	(2,-1) -- (3,-1);
	
	\draw [>-]	(3,1) -- (3.5,1);
	\draw [<-]	(3,0) -- (3.5,0);
	\draw [<-]	(3,-1) -- (3.5,-1);
	
	
	\draw [-] (3.5,1) arc (90:0:1);
	\draw [*-] (4.5,0) arc (0:-90:1);
	
	\draw [*-<] (3.5,0) -- (3.5,-1.5);
		
\end{tikzpicture}
&

&
\begin{tikzpicture}[scale=.5]
	
	\node [label=left:$\overline{j}$] at (0,1) {};
	\node [label=left:$0$] at (0,0) {};
	\node [label=left:$j$] at (0,-1) {};

	\draw [<-]	(0,1) .. controls (.5,1) and (.5,0) .. (1,0);
	\draw [<-]	(0,0) .. controls (.5,0) and (.5,1) .. (1,1);
	\draw [>-]	(0,-1) -- (1,-1);

	\draw [<-]	(1,1) -- (2,1);	
	\draw [<-]	(1,0) .. controls (1.5,0) and (1.5,-1) .. (2,-1);
	\draw [>-]	(1,-1) .. controls (1.5,-1) and (1.5,0) .. (2,0);
	
	\draw [<-]	(2,1) .. controls (2.5,1) and (2.5,0) .. (3,0);
	\draw [>-]	(2,0) .. controls (2.5,0) and (2.5,1) .. (3,1);
	\draw [<-]	(2,-1) -- (3,-1);
	
	\draw [<-]	(3,1) -- (3.5,1);
	\draw [>-]	(3,0) -- (3.5,0);
	\draw [<-]	(3,-1) -- (3.5,-1);
	
	
	\draw [-] (3.5,1) arc (90:0:1);
	\draw [*-] (4.5,0) arc (0:-90:1);
	
	\draw [*-<] (3.5,0) -- (3.5,-1.5);
		
\end{tikzpicture}
&
&
\begin{tikzpicture}[scale=.5]
	
	\node [label=left:$\overline{j}$] at (0,1) {};
	\node [label=left:$0$] at (0,0) {};
	\node [label=left:$j$] at (0,-1) {};

	\draw [<-]	(0,1) .. controls (.5,1) and (.5,0) .. (1,0);
	\draw [<-]	(0,0) .. controls (.5,0) and (.5,1) .. (1,1);
	\draw [>-]	(0,-1) -- (1,-1);

	\draw [<-]	(1,1) -- (2,1);	
	\draw [<-]	(1,0) .. controls (1.5,0) and (1.5,-1) .. (2,-1);
	\draw [>-]	(1,-1) .. controls (1.5,-1) and (1.5,0) .. (2,0);
	
	\draw [<-]	(2,1) .. controls (2.5,1) and (2.5,0) .. (3,0);
	\draw [<-]	(2,0) .. controls (2.5,0) and (2.5,1) .. (3,1);
	\draw [>-]	(2,-1) -- (3,-1);
	
	\draw [<-]	(3,1) -- (3.5,1);
	\draw [<-]	(3,0) -- (3.5,0);
	\draw [>-]	(3,-1) -- (3.5,-1);
	
	
	\draw [-] (3.5,1) arc (90:0:1);
	\draw [*-] (4.5,0) arc (0:-90:1);
	
	\draw [*-<] (3.5,0) -- (3.5,-1.5);
		
\end{tikzpicture}
&
\begin{tikzpicture}[scale=.5]
\node [label=left:$\overline{j}$] at (0,1) {};
\node [label=left:$0$] at (0,0) {};
\node [label=left:$j$] at (0,-1) {};

\draw [<-] (0,1) -- (.5,1);
\draw [<-] (0,0) -- (.5,0);
\draw [>-] (0,-1) -- (.5,-1);

\draw [*-<] (.5,0) -- (.5,-1.5);

\draw [-]
	(.5,1) arc (90:0:1);
\draw [*-]
	(1.5,0) arc (0:-90:1);
\end{tikzpicture}
\\[10pt]\hline
\end{tabular}
\caption{Two of the six fish configurations in type $\mathfrak{C}$.}
\label{fig:C.fish}
\end{figure}
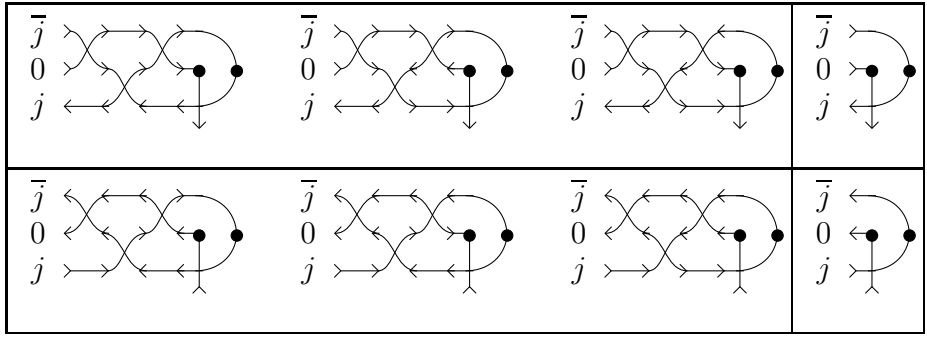

A computer algebra system allows one to efficiently compute the relevant quotients.  Insisting that the quotient be a polynomial allows one to then show that certain terms vanish.  For instance, when expanding the quotient in the first case, the following are the only terms that involve $b_2^{(j)}$ as a denominator:
$$\frac{\left(a^{(0)}\right)^2 \left(a_1^{(j)}\right)^3 a_2^{(j)} b_1^{(j)}}{b_2^{(j)}} + \frac{a^{(0)}\left(a_1^{(j)}\right)^3 a_2^{(j)}b^{(0)}b_1^{(j)}U^{(j)}}{b_2^{(j)}D^{(j)}} - \frac{a^{(0)}\left(a_1^{(j)}\right)^3a_2^{(j)}b_1^{(j)}L}{b_2^{(j)}R}.$$  In order for this quotient to be a polynomial we must have \begin{equation}\label{eq:fish.C.identity.one}a^{(0)}+b^{(0)}\frac{U^{(j)}}{D^{(j)}} - \frac{L}{R} = 0.\end{equation}  
On the other hand, setting the two quotients equal to each other forces equalities of coefficients; for example, equating terms involving $\left(a_1^{(j)}\right)^2\left(a_2^{(j)}\right)^2$ in each of the above quotients then shows that 
\begin{equation}\label{eq:fish.C.identity.three}-b^{(0)} \frac{U^{(j)}}{D^{(j)}} + \frac{L}{R} = \left(\left(a^{(0)}\right)^2+\left(b^{(0)}\right)^2\right)\frac{R}{L}-b^{(0)}\frac{D^{(j)}}{U^{(j)}}.\end{equation}

Adding together equations (\ref{eq:fish.C.identity.one}) and (\ref{eq:fish.C.identity.three}) gives
$$a^{(0)}+b^{(0)}\frac{D^{(j)}}{U^{(j)}} = \left(\left(a^{(0)}\right)^2+\left(b^{(0)}\right)^2\right)\frac{R}{L},$$
from which we recover $R/L = (a^{(0)} \pm i~b^{(0)})^{-1}$ and $D^{(j)}/U^{(j)} =\mp i$.
\end{proof}

\begin{lemma}[Jellyfish relation, type $\mathfrak{B}_*$]\label{le:fish.B*}
Assume $U^{(j)}, D^{(j)} \neq 0$.  A necessary and sufficient condition for the ratio of
$$
\mathcal{Z}
\left(
\begin{minipage}{2.3in}
\begin{tikzpicture}
	

	\node [label=left:$\overline{j}$] at (-.5,1) {};
	\node [label=left:$0$] at (-.5,0) {};
	\node [label=left:$j$] at (-.5,-1) {};

	\node [circle,draw,scale=.6] at (-.25,1) {$\alpha$};
	\node [circle,draw,scale=.55] at (-.25,-1) {$\beta$};

	\draw [-]	(0,1) .. controls (.5,1) and (.5,0) .. (1,0);
	\draw [>-]	(-.25,0) .. controls (.5,0) and (.5,1) .. (1,1);
	\draw [-]	(0,-1) -- (1,-1);

	\draw [-]	(1,1) -- (2,1);	
	\draw [-]	(1,0) .. controls (1.5,0) and (1.5,-1) .. (2,-1);
	\draw [-]	(1,-1) .. controls (1.5,-1) and (1.5,0) .. (2,0);
	
	\draw [-]	(2,1) .. controls (2.5,1) and (2.5,0) .. (3,0);
	\draw [-]	(2,0) .. controls (2.5,0) and (2.5,1) .. (3,1);
	\draw [-]	(2,-1) -- (3,-1);
	
	\draw [-]	(3,1) -- (3.5,1);
	\draw [->]	(3,0) -- (3.5,0);
	\draw [-]	(3,-1) -- (3.5,-1);
	
	
	\draw [-] (3.5,1) arc (90:0:1);
	\draw [*-] (4.5,0) arc (0:-90:1);
	
		
\end{tikzpicture}

\end{minipage}
\right)
\quad \mbox{and} \quad
\text{wt}
\left(
\begin{minipage}{.9in}
\begin{tikzpicture}
\node [label=left:$\overline{j}$] at (-.5,1) {};
\node [label=left:$j$] at (-.5,-1) {};
\node [circle,draw,scale=.6] at (-.25,1) {$\alpha$};
\node [circle,draw,scale=.55] at (-.25,-1) {$\beta$};



\draw [-]
	(0,1) arc (90:0:1);
\draw [*-]
	(1,0) arc (0:-90:1);
\end{tikzpicture}
\end{minipage}
\right)
$$ to be a polynomial in $\Z[i,a^{(0)},b^{(0)},a_1^{(j)},a_2^{(j)},b_1^{(j)},b_2^{(j)}]$ that is independent of $\alpha$ and $\beta$ is that $\left(D^{(j)}\right)^2 = \left(U^{(j)}\right)^2$.  If we set $D^{(j)}=U^{(j)}=1$, then this ratio is equal to $$\left(a^{(0)}a_2^{(j)}+b_1^{(j)}b^{(0)}\right)\left(a_1^{(j)}a^{(0)}+b^{(0)}b_2^{(j)}\right)\left(\left(a_1^{(j)}\right)^2+\left(b_2^{(j)}\right)^2\right).$$  
\end{lemma}

\begin{proof}
As usual, the proof boils down to writing out all the possibilities and performing the relevant calculations.  In this case there are only two choices for $\alpha$ and $\beta$ which lead to valid diagrams.  These various fillings in each case are given in Figure \ref{fig:B*.fish}.

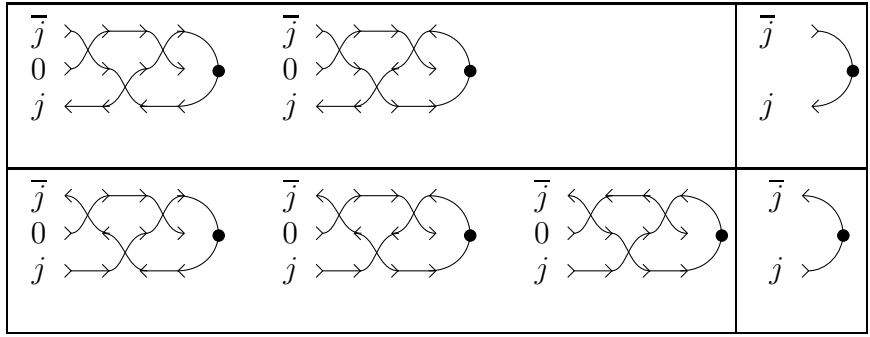
\begin{figure}[!ht]
\begin{tabular}{|ccccc|c|} \hline
%
%

\begin{tikzpicture}[scale=.5]

	\node [label=left:$\overline{j}$] at (0,1) {};
	\node [label=left:$0$] at (0,0) {};
	\node [label=left:$j$] at (0,-1) {};

	\draw [>-]	(0,1) .. controls (.5,1) and (.5,0) .. (1,0);
	\draw [>-]	(0,0) .. controls (.5,0) and (.5,1) .. (1,1);
	\draw [<-]	(0,-1) -- (1,-1);

	\draw [>-]	(1,1) -- (2,1);	
	\draw [>-]	(1,0) .. controls (1.5,0) and (1.5,-1) .. (2,-1);
	\draw [<-]	(1,-1) .. controls (1.5,-1) and (1.5,0) .. (2,0);
	
	\draw [>->]	(2,1) .. controls (2.5,1) and (2.5,0) .. (3,0);
	\draw [>->]	(2,0) .. controls (2.5,0) and (2.5,1) .. (3,1);
	\draw [<-<]	(2,-1) -- (3,-1);
	
	
	
	\draw [-] (3,1) arc (90:0:1);
	\draw [*-] (4,0) arc (0:-90:1);
	
		
\end{tikzpicture}
&

&
\begin{tikzpicture}[scale=.5]
	
	\node [label=left:$\overline{j}$] at (0,1) {};
	\node [label=left:$0$] at (0,0) {};
	\node [label=left:$j$] at (0,-1) {};

	\draw [>-]	(0,1) .. controls (.5,1) and (.5,0) .. (1,0);
	\draw [>-]	(0,0) .. controls (.5,0) and (.5,1) .. (1,1);
	\draw [<-]	(0,-1) -- (1,-1);

	\draw [>-]	(1,1) -- (2,1);	
	\draw [>-]	(1,0) .. controls (1.5,0) and (1.5,-1) .. (2,-1);
	\draw [<-]	(1,-1) .. controls (1.5,-1) and (1.5,0) .. (2,0);
	
	\draw [>->]	(2,1) .. controls (2.5,1) and (2.5,0) .. (3,0);
	\draw [<-<]	(2,0) .. controls (2.5,0) and (2.5,1) .. (3,1);
	\draw [>->]	(2,-1) -- (3,-1);
	
	
	
	\draw [-] (3,1) arc (90:0:1);
	\draw [*-] (4,0) arc (0:-90:1);
	
		
\end{tikzpicture}
&
&
&
\begin{tikzpicture}[scale=.5]
\node [label=left:$\overline{j}$] at (0,1) {};
\node [label=left:$j$] at (0,-1) {};



\draw [>-]
	(.5,1) arc (90:0:1);
\draw [*->]
	(1.5,0) arc (0:-90:1);
\end{tikzpicture}
\\[10pt]\hline
%
%

\begin{tikzpicture}[scale=.5]

	\node [label=left:$\overline{j}$] at (0,1) {};
	\node [label=left:$0$] at (0,0) {};
	\node [label=left:$j$] at (0,-1) {};

	\draw [<-]	(0,1) .. controls (.5,1) and (.5,0) .. (1,0);
	\draw [>-]	(0,0) .. controls (.5,0) and (.5,1) .. (1,1);
	\draw [>-]	(0,-1) -- (1,-1);

	\draw [>-]	(1,1) -- (2,1);	
	\draw [<-]	(1,0) .. controls (1.5,0) and (1.5,-1) .. (2,-1);
	\draw [>-]	(1,-1) .. controls (1.5,-1) and (1.5,0) .. (2,0);
	
	\draw [>->]	(2,1) .. controls (2.5,1) and (2.5,0) .. (3,0);
	\draw [>->]	(2,0) .. controls (2.5,0) and (2.5,1) .. (3,1);
	\draw [<-<]	(2,-1) -- (3,-1);
	
	
	
	\draw [-] (3,1) arc (90:0:1);
	\draw [*-] (4,0) arc (0:-90:1);
	
		
\end{tikzpicture}
&

&
\begin{tikzpicture}[scale=.5]

	\node [label=left:$\overline{j}$] at (0,1) {};
	\node [label=left:$0$] at (0,0) {};
	\node [label=left:$j$] at (0,-1) {};

	\draw [<-]	(0,1) .. controls (.5,1) and (.5,0) .. (1,0);
	\draw [>-]	(0,0) .. controls (.5,0) and (.5,1) .. (1,1);
	\draw [>-]	(0,-1) -- (1,-1);

	\draw [>-]	(1,1) -- (2,1);	
	\draw [<-]	(1,0) .. controls (1.5,0) and (1.5,-1) .. (2,-1);
	\draw [>-]	(1,-1) .. controls (1.5,-1) and (1.5,0) .. (2,0);
	
	\draw [>->]	(2,1) .. controls (2.5,1) and (2.5,0) .. (3,0);
	\draw [<-<]	(2,0) .. controls (2.5,0) and (2.5,1) .. (3,1);
	\draw [>->]	(2,-1) -- (3,-1);
	
	
	
	\draw [-] (3,1) arc (90:0:1);
	\draw [*-] (4,0) arc (0:-90:1);
	
		
\end{tikzpicture}
&
&
\begin{tikzpicture}[scale=.5]

	\node [label=left:$\overline{j}$] at (0,1) {};
	\node [label=left:$0$] at (0,0) {};
	\node [label=left:$j$] at (0,-1) {};

	\draw [<-]	(0,1) .. controls (.5,1) and (.5,0) .. (1,0);
	\draw [>-]	(0,0) .. controls (.5,0) and (.5,1) .. (1,1);
	\draw [>-]	(0,-1) -- (1,-1);

	\draw [<-]	(1,1) -- (2,1);	
	\draw [>-]	(1,0) .. controls (1.5,0) and (1.5,-1) .. (2,-1);
	\draw [>-]	(1,-1) .. controls (1.5,-1) and (1.5,0) .. (2,0);
	
	\draw [<->]	(2,1) .. controls (2.5,1) and (2.5,0) .. (3,0);
	\draw [>-<]	(2,0) .. controls (2.5,0) and (2.5,1) .. (3,1);
	\draw [>->]	(2,-1) -- (3,-1);
	
	
	
	\draw [-] (3,1) arc (90:0:1);
	\draw [*-] (4,0) arc (0:-90:1);
	
		
\end{tikzpicture}
&
\begin{tikzpicture}[scale=.5]
\node [label=left:$\overline{j}$] at (0,1) {};
\node [label=left:$j$] at (0,-1) {};



\draw [<-]
	(0,1) arc (90:0:1);
\draw [*-<]
	(1,0) arc (0:-90:1);
\end{tikzpicture}
\\[10pt]\hline
\end{tabular}
\caption{The two jellyfish relations in type $\mathfrak{B}_*$.}
\label{fig:B*.fish}
\end{figure}

In order to ensure that the quotient for each of these two possibilities is equal, one expands each partition function and compares terms; one finds relations such as
$$a^{(0)}\left(a_1^{(j)}\right)^3b^{(0)}b_1^{(j)}\frac{U^{(j)}}{D^{(j)}} = a^{(0)}\left(a_1^{(j)}\right)^3b^{(0)}b_1^{(j)} \frac{D^{(j)}}{U^{(j)}},$$ from which we deduce $\left(D^{(j)}\right)^2 = \left(U^{(j)}\right)^2.$
\end{proof}

\begin{lemma}[Jellyfish relation, type $\mathfrak{BC}$]\label{le:fish.X}
Assume $U^{(j)}, D^{(j)} \neq 0$.  A necessary and sufficient condition for the ratio of
$$
\mathcal{Z}
\left(
\begin{minipage}{2.3in}
\begin{tikzpicture}
	

	\node [label=left:$\overline{j}$] at (-.5,1) {};
	\node [label=left:$n$] at (-.5,0) {};
	\node [label=left:$j$] at (-.5,-1) {};

	\node [circle,draw,scale=.6] at (-.25,1) {$\alpha$};
	\node [circle,draw,scale=.55] at (-.25,-1) {$\beta$};

	\draw [-]	(0,1) .. controls (.5,1) and (.5,0) .. (1,0);
	\draw [<-]	(-.25,0) .. controls (.5,0) and (.5,1) .. (1,1);
	\draw [-]	(0,-1) -- (1,-1);

	\draw [-]	(1,1) -- (2,1);	
	\draw [-]	(1,0) .. controls (1.5,0) and (1.5,-1) .. (2,-1);
	\draw [-]	(1,-1) .. controls (1.5,-1) and (1.5,0) .. (2,0);
	
	\draw [-]	(2,1) .. controls (2.5,1) and (2.5,0) .. (3,0);
	\draw [-]	(2,0) .. controls (2.5,0) and (2.5,1) .. (3,1);
	\draw [-]	(2,-1) -- (3,-1);
	
	\draw [-]	(3,1) -- (3.5,1);
	\draw [-<]	(3,0) -- (3.5,0);
	\draw [-]	(3,-1) -- (3.5,-1);
	
	
	\draw [-] (3.5,1) arc (90:0:1);
	\draw [*-] (4.5,0) arc (0:-90:1);
	
		
\end{tikzpicture}
\end{minipage}
\right)
\quad \mbox{and} \quad
\text{wt}
\left(
\begin{minipage}{.9in}
\begin{tikzpicture}
\node [label=left:$\overline{j}$] at (-.5,1) {};
\node [label=left:$j$] at (-.5,-1) {};
\node [circle,draw,scale=.6] at (-.25,1) {$\alpha$};
\node [circle,draw,scale=.55] at (-.25,-1) {$\beta$};



\draw [-]
	(0,1) arc (90:0:1);
\draw [*-]
	(1,0) arc (0:-90:1);
\end{tikzpicture}
\end{minipage}
\right)
$$ to be a polynomial in $\Z[i,a^{(n)},b^{(n)},a_1^{(j)},a_2^{(j)},b_1^{(j)},b_2^{(j)}]$ that is independent of $\alpha$ and $\beta$ is that $\left(D^{(j)}\right)^2 = \left(U^{(j)}\right)^2$.  If we set $D^{(j)}=U^{(j)}=1$, then this ratio is equal to $$\left(a^{(n)}a_2^{(j)}+b_1^{(j)}b^{(n)}\right)\left(a_1^{(j)}a^{(n)}+b^{(n)}b_2^{(j)}\right)\left(\left(a_2^{(j)}\right)^2+\left(b_1^{(j)}\right)^2\right).$$  
\end{lemma}

\begin{proof}
This proof is essentially the same as the previous proof, just with different diagrams.  We will leave the proof to the reader, but provide the relevant diagrams in Figure \ref{fig:B**.fish} so these do not have to be recreated.
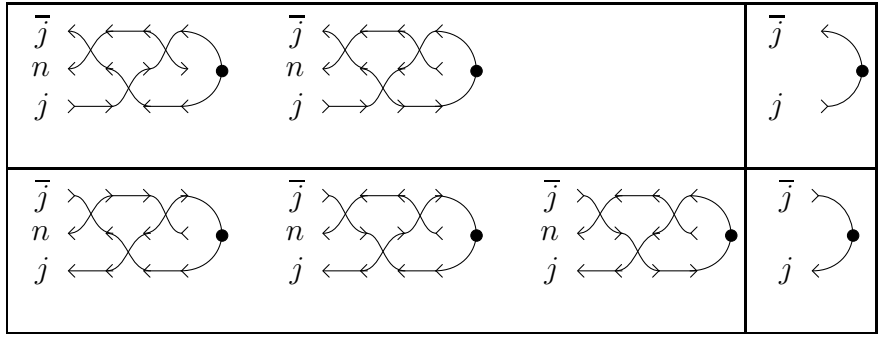
\begin{figure}[!ht]
\begin{tabular}{|ccccc|c|} \hline
%
%

\begin{tikzpicture}[scale=.5]

	\node [label=left:$\overline{j}$] at (0,1) {};
	\node [label=left:$n$] at (0,0) {};
	\node [label=left:$j$] at (0,-1) {};

	\draw [<-]	(0,1) .. controls (.5,1) and (.5,0) .. (1,0);
	\draw [<-]	(0,0) .. controls (.5,0) and (.5,1) .. (1,1);
	\draw [>-]	(0,-1) -- (1,-1);

	\draw [<-]	(1,1) -- (2,1);	
	\draw [<-]	(1,0) .. controls (1.5,0) and (1.5,-1) .. (2,-1);
	\draw [>-]	(1,-1) .. controls (1.5,-1) and (1.5,0) .. (2,0);
	
	\draw [<->]	(2,1) .. controls (2.5,1) and (2.5,0) .. (3,0);
	\draw [>-<]	(2,0) .. controls (2.5,0) and (2.5,1) .. (3,1);
	\draw [<-<]	(2,-1) -- (3,-1);
	
	
	
	\draw [-] (3,1) arc (90:0:1);
	\draw [*-] (4,0) arc (0:-90:1);
	
		
\end{tikzpicture}
&

&
\begin{tikzpicture}[scale=.5]
	
	\node [label=left:$\overline{j}$] at (0,1) {};
	\node [label=left:$n$] at (0,0) {};
	\node [label=left:$j$] at (0,-1) {};

	\draw [<-]	(0,1) .. controls (.5,1) and (.5,0) .. (1,0);
	\draw [<-]	(0,0) .. controls (.5,0) and (.5,1) .. (1,1);
	\draw [>-]	(0,-1) -- (1,-1);

	\draw [<-]	(1,1) -- (2,1);	
	\draw [<-]	(1,0) .. controls (1.5,0) and (1.5,-1) .. (2,-1);
	\draw [>-]	(1,-1) .. controls (1.5,-1) and (1.5,0) .. (2,0);
	
	\draw [<-<]	(2,1) .. controls (2.5,1) and (2.5,0) .. (3,0);
	\draw [<-<]	(2,0) .. controls (2.5,0) and (2.5,1) .. (3,1);
	\draw [>->]	(2,-1) -- (3,-1);
	
	
	
	\draw [-] (3,1) arc (90:0:1);
	\draw [*-] (4,0) arc (0:-90:1);
	
		
\end{tikzpicture}
&
&
&
\begin{tikzpicture}[scale=.5]
\node [label=left:$\overline{j}$] at (0,1) {};
\node [label=left:$j$] at (0,-1) {};



\draw [<-]
	(.5,1) arc (90:0:1);
\draw [*-<]
	(1.5,0) arc (0:-90:1);
\end{tikzpicture}
\\[10pt]\hline
%
%

\begin{tikzpicture}[scale=.5]

	\node [label=left:$\overline{j}$] at (0,1) {};
	\node [label=left:$n$] at (0,0) {};
	\node [label=left:$j$] at (0,-1) {};

	\draw [>-]	(0,1) .. controls (.5,1) and (.5,0) .. (1,0);
	\draw [<-]	(0,0) .. controls (.5,0) and (.5,1) .. (1,1);
	\draw [<-]	(0,-1) -- (1,-1);

	\draw [>-]	(1,1) -- (2,1);	
	\draw [<-]	(1,0) .. controls (1.5,0) and (1.5,-1) .. (2,-1);
	\draw [<-]	(1,-1) .. controls (1.5,-1) and (1.5,0) .. (2,0);
	
	\draw [>-<]	(2,1) .. controls (2.5,1) and (2.5,0) .. (3,0);
	\draw [<->]	(2,0) .. controls (2.5,0) and (2.5,1) .. (3,1);
	\draw [<-<]	(2,-1) -- (3,-1);
	
	
	
	\draw [-] (3,1) arc (90:0:1);
	\draw [*-] (4,0) arc (0:-90:1);
	
		
\end{tikzpicture}
&

&
\begin{tikzpicture}[scale=.5]

	\node [label=left:$\overline{j}$] at (0,1) {};
	\node [label=left:$n$] at (0,0) {};
	\node [label=left:$j$] at (0,-1) {};

	\draw [>-]	(0,1) .. controls (.5,1) and (.5,0) .. (1,0);
	\draw [<-]	(0,0) .. controls (.5,0) and (.5,1) .. (1,1);
	\draw [<-]	(0,-1) -- (1,-1);

	\draw [<-]	(1,1) -- (2,1);	
	\draw [>-]	(1,0) .. controls (1.5,0) and (1.5,-1) .. (2,-1);
	\draw [<-]	(1,-1) .. controls (1.5,-1) and (1.5,0) .. (2,0);
	
	\draw [<-<]	(2,1) .. controls (2.5,1) and (2.5,0) .. (3,0);
	\draw [>->]	(2,0) .. controls (2.5,0) and (2.5,1) .. (3,1);
	\draw [<-<]	(2,-1) -- (3,-1);
	
	
	
	\draw [-] (3,1) arc (90:0:1);
	\draw [*-] (4,0) arc (0:-90:1);
	
		
\end{tikzpicture}
&
&
\begin{tikzpicture}[scale=.5]

	\node [label=left:$\overline{j}$] at (0,1) {};
	\node [label=left:$n$] at (0,0) {};
	\node [label=left:$j$] at (0,-1) {};

	\draw [>-]	(0,1) .. controls (.5,1) and (.5,0) .. (1,0);
	\draw [<-]	(0,0) .. controls (.5,0) and (.5,1) .. (1,1);
	\draw [<-]	(0,-1) -- (1,-1);

	\draw [<-]	(1,1) -- (2,1);	
	\draw [>-]	(1,0) .. controls (1.5,0) and (1.5,-1) .. (2,-1);
	\draw [<-]	(1,-1) .. controls (1.5,-1) and (1.5,0) .. (2,0);
	
	\draw [<-<]	(2,1) .. controls (2.5,1) and (2.5,0) .. (3,0);
	\draw [<-<]	(2,0) .. controls (2.5,0) and (2.5,1) .. (3,1);
	\draw [>->]	(2,-1) -- (3,-1);
	
	
	
	\draw [-] (3,1) arc (90:0:1);
	\draw [*-] (4,0) arc (0:-90:1);
	
		
\end{tikzpicture}
&
\begin{tikzpicture}[scale=.5]
\node [label=left:$\overline{j}$] at (0,1) {};
\node [label=left:$j$] at (0,-1) {};



\draw [>-]
	(0,1) arc (90:0:1);
\draw [*->]
	(1,0) arc (0:-90:1);
\end{tikzpicture}
\\[10pt]\hline
\end{tabular}
\caption{The two jellyfish relations for type $\mathfrak{BC}$.}
\label{fig:B**.fish}
\end{figure}
\end{proof}

Since each of the separate models requires slightly different choices for $D^{(j)}/U^{(j)}$, we will summarize the conventions used in the rest of the paper in Table \ref{tab:summarized.bending.weights}.  If one is willing to assume $U^{(j)} \neq 0$ there is no loss of generality in selecting $U^{(j)}=1$, so assume this. Likewise, in the $\mathfrak{C}$ regime we will set $R=1$ without loss of generality. Lemma \ref{le:bend.ybe} shows that if $c_1^{(j)} \neq 0$ for all $j$, then $D^{(j)}$ must be chosen independent of $j$.  With the exception of the $\mathfrak{B}_*$, $\mathfrak{C}$ and $\mathfrak{BC}$, this determines the value of $D^{(j)}$.  In the case of $\mathfrak{B}_*$ and $\mathfrak{BC}$, however, there seems to be some meaningful flexibility in the choice for $D^{(j)}$ in order for the relevant ``jellyfish equations" to hold.  We will see later, however, that for each of these models there is a particular value that leads to a desirable factorization for the corresponding partition function.  It is these choices that we record in Table \ref{tab:summarized.bending.weights}.

\begin{table}[!ht]
\centering
\begin{tabular}{|c|c|c|c|c|c|c|}
\hline
Model&$\mathfrak{B}$&$\mathfrak{B}^*$&$\mathfrak{C}$&$\mathfrak{C}^*$&$\mathfrak{D}$&$\mathfrak{BC}$\\
\hline
$D^{(j)}$&$i$&$1$&$i$&$1$&$1$&$1$\\
\hline
$L$&N/A&N/A&$a^{(0)}-i~b^{(0)}$&N/A&N/A&N/A\\
\hline
\end{tabular}
\caption{Our conventions for assigning weights along the bend, as well as the corner vertex in the $\mathfrak{C}^\lambda$ model.}
\label{tab:summarized.bending.weights}
\end{table}

Finally, we observe that when $c_1^{(j)}=0$, Lemma \ref{le:bend.ybe} no longer ensures that $D^{(j)}/U^{(j)} = D^{(k)}/U^{(k)}$ for $k \neq j$.  However, the various fish and jellyfish equations still impose significant restrictions on the ratio $D^{(j)}/U^{(j)}$ (allowing at most a difference in sign), hence our assumption that the value of $D^{(j)}$ is constant regardless of $j$.

\section{Computing factors of partition functions \label{usingYBE}}

If $\mathfrak{M} \in \{\mathfrak{B},\mathfrak{B}_*,\mathfrak{C},\mathfrak{C}_*,\mathfrak{D},\mathfrak{BC}\}$, recall that $\mathcal{Z}(\mathfrak{M}^\lambda)$ denotes the corresponding partition function.  Our goal for this section is to compute explicit factors of $\mathcal{Z}(\mathfrak{M}^\lambda)$, and to determine $\mathcal{Z}(\mathfrak{M}^\lambda)$ completely when $\lambda = \rho = [n,n-1,\cdots,1]$.  Since our partition functions are homogeneous polynomials in the variables $\{a_1^{(j)},a_2^{(j)},b_1^{(j)},b_2^{(j)},c_1^{(j)},c_2^{(j)}: 1 \leq j \leq n\}$ (and perhaps the variables $a^{(0)}$ and $b^{(0)}$, depending on the family) and we have the relations $a_1^{(j)}a_2^{(j)}+b_1^{(j)}b_2^{(j)}=c_1^{(j)}c_2^{(j)}$, we specialize to the case $c_2^{(j)}=1$ without loss of generality.  Note in particular this means $c_1^{(j)}=a_1^{(j)}a_2^{(j)}+b_1^{(j)}b_2^{(j)}$, that $c_1^{(0)}=\left(a^{(0)}\right)^2 + \left(b^{(0)}\right)^2$ in the $\mathfrak{B}_*$ and $\mathfrak{C}$ families, and that $c_1^{(n)} = \left(a^{(n)}\right)^2 + \left(b^{(n)}\right)^2$ in the $\mathfrak{BC}$ family.  

\begin{lemma}\label{le:transposition.invariance}
For any $\mathfrak{M} \in \{\mathfrak{B},\mathfrak{B}_*,\mathfrak{C},\mathfrak{C}_*,\mathfrak{D},\mathfrak{BC}\}$ the expression $$\left(a_1^{(j)} a_2^{(j+1)} + b_1^{(j+1)}b_2^{(j)}\right) \mathcal{Z}(\mathfrak{M}^\lambda)$$ is invariant under the spectral index exchange $j \leftrightarrow j+1$ .
\end{lemma}
\begin{remark*}
In all of our diagrams, we include a dashed half-column since not all models include this line, as well as boxes along the top boundary to indicate that these orientations are determined by $\lambda$ and the underlying family. (We will use similar notation in subsequent lemmas as well.)

\end{remark*}
\begin{proof}
We claim that 
\begin{equation}\label{eq:twisted.diagram.type.B}
\left(a_1^{(j)}a_2^{(j+1)}+b_1^{(j+1)}b_2^{(j)}\right)\mathcal{Z}(\mathfrak{M}^\lambda) = 
\mathcal{Z}
\left(
\begin{minipage}{2.7in}
  \begin{tikzpicture}[scale=0.25]

\node [label=left:$1$] at (0,17) {};
\node [label=left:${j}$] at (-4,11) {};
\node [label=left:${j+1}$] at (-4,13) {};
\node [label=left:$\overline{{j+1}}$] at (0,7) {};
\node [label=left:$\overline{j}$] at (0,5) {};
\node [label=left:$\overline{1}$] at (0,1) {};
\node [rectangle,draw,scale=1,style=densely dashed] at (1,18.7) {};
\node [rectangle,draw,scale=1,style=densely dashed] at (5,18.7) {};
\node [rectangle,draw,scale=1,style=densely dashed] at (7,18.7) {};
\node [rectangle,draw,scale=1,style=densely dashed] at (9,9) {};
\node at (0,15.5) {$\vdots$};
\node at (0,9.5) {$\vdots$};
\node at (0,3.5) {$\vdots$};
\node at (3,18) {$\cdots$};

\draw [>->] (-4,13) .. controls (-2,13) and (-2,11) .. (0,11);
\draw [>->] (-4,11) .. controls (-2,11) and (-2,13) .. (0,13);

\draw [>-] (0,17) -- (9,17);
\draw [>-] (0,13) -- (9,13);
\draw [>-] (0,11) -- (9,11);
\draw [>-] (0,7) -- (9,7);
\draw [>-] (0,5) -- (9,5);
\draw [>-] (0,1) -- (9,1);

\draw [->] (1,18) -- (1,0);
\draw [->] (5,18) -- (5,0);
\draw [->] (7,18) -- (7,0);
\draw [style=densely dashed,->] (9,8.25) -- (9,0);

\draw [-]
	(9,17) arc (90:0:8);
\draw [*-]
	(17,9) arc (0:-90:8);

\draw [-]
	(9,13) arc (90:0:4);
\draw [*-]
	(13,9) arc (0:-90:4);

\draw [-]
	(9,11) arc (90:0:2);
\draw [*-]
	(11,9) arc (0:-90:2);

\end{tikzpicture}
\end{minipage}
\right)
\end{equation}
This follows because each filling of this diagram yields an admissible configuration of $\mathfrak{M}^\lambda$ that is attached to the twisting diagram that satisfies 
$$\text{wt}\left(\begin{minipage}{1.5in}
\begin{centering}\begin{tikzpicture}[scale=0.25] 
\node [label=left:${j}$] at (-4,-1) {};
\node [label=left:${j+1}$] at (-4,1) {};
\node [label=right:${j}$] at (0,1) {};
\node [label=right:${j+1}$] at (0,-1) {};
\draw [>->] (-4,1) .. controls (-2,1) and (-2,-1) .. (0,-1);
\draw [>->] (-4,-1) .. controls (-2,-1) and (-2,1) .. (0,1);
\end{tikzpicture}\end{centering}\end{minipage}
\right) = a_1^{(j)}a_2^{(j+1)}+b_1^{(j)}b_2^{(j+1)}.$$

Iteratively applying the Yang-Baxter equation, together with Lemma \ref{le:bend.ybe} along the bend, tells us that the partition function for diagram (\ref{eq:twisted.diagram.type.B}) is equal to 
$$
\mathcal{Z}\left(
\begin{minipage}{2.3in}
\begin{centering}
  \begin{tikzpicture}[scale=0.25]

\node [label=left:$1$] at (0,17) {};
\node [label=left:${j}$] at (0,11) {};
\node [label=left:${j+1}$] at (0,13) {};
\node [label=left:$\overline{{j+1}}$] at (0,7) {};
\node [label=left:$\overline{j}$] at (0,5) {};
\node [label=left:$\overline{1}$] at (0,1) {};
\node [rectangle,draw,scale=1,style=densely dashed] at (1,18.7) {};
\node [rectangle,draw,scale=1,style=densely dashed] at (5,18.7) {};
\node [rectangle,draw,scale=1,style=densely dashed] at (7,18.7) {};
\node [rectangle,draw,scale=1,style=densely dashed] at (9,9) {};
\node at (0,15.5) {$\vdots$};
\node at (0,9.5) {$\vdots$};
\node at (0,3.5) {$\vdots$};
\node at (3,18) {$\cdots$};

\draw [-] (1,13) .. controls (3,13) and (3,11) .. (5,11);
\draw [-] (1,11) .. controls (3,11) and (3,13) .. (5,13);

\draw [>-] (0,17) -- (9,17);
\draw [>-] (0,13) -- (1,13);
\draw [-] (5,13) -- (9,13);
\draw [>-] (0,11) -- (1,11);
\draw [-] (5,11) -- (9,11);
\draw [>-] (0,7) -- (9,7);
\draw [>-] (0,5) -- (9,5);
\draw [>-] (0,1) -- (9,1);

\draw [->] (1,18) -- (1,0);
\draw [->] (5,18) -- (5,0);
\draw [->] (7,18) -- (7,0);
\draw [style=densely dashed,->] (9,8.25) -- (9,0);

\draw [-]
	(9,17) arc (90:0:8);
\draw [*-]
	(17,9) arc (0:-90:8);

\draw [-]
	(9,13) arc (90:0:4);
\draw [*-]
	(13,9) arc (0:-90:4);

\draw [-]
	(9,11) arc (90:0:2);
\draw [*-]
	(11,9) arc (0:-90:2);

\end{tikzpicture}
\end{centering}\end{minipage}
\right)
= \cdots = 
\mathcal{Z}\left(
\begin{minipage}{2.3in}
\begin{centering}
  \begin{tikzpicture}[scale=0.25]

\node [label=left:$1$] at (0,17) {};
\node [label=left:${j}$] at (0,11) {};
\node [label=left:${j+1}$] at (0,13) {};
\node [label=left:$\overline{{j+1}}$] at (0,7) {};
\node [label=left:$\overline{j}$] at (0,5) {};
\node [label=left:$\overline{1}$] at (0,1) {};
\node [rectangle,draw,scale=1,style=densely dashed] at (1,18.7) {};
\node [rectangle,draw,scale=1,style=densely dashed] at (5,18.7) {};
\node [rectangle,draw,scale=1,style=densely dashed] at (7,18.7) {};
\node [rectangle,draw,scale=1,style=densely dashed] at (9,9) {};
\node at (0,15.5) {$\vdots$};
\node at (0,9.5) {$\vdots$};
\node at (0,3.5) {$\vdots$};
\node at (3,18) {$\cdots$};

\draw [-] (1,5) .. controls (3,5) and (3,7) .. (5,7);
\draw [-] (1,7) .. controls (3,7) and (3,5) .. (5,5);

\draw [>-] (0,17) -- (9,17);
\draw [>-] (0,13) -- (9,13);
\draw [>-] (0,11) -- (9,11);
\draw [>-] (0,7) -- (1,7);
\draw [-] (5,7) -- (9,7);
\draw [>-] (0,5) -- (1,5);
\draw [-] (5,5) -- (9,5);
\draw [>-] (0,1) -- (9,1);

\draw [->] (1,18) -- (1,0);
\draw [->] (5,18) -- (5,0);
\draw [->] (7,18) -- (7,0);
\draw [style=densely dashed,->] (9,8.25) -- (9,0);

\draw [-]
	(9,17) arc (90:0:8);
\draw [*-]
	(17,9) arc (0:-90:8);

\draw [-]
	(9,13) arc (90:0:4);
\draw [*-]
	(13,9) arc (0:-90:4);

\draw [-]
	(9,11) arc (90:0:2);
\draw [*-]
	(11,9) arc (0:-90:2);

\end{tikzpicture}
\end{centering}\end{minipage}
\right)
$$

One last application of the Yang-Baxter equation shows us that this quantity is equal to
$$
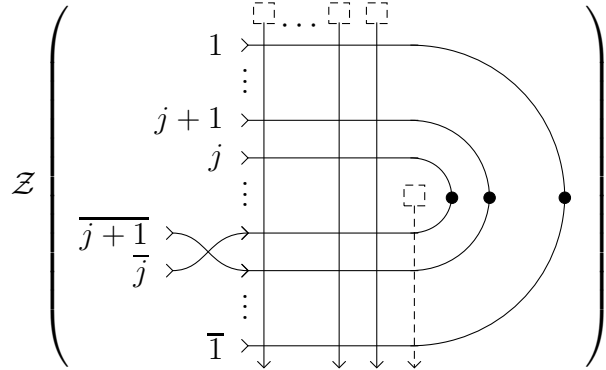
\mathcal{Z}
\left(  
\begin{minipage}{2.7in}
\begin{centering}
\begin{tikzpicture}[scale=0.25]

\node [label=left:$1$] at (0,17) {};
\node [label=left:${j+1}$] at (0,13) {};
\node [label=left:${j}$] at (0,11) {};
\node [label=left:$\overline{{j+1}}$] at (-4,7) {};
\node [label=left:$\overline{j}$] at (-4,5) {};
\node [label=left:$\overline{1}$] at (0,1) {};
\node [rectangle,draw,scale=1,style=densely dashed] at (1,18.7) {};
\node [rectangle,draw,scale=1,style=densely dashed] at (5,18.7) {};
\node [rectangle,draw,scale=1,style=densely dashed] at (7,18.7) {};
\node [rectangle,draw,scale=1,style=densely dashed] at (9,9) {};
\node at (0,15.5) {$\vdots$};
\node at (0,9.5) {$\vdots$};
\node at (0,3.5) {$\vdots$};
\node at (3,18) {$\cdots$};

\draw [>->] (-4,5) .. controls (-2,5) and (-2,7) .. (0,7);
\draw [>->] (-4,7) .. controls (-2,7) and (-2,5) .. (0,5);

\draw [>-] (0,17) -- (9,17);
\draw [>-] (0,13) -- (9,13);
\draw [>-] (0,11) -- (9,11);
\draw [>-] (0,7) -- (9,7);
\draw [>-] (0,5) -- (9,5);
\draw [>-] (0,1) -- (9,1);

\draw [->] (1,18) -- (1,0);
\draw [->] (5,18) -- (5,0);
\draw [->] (7,18) -- (7,0);
\draw [style=densely dashed,->] (9,8.25) -- (9,0);

\draw [-]
	(9,17) arc (90:0:8);
\draw [*-]
	(17,9) arc (0:-90:8);

\draw [-]
	(9,13) arc (90:0:4);
\draw [*-]
	(13,9) arc (0:-90:4);

\draw [-]
	(9,11) arc (90:0:2);
\draw [*-]
	(11,9) arc (0:-90:2);
\end{tikzpicture}
\end{centering}
\end{minipage}
\right)
$$ (Notice in particular that in the twisted portion of the diagram at the bottom, the fact that the arrows on the left side of the twist both point inward forces the arrows on the right side of the twist to point outward.) 

This latter expression is simply $$\left(a_1^{(\overline{j})}a_2^{(\overline{j+1})} + b_1^{(\overline{j+1})}b_2^{(\overline{j})}\right) \widehat{\mathcal{Z}}(\mathfrak{M}^\lambda) = \left(a_1^{(j+1)}a_2^{(j)}+b_1^{(j)}b_2^{(j+1)}\right)\widehat{\mathcal{Z}}(\mathfrak{M}^\lambda),$$ where here $\widehat{\mathcal{Z}}(\mathfrak{M}^\lambda)$ denotes the partition function that one gets by exchanging the role of spectral indices $j$ and $j+1$ in the original family. Hence we have $$\left(a_1^{(j)} a_2^{(j+1)} + b_1^{(j+1)}b_2^{(j)}\right) \mathcal{Z}(\mathfrak{M}^\lambda) = \left(a_1^{(j+1)}a_2^{(j)}+b_1^{(j)}b_2^{(j+1)}\right)\widehat{\mathcal{Z}}(\mathfrak{M}^\lambda).$$  Clearly the right had side of this equation is the result of the exchange of spectral indices  $j \leftrightarrow j+1$, and so we have the desired invariance.
%
\end{proof}

\begin{lemma}\label{le:fish.invariance}
When using the bending weights from Table \ref{tab:summarized.bending.weights}, the quantities 
$$
\begin{array}{ll}
 \mathcal{Z}(\mathfrak{B}^\lambda) \left(a_1^{(n)} + i~b_2^{(n)}\right)&\\[10pt]
\mathcal{Z}(\mathfrak{C}_*^\lambda)\left(\left(a_1^{(n)}\right)^2+\left(b_2^{(n)}\right)^2\right) &\\[10pt]
\mathcal{Z}(\mathfrak{D}^\lambda)\left(\left(a_1^{(n)}\right)^2+\left(b_2^{(n)}\right)^2\right)&\qquad \text{(if $1 \not\in\lambda$)}\\[10pt]
\mathcal{Z}(\mathfrak{D}^\lambda) &\qquad\text{(if $1 \in\lambda$)}
\end{array}$$
are each invariant under the spectral index inversion $n \leftrightarrow \overline{n}$.
\end{lemma}

\begin{proof}

Let $\mathfrak{M} \in \{\mathfrak{B},\mathfrak{C}_*,\mathfrak{D}\}$. Notice that we have
\begin{equation*}
\left(\left(a_1^{(n)}\right)^2+\left(b_2^{(n)}\right)^2\right)\mathcal{Z}(\mathfrak{M}^\lambda) = 
\mathcal{Z}\left(
\begin{minipage}{2.45in}
\begin{centering}
  \begin{tikzpicture}[scale=0.25]

\node [label=left:$1$] at (0,17.5) {};
\node [label=left:${n-1}$] at (0,13.5) {};
\node [label=left:$\overline{{n}}$] at (-4,11.5) {};
\node [label=left:${n}$] at (-4,9) {};
\node [label=left:$\overline{{n-1}}$] at (0,7) {};
\node [label=left:$\overline{1}$] at (0,3) {};
\node [rectangle,draw,scale=1,style=densely dashed] at (1,19.2) {};
\node [rectangle,draw,scale=1,style=densely dashed] at (5,19.2) {};
\node [rectangle,draw,scale=1,style=densely dashed] at (7,19.2) {};
\node [rectangle,draw,scale=1,style=densely dashed] at (9,10.25) {};
\node at (0,15.5) {$\vdots$};
\node at (0,5) {$\vdots$};
\node at (3,18.5) {$\cdots$};

\draw [>->] (-4,9) .. controls (-2,9) and (-2,11.5) .. (0,11.5);
\draw [>->] (-4,11.5) .. controls (-2,11.5) and (-2,9) .. (0,9);

\draw [>-] (0,17.5) -- (10,17.5);
\draw [>-] (0,13.5) -- (10,13.5);
\draw [>-] (0,11.5) -- (10,11.5);
\draw [>-] (0,9) -- (10,9);
\draw [>-] (0,7) -- (10,7);
\draw [>-] (0,3) -- (10,3);

\draw [->] (1,18.5) -- (1,2);
\draw [->] (5,18.5) -- (5,2);
\draw [->] (7,18.5) -- (7,2);
\draw [style=densely dashed,->] (9,9.25) -- (9,2);

\draw [-]
	(10,17.5) arc (90:0:7.25);
\draw [*-]
	(17.25,10.25) arc (0:-90:7.25);

\draw [-]
	(10,13.5) arc (90:0:3.25);
\draw [*-]
	(13.25,10.25) arc (0:-90:3.25);

\draw [-]
	(10,11.5) arc (90:0:1.25);
\draw [*-]
	(11,10.25) arc (0:-90:1.25);

\end{tikzpicture}
\end{centering}
\end{minipage}
\right)
\end{equation*}

The Yang-Baxter equation allows us to ``pass the twist" through the diagram without disturbing the overall partition function, until eventually we arrive at
$$\mathcal{Z}\left(
\begin{minipage}{2.45in}
\begin{centering}
  \begin{tikzpicture}[scale=0.25]

\node [label=left:$1$] at (0,17.5) {};
\node [label=left:${n-1}$] at (0,13.5) {};
\node [label=left:$\overline{{n}}$] at (-4,11.5) {};
\node [label=left:${n}$] at (-4,9) {};
\node [label=left:$\overline{{n-1}}$] at (0,7) {};
\node [label=left:$\overline{1}$] at (0,3) {};
\node [rectangle,draw,scale=1,style=densely dashed] at (1,19.2) {};
\node [rectangle,draw,scale=1,style=densely dashed] at (5,19.2) {};
\node [rectangle,draw,scale=1,style=densely dashed] at (7,19.2) {};
\node [rectangle,draw,scale=1,style=densely dashed] at (9,10.25) {};
\node at (0,15.5) {$\vdots$};
\node at (0,5) {$\vdots$};
\node at (3,18.5) {$\cdots$};

\draw [>->] (-4,9) .. controls (-2,9) and (-2,11.5) .. (0,11.5);
\draw [>->] (-4,11.5) .. controls (-2,11.5) and (-2,9) .. (0,9);

\draw [>-] (0,17.5) -- (10,17.5);
\draw [>-] (0,13.5) -- (10,13.5);
\draw [>-] (0,11.5) -- (10,11.5);
\draw [>-] (0,9) -- (10,9);
\draw [>-] (0,7) -- (10,7);
\draw [>-] (0,3) -- (10,3);

\draw [->] (1,18.5) -- (1,2);
\draw [->] (5,18.5) -- (5,2);
\draw [->] (7,18.5) -- (7,2);
\draw [style=densely dashed,->] (9,9.5) -- (9,2);

\draw [-]
	(10,17.5) arc (90:0:7.25);
\draw [*-]
	(17.25,10.25) arc (0:-90:7.25);

\draw [-]
	(10,13.5) arc (90:0:3.25);
\draw [*-]
	(13.25,10.25) arc (0:-90:3.25);

\draw [-]
	(10,11.5) arc (90:0:1.25);
\draw [*-]
	(11.25,10.25) arc (0:-90:1.25);

\end{tikzpicture}
\end{centering}
\end{minipage}
\right)
 = \cdots = 
\mathcal{Z}\left(\begin{minipage}{2.35in}
\begin{centering}
  \begin{tikzpicture}[scale=0.25]

\node [label=left:$1$] at (0,17.5) {};
\node [label=left:${n-1}$] at (0,13.5) {};
\node [label=left:$\overline{{n}}$] at (0,11.5) {};
\node [label=left:${n}$] at (0,9) {};
\node [label=left:$\overline{{n-1}}$] at (0,7) {};
\node [label=left:$\overline{1}$] at (0,3) {};
\node [rectangle,draw,scale=1,style=densely dashed] at (1,19.2) {};
\node [rectangle,draw,scale=1,style=densely dashed] at (5,19.2) {};
\node [rectangle,draw,scale=1,style=densely dashed] at (7,19.2) {};
\node [rectangle,draw,scale=1,style=densely dashed] at (9.5,10.25) {};
\node at (0,15.5) {$\vdots$};
\node at (0,5) {$\vdots$};
\node at (3,18.5) {$\cdots$};

\draw [-] (7.5,9) .. controls (8,9) and (8,11.5) .. (9.5,11.5);
\draw [-] (7.5,11.5) .. controls (8,11.5) and (8,9) .. (9.5,9);

\draw [>-] (0,17.5) -- (10,17.5);
\draw [>-] (0,13.5) -- (10,13.5);
\draw [>-] (0,11.5) -- (7,11.5);
\draw [>-] (0,9) -- (7,9);
\draw [>-] (0,7) -- (10,7);
\draw [>-] (0,3) -- (10,3);

\draw [->] (1,18.5) -- (1,2);
\draw [->] (5,18.5) -- (5,2);
\draw [->] (7,18.5) -- (7,2);
\draw [style=densely dashed,->] (9.5,9.5) -- (9.5,2);

\draw [-]
	(10,17.5) arc (90:0:7.25);
\draw [*-]
	(17.25,10.25) arc (0:-90:7.25);

\draw [-]
	(10,13.5) arc (90:0:3.25);
\draw [*-]
	(13.25,10.25) arc (0:-90:3.25);

\draw [-]
	(10,11.5) arc (90:0:1.25);
\draw [*-]
	(11.25,10.25) arc (0:-90:1.25);

\end{tikzpicture}
\end{centering}
\end{minipage}
\right)
$$
\comment{
Suppose that we have a filling $F$ of this diagram except at the inner portion of the central twist; zooming in on this component, we have:

$$
\begin{tikzpicture}
\node [label=left:$\overline{n}$] at (-0.5,1) {};
\node [label=left:$n$] at (-0.5,0) {};

\node [circle,draw,scale=.6] at (-.25,1) {$\alpha$};
\node [circle,draw,scale=0.55] at (-.25,0) {$\beta$};
\node [circle,draw,scale=0.55,style=densely dashed] at (1.25,-.75) {$\gamma$};
\node [rectangle,draw,scale=.6,style=densely dashed] at (1.25,.5) {$\delta$};


\draw [-] (0,0) .. controls (0.5,0) and (.5,1) .. (1,1);
\draw [-] (0,1) .. controls (0.5,1) and (.5,0) .. (1,0);
\draw [style=densely dashed,-] (1.25,.25) -- (1.25,-.5);
\draw [-] (1,1) -- (1.5,1);
\draw [-] (1,0) -- (1.5,0);

\draw [-]
	(1.5,1) arc (90:0:.5);
\draw [*-]
	(2,.5) arc (0:-90:.5);
\end{tikzpicture}
$$}

According to Lemma \ref{le:fish.B}, \ref{le:fish.C*} or \ref{le:fish.D} (depending on the model), there exists a polynomial $P$ such that this latter partition function equals $P \widehat{\mathcal{Z}}(\mathfrak{M}^\lambda)$, where $\widehat{Z}(\mathfrak{M}^\lambda)$ is the partition function for the model, but with the roles of spectral indices $n$ and $\overline{n}$ reversed.

Hence we have
$$\left(\left(a_1^{(n)}\right)^2+\left(b_2^{(n)}\right)^2\right)\mathcal{Z}(\mathfrak{M}^\lambda) =  \left\{\begin{array}{ll}
\left(a_2^{(n)}+i~b_1^{(n)}\right)\left(a_1^{(n)}-i~b_2^{(n)}\right)~\widehat{\mathcal{Z}}(\mathfrak{B}^\lambda) &\quad \mbox{if } \mathfrak{M}^\lambda = \mathfrak{B}^\lambda, \\ [10pt]
\left(\left(a_2^{(n)}\right)^2+\left(b_1^{(n)}\right)^2\right)~\widehat{\mathcal{Z}}(\mathfrak{C}_*^\lambda) &\quad \mbox{if } \mathfrak{M}^\lambda = \mathfrak{C}_*^\lambda, \\ [10pt]
\left(\left(a_2^{(n)}\right)^2+\left(b_1^{(n)}\right)^2\right)~\widehat{\mathcal{Z}}(\mathfrak{D}^\lambda) & \mkern-18mu \mbox{if } \mathfrak{M}^\lambda = \mathfrak{D}^\lambda, 1 \not\in \lambda, \\ [10pt]
\left(\left(a_1^{(n)}\right)^2+\left(b_2^{(n)}\right)^2\right)~\widehat{\mathcal{Z}}(\mathfrak{D}^\lambda) & \mkern-18mu \mbox{if } \mathfrak{M}^\lambda = \mathfrak{D}^\lambda, 1 \in \lambda. \\
\end{array}\right.
$$

After canceling like factors in each case, we get
\begin{align*}
\left(a_1^{(n)}+i~b_2^{(n)}\right)\mathcal{Z}(\mathfrak{B}^\lambda) &=  \left(a_2^{(n)}+i~b_1^{(n)}\right)\widehat{\mathcal{Z}}(\mathfrak{B}^\lambda)\\[10pt]
\left(\left(a_1^{(n)}\right)^2+\left(b_2^{(n)}\right)^2\right)\mathcal{Z}(\mathfrak{C}_*^\lambda) &=\left(\left(a_2^{(n)}\right)^2+\left(b_1^{(n)}\right)^2\right)~\widehat{\mathcal{Z}}(\mathfrak{C}_*^\lambda)\\[10pt]
\left(\left(a_1^{(n)}\right)^2+\left(b_2^{(n)}\right)^2\right)\mathcal{Z}(\mathfrak{D}^\lambda) &= \left(\left(a_2^{(n)}\right)^2+\left(b_1^{(n)}\right)^2\right)~\widehat{\mathcal{Z}}(\mathfrak{D}^\lambda)&\text{(if $1 \not\in \lambda$)}\\[10pt]
\mathcal{Z}(\mathfrak{D}^\lambda) &= \widehat{\mathcal{Z}}(\mathfrak{D}^\lambda) &\text{(if $1 \in \lambda$).}
\end{align*}
  In each case, the terms on either side of the equation are easily seen to be equal under the exchange $n \leftrightarrow \overline{n}$.  This gives the desired result.
\end{proof}

\begin{lemma}\label{le:fish.invariance.C}
When using the bending weights from Table \ref{tab:summarized.bending.weights}, the quantities 
\begin{align*}
&\left(a_1^{(n)}a^{(0)} + b^{(0)} b_2^{(n)}\right)\mathcal{Z}(\mathfrak{B}_*^\lambda)\\
&\left(a_1^{(n)}a^{(0)} + b^{(0)} b_2^{(n)}\right)\left(a_1^{(n)}+ i~b_2^{(n)}\right)\mathcal{Z}(\mathfrak{C}^\lambda)\\
\end{align*}
are each invariant under the spectral index inversion $n \leftrightarrow \overline{n}$.  Likewise, the quantity 
$$
\left(a_1^{(n-1)}a^{(n)} + b^{(n)} b_2^{(n-1)}\right)\left(\left(a_1^{(n-1)}\right)^2+\left(b_2^{(n-1)}\right)^2\right)\mathcal{Z}(\mathfrak{BC}^\lambda)\\
$$
is invariant under $n-1 \leftrightarrow \overline{n-1}$.
\end{lemma}

\begin{proof}

If $\mathfrak{M} \in \{\mathfrak{B}_*,\mathfrak{C}\}$, then $\left(a_1^{(n)}a^{(0)} + b^{(0)} b_2^{(n)}\right)^2\left(\left(a_1^{(n)}\right)^2+\left(b_2^{(n)}\right)^2\right)
\mathcal{Z}(\mathfrak{M}^\lambda)$ is equal to
$$
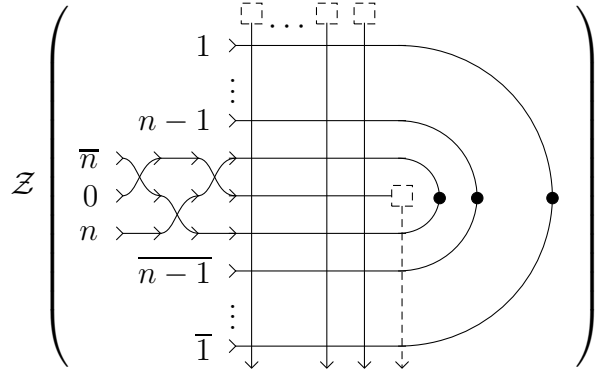
\mathcal{Z}\left(
\begin{minipage}{2.65in}
\begin{centering}
  \begin{tikzpicture}[scale=0.25]

\node [label=left:$1$] at (0,19) {};
\node [label=left:${n-1}$] at (0,15) {};
\node [label=left:$\overline{{n}}$] at (-6,13) {};
\node [label=left:$0$] at (-6,11) {};
\node [label=left:${n}$] at (-6,9) {};
\node [label=left:$\overline{{n-1}}$] at (0,7) {};
\node [label=left:$\overline{1}$] at (0,3) {};
\node [rectangle,draw,scale=1,style=densely dashed] at (1,20.7) {};
\node [rectangle,draw,scale=1,style=densely dashed] at (5,20.7) {};
\node [rectangle,draw,scale=1,style=densely dashed] at (7,20.7) {};
\node [rectangle,draw,scale=1,style=densely dashed] at (9,11) {};
\node at (0,17) {$\vdots$};
\node at (0,5) {$\vdots$};
\node at (3,20) {$\cdots$};

\draw [>-] (-6,13) .. controls (-5,13) and (-5,11) .. (-4,11);
\draw [>-] (-4,13) -- (-2,13);
\draw [>-] (-6,11) .. controls (-5,11) and (-5,13) .. (-4,13);
\draw [>-] (-4,11) .. controls (-3,11) and (-3,9) .. (-2,9);
\draw [>-] (-6,9) -- (-4,9);
\draw [>-] (-4,9) .. controls (-3,9) and (-3,11) .. (-2,11);
\draw [>->] (-2,13) .. controls (-1,13) and (-1,11) .. (0,11);
\draw [>->] (-2,11) .. controls (-1,11) and (-1,13) .. (0,13);
\draw [>->] (-2,9) -- (0,9);

\draw [>-] (0,19) -- (9,19);
\draw [>-] (0,15) -- (9,15);
\draw [-] (0,13) -- (9,13);
\draw [-] (0,11) -- (8.25,11);
\draw [-] (0,9) -- (9,9);
\draw [>-] (0,7) -- (9,7);
\draw [>-] (0,3) -- (9,3);

\draw [->] (1,20) -- (1,2);
\draw [->] (5,20) -- (5,2);
\draw [->] (7,20) -- (7,2);
\draw [densely dashed,->] (9,10.25) -- (9,2);

\draw [-]
	(9,19) arc (90:0:8);
\draw [*-]
	(17,11) arc (0:-90:8);

\draw [-]
	(9,15) arc (90:0:4);
\draw [*-]
	(13,11) arc (0:-90:4);

\draw [-]
	(9,13) arc (90:0:2);
\draw [*-]
	(11,11) arc (0:-90:2);

\end{tikzpicture}
\end{centering}
\end{minipage}
\right)
$$
  
The caduceus relation (Lemma \ref{le:caduceus}) allows us to pass this braid through without disturbing the partition function, until eventually we have
$$
\mathcal{Z}\left(
  \begin{aligned}
  \begin{tikzpicture}[scale=0.22]

\node [label=left:$1$] at (0,19) {};
\node [label=left:${n-1}$] at (0,15) {};
\node [label=left:$\overline{{n}}$] at (-6,13) {};
\node [label=left:$0$] at (-6,11) {};
\node [label=left:${n}$] at (-6,9) {};
\node [label=left:$\overline{{n-1}}$] at (0,7) {};
\node [label=left:$\overline{1}$] at (0,3) {};
\node [rectangle,draw,scale=1,style=densely dashed] at (1,20.7) {};
\node [rectangle,draw,scale=1,style=densely dashed] at (5,20.7) {};
\node [rectangle,draw,scale=1,style=densely dashed] at (7,20.7) {};

\node [rectangle,draw,scale=1,style=densely dashed] at (9,11) {};
\node at (0,17) {$\vdots$};
\node at (0,5) {$\vdots$};
\node at (3,20) {$\cdots$};

\draw [>-] (-6,13) .. controls (-5,13) and (-5,11) .. (-4,11);
\draw [>-] (-4,13) -- (-2,13);
\draw [>-] (-6,11) .. controls (-5,11) and (-5,13) .. (-4,13);
\draw [>-] (-4,11) .. controls (-3,11) and (-3,9) .. (-2,9);
\draw [>-] (-6,9) -- (-4,9);
\draw [>-] (-4,9) .. controls (-3,9) and (-3,11) .. (-2,11);
\draw [>->] (-2,13) .. controls (-1,13) and (-1,11) .. (0,11);
\draw [>->] (-2,11) .. controls (-1,11) and (-1,13) .. (0,13);
\draw [>->] (-2,9) -- (0,9);

\draw [>-] (0,19) -- (9,19);
\draw [>-] (0,15) -- (9,15);
\draw [-] (0,13) -- (9,13);
\draw [-] (0,11) -- (8.25,11);
\draw [-] (0,9) -- (9,9);
\draw [>-] (0,7) -- (9,7);
\draw [>-] (0,3) -- (9,3);

\draw [->] (1,20) -- (1,2);
\draw [->] (5,20) -- (5,2);
\draw [->] (7,20) -- (7,2);
\draw [densely dashed,->] (9,10.25) -- (9,2);

\draw [-]
	(9,19) arc (90:0:8);
\draw [*-]
	(17,11) arc (0:-90:8);

\draw [-]
	(9,15) arc (90:0:4);
\draw [*-]
	(13,11) arc (0:-90:4);

\draw [-]
	(9,13) arc (90:0:2);
\draw [*-]
	(11,11) arc (0:-90:2);

\end{tikzpicture} 
\end{aligned}
\right)
= \cdots =
\mathcal{Z}\left(
\begin{aligned}
  \begin{tikzpicture}[scale=0.22]

\node [label=left:$1$] at (0,19) {};
\node [label=left:${n-1}$] at (0,15) {};
\node [label=left:$\overline{{n}}$] at (0,13) {};
\node [label=left:$0$] at (0,11) {};
\node [label=left:${n}$] at (0,9) {};
\node [label=left:$\overline{{n-1}}$] at (0,7) {};
\node [label=left:$\overline{1}$] at (0,3) {};
\node [rectangle,draw,scale=1,style=densely dashed] at (1,20.7) {};
\node [rectangle,draw,scale=1,style=densely dashed] at (5,20.7) {};
\node [rectangle,draw,scale=1,style=densely dashed] at (7,20.7) {};
\node [rectangle,draw,scale=1,style=densely dashed] at (14.75,11) {};
\node at (0,17) {$\vdots$};
\node at (0,5) {$\vdots$};
\node at (3,20) {$\cdots$};

\draw [-] (8,13) .. controls (9,13) and (9,11) .. (10,11);
\draw [-] (10,13) -- (12,13);
\draw [-] (8,11) .. controls (9,11) and (9,13) .. (10,13);
\draw [-] (10,11) .. controls (11,11) and (11,9) .. (12,9);
\draw [-] (8,9) -- (10,9);
\draw [-] (10,9) .. controls (11,9) and (11,11) .. (12,11);
\draw [-] (12,13) .. controls (13,13) and (13,11) .. (14,11);
\draw [-] (12,11) .. controls (13,11) and (13,13) .. (14,13);
\draw [-] (12,9) -- (14.75,9);
\draw [-] (14,13) -- (14.75,13);

\draw [>-] (0,19) -- (14.75,19);
\draw [>-] (0,15) -- (14.75,15);
\draw [>-] (0,13) -- (8,13);
\draw [>-] (0,11) -- (8,11);
\draw [>-] (0,9) -- (8,9);
\draw [>-] (0,7) -- (14.75,7);
\draw [>-] (0,3) -- (14.75,3);

\draw [->] (1,20) -- (1,2);
\draw [->] (5,20) -- (5,2);
\draw [->] (7,20) -- (7,2);
\draw [densely dashed,->] (14.75,10.25) -- (14.75,2);

\draw [-]
	(14.75,19) arc (90:0:8);
\draw [*-]
	(22.75,11) arc (0:-90:8);

\draw [-]
	(14.75,15) arc (90:0:4);
\draw [*-]
	(18.75,11) arc (0:-90:4);

\draw [-]
	(14.75,13) arc (90:0:2);
\draw [*-]
	(16.75,11) arc (0:-90:2);

\end{tikzpicture}
\end{aligned}
\right)
$$
By Lemmas \ref{le:fish.C} and \ref{le:fish.B*}, 
this latter partition function is equal to
\begin{align*}
\left(a_1^{(n)} - i~b_2^{(n)}\right)\left(a_2^{(n)} + i~b_1^{(n)}\right)\left(a_1^{(n)}a^{(0)} + b^{(0)} b_2^{(n)}\right)\left(a^{(0)}a_2^{(n)}+b_1^{(n)}b^{(0)}\right) &\widehat{\mathcal{Z}}(\mathfrak{C}^\lambda)\\
\left(a^{(0)}a_2^{(n)}+b^{(0)}b_1^{(n)}\right)\left(a^{(0)}a_1^{(n)}+b^{(0)}b_2^{(n)}\right)\left(\left(a_1^{(n)}\right)^2+\left(b_2^{(n)}\right)^2\right)\widehat{\mathcal{Z}}(\mathfrak{B}_*^\lambda)\\
\end{align*}
 where for each $\mathfrak{M} \in \{\mathfrak{C},\mathfrak{B}_*\}$ the expression $\widehat{\mathcal{Z}}(\mathfrak{M}^\lambda)$ represents the partition function for the $\mathfrak{M}^\lambda$ family, but with spectral indices $n$ and $\overline{n}$ exchanged. 
 
 Setting these equal to the appropriate $$\left(a_1^{(n)}a^{(0)} + b^{(0)} b_2^{(n)}\right)^2\left(\left(a_1^{(n)}\right)^2+\left(b_2^{(n)}\right)^2\right)\mathcal{Z}(\mathfrak{M}^\lambda)$$ and canceling like factors then gives
 \begin{align*}
\left(a_1^{(n)}a^{(0)} + b^{(0)} b_2^{(n)}\right)\left(a_1^{(n)} + i~b_2^{(n)}\right)\mathcal{Z}(\mathfrak{C}^\lambda)  &= \left(a_2^{(n)} + i~b_1^{(n)}\right)\left(a^{(0)}a_2^{(n)}+b_1^{(n)}b^{(0)}\right) \widehat{\mathcal{Z}}(\mathfrak{C}^\lambda)\\
\left(a_1^{(n)}a^{(0)} + b^{(0)} b_2^{(n)}\right)\mathcal{Z}(\mathfrak{B}_*^\lambda) &= \left(a^{(0)}a_2^{(n)}+b_1^{(n)}b^{(0)}\right) \widehat{\mathcal{Z}}(\mathfrak{B}_*^\lambda).\\
\end{align*} The result follows by observing that each factor on the left hand side transforms to a factor on the right hand side under the appropriate inversion.

The same argument works when $\mathfrak{M}^\lambda = \mathfrak{BC}^\lambda$ if one makes the appropriate spectral index substitutions and uses Lemma \ref{le:fish.X} in place of Lemmas \ref{le:fish.C} and \ref{le:fish.B*}.
\end{proof}

With these symmetry properties of the various partition functions proven, we give a lemma that tells us how to extract certain obvious factors of a symmetric polynomial.  

\begin{lemma}\label{le:finding.factors}
Suppose $p \in \Z[i,a^{(0)},b^{(0)},a_1^{(1)},a_2^{(1)},\cdots,b_1^{(n)},b_2^{(n)}]$ is a polynomial that is invariant under the spectral index inversion $n \leftrightarrow \bar n$, as well as under the spectral index transposition $\ell \leftrightarrow \ell+1$ for each $1 \leq \ell <n$.  Then we have the following implications:
\begin{align*}
\left(a_1^{(1)}a_2^{(2)}+b_1^{(2)}b_2^{(1)}\right) \mid p \Rightarrow & \prod_{j \neq k} \left(a_1^{(j)}a_2^{(k)}+b_1^{(k)}b_2^{(j)}\right) \\ & \qquad \prod_{j<k}\left(a_1^{(j)}a_1^{(k)}+b_2^{(k)}b_2^{(j)}\right)\left(a_2^{(j)}a_2^{(k)}+b_1^{(k)}b_1^{(j)}\right) \mid p \\[10pt]
\left(a_1^{(n)}+ i~b_2^{(n)}\right) \mid p \Rightarrow & \prod_{j=1}^n \left(a_1^{(j)}+ i~b_2^{(j)}\right)\left(a_2^{(j)} + i~b_1^{(j)}\right) \mid p\\[10pt]
\left(\left(a_1^{(n)}\right)^2 + \left(b_2^{(n)}\right)^2\right) \mid p \Rightarrow & \prod_{j=1}^n \left(\left(a_1^{(j)}\right)^2 + \left(b_2^{(j)}\right)^2\right)\left(\left(a_2^{(j)}\right)^2+\left(b_1^{(j)}\right)^2\right) \mid p\\[10pt]
\left(a_1^{(n)} a^{(0)} + b^{(0)}a_2^{(n)}\right) \mid p \Rightarrow & \prod_{j=1}^n \left(a_1^{(j)} a^{(0)}+b^{(0)}b_2^{(j)}\right)\left(a^{(0)}a_2^{(j)}+b_1^{(j)}b^{(0)}\right) \mid p
\end{align*}
\end{lemma}

\begin{proof}
Certainly if $p$ is invariant under the stated exchanges, then it is also invariant under $j \leftrightarrow k$ as well as $j \leftrightarrow \overline{k}$ for any $j,k$.  

If $j$ and $k$ are given with $j \neq k$, then the exchange $1 \leftrightarrow j, 2 \leftrightarrow k$ sends the factor $\left(a_1^{(1)}a_2^{(2)}+b_1^{(2)}b_2^{(1)}\right)$ to $\left(a_1^{(j)}a_2^{(k)}+b_1^{(k)}b_2^{(j)}\right)$.  The exchange $1 \leftrightarrow j, 2 \leftrightarrow \overline{k}$ sends this factor to $\left(a_1^{(j)}a_1^{(k)}+b_2^{(k)}b_2^{(j)}\right)$, and the exchange $1 \leftrightarrow \overline{j}, 2 \leftrightarrow k$ sends this factor to $ \left(a_2^{(j)}a_2^{(k)} + b_1^{(k)}b_1^{(j)}\right)$.

For a given $j$, the exchange $n \leftrightarrow j$ sends the factor $\left(a_1^{(n)}+ i~b_2^{(n)}\right)$ to $\left(a_1^{(j)}+ i~b_2^{(j)}\right)$, and hence this term must also divide $p$.  Likewise the exchange $n \leftrightarrow \overline{j}$ sends the factor $\left(a_1^{(n)}+ i~b_2^{(n)}\right)$ to $\left(a_2^{(j)}+ i~b_1^{(j)}\right),$ so this must appear in the factorization of $p$ as well.

The exchanges $n \leftrightarrow j$ and $n \leftrightarrow \overline{j}$ take $\left(\left(a_1^{(n)}\right)^2+\left(b_2^{(n)}\right)^2\right)$ to $\left(\left(a_1^{(j)}\right)^2+\left(b_2^{(j)}\right)^2\right)$ and $\left(\left(a_2^{(j)}\right)^2+\left(b_1^{(j)}\right)^2\right)$, respectively.

Finally, the exchange $n \leftrightarrow j$ and $n \leftrightarrow \overline{j}$ send $\left(a_1^{(n)}a^{(0)}+b^{(0)}a_2^{(n)}\right)$ to $\left(a_1^{(j)}a^{(0)}+b^{(0)}a_2^{(j)}\right)$ and $\left(a^{(0)}a_2^{(j)}+b_1^{(j)}b^{(0)}\right)$.
\end{proof}

\begin{theorem}\label{prop:divisibility.for.partition.functions}
When using the bending weights from Table \ref{tab:summarized.bending.weights}, we have the following divisibilities: 
\begin{align*}
&\prod_{j\leq n} \left(a_2^{(j)} + i~b_1^{(j)}\right)\prod_{j<k\leq n} \left(a_1^{(k)}a_2^{(j)}+b_1^{(j)}b_2^{(k)}\right)\left(a_2^{(j)}a_2^{(k)}+b_1^{(k)}b_1^{(j)}\right) \mid \mathcal{Z}(\mathfrak{B}^\lambda)\\[10pt]
&\prod_{j \leq n} \left(a^{(0)}a_2^{(j)}+b_1^{(j)}b^{(0)}\right) \prod_{j<k\leq n} \left(a_1^{(k)}a_2^{(j)}+b_1^{(j)}b_2^{(k)}\right)\left(a_2^{(j)}a_2^{(k)}+b_1^{(k)}b_1^{(j)}\right) \mid \mathcal{Z}(\mathfrak{B}_*^\lambda)\\[10pt]
&\prod_{j \leq n}\left(a_2^{(j)} + i~b_1^{(j)}\right)\left(a^{(0)}a_2^{(j)}+b_1^{(j)}b^{(0)}\right)\prod_{j<k\leq n}\left(a_1^{(k)}a_2^{(j)}+b_1^{(j)}b_2^{(k)}\right)\left(a_2^{(j)}a_2^{(k)}+b_1^{(k)}b_1^{(j)}\right) \mid \mathcal{Z}(\mathfrak{C}^\lambda)\\[10pt]
&\prod_{j \leq n} \left(\left(a_2^{(j)}\right)^2 + \left(b_1^{(j)}\right)^2\right) \prod_{j<k\leq n} \left(a_1^{(k)}a_2^{(j)}+b_1^{(j)}b_2^{(k)}\right) \left(a_2^{(j)}a_2^{(k)} + b_1^{(k)}b_1^{(j)}\right) \mid \mathcal{Z}(\mathfrak{C}_*^\lambda)\\[10pt]
&\prod_{j<k\leq n} \left(a_1^{(k)}a_2^{(j)}+b_1^{(j)}b_2^{(k)}\right)\left(a_2^{(j)}a_2^{(k)}+b_1^{(k)}b_1^{(j)}\right) \mid \mathcal{Z}(\mathfrak{D}^\lambda) \qquad (\text{if } 1 \in \lambda) \\[10pt]
&\prod_{j\leq n} \left(\left(a_2^{(j)}\right)^2 + \left(b_1^{(j)}\right)^2\right)\prod_{j<k\leq n} \left(a_1^{(k)}a_2^{(j)}+b_1^{(j)}b_2^{(k)}\right)\left(a_2^{(j)}a_2^{(k)}+b_1^{(k)}b_1^{(j)}\right) \mid \mathcal{Z}(\mathfrak{D}^\lambda) \qquad (\text{if } 1 \not\in \lambda) \\[10pt]
&\prod_{j<n} \left(a^{(n)}a_2^{(j)}+b_1^{(j)}b^{(n)}\right) \mkern-5mu \left( \mkern-4mu \left(a_2^{(j)}\right)^{\mkern-2mu 2} \mkern-5mu +\left(b_1^{(j)}\right)^{\mkern-2mu 2} \right) \mkern-7mu \prod_{j<k<n} \mkern-6mu \left(a_1^{(k)}a_2^{(j)}+b_1^{(j)}b_2^{(k)}\right) \mkern-5mu \left(a_2^{(j)}a_2^{(k)}+b_1^{(k)}b_1^{(j)}\right) \mid \mathcal{Z}(\mathfrak{BC}^\lambda)\\[10pt]
\end{align*}
In each case, the quotient is a polynomial that is symmetric under the exchange of spectral indices $j$ and $k$, and is also symmetric under the replacement of $j$ by $\overline{j}$.  
\end{theorem}

\begin{proof}
We will give a complete proof for the result in the $\mathfrak{B}^\lambda$ family; the proofs for the other families are essentially the same, with only minor variations on this same theme. 

Consider the expression
\begin{align}
&\prod_j\left(a_1^{(j)}+i~b_2^{(j)}\right) \prod_{j<k}\left(a_1^{(j)}a_2^{(k)}+b_1^{(k)}b_2^{(j)}\right)\left(a_1^{(j)}a_1^{(k)}+b_2^{(k)}b_2^{(j)}\right) \mathcal{Z}(\mathfrak{B}^\lambda).
\end{align}\label{eq:symmetric.function.type.B}  
By Lemma \ref{le:finding.factors}, if we can show this polynomial is invariant under $n \leftrightarrow \overline{n}$ as well as $\ell \leftrightarrow {\ell+1}$ for each $\ell$, then the desired result will follow. 

Lemma \ref{le:transposition.invariance} tells us that the term $\left(a_1^{(\ell)}a_2^{(\ell+1)}+b_1^{(\ell+1)}b_2^{(\ell)}\right)\mathcal{Z}(\mathfrak{B}^\lambda)$ is invariant under $\ell \leftrightarrow \ell+1$.
The remaining factors of (\ref{eq:symmetric.function.type.B}) are permuted under the exchange $\ell \leftrightarrow \ell+1$, so that the above polynomials is invariant under any adjacent transposition.   

Lemma \ref{le:fish.invariance} tells us that $\left(a_1^{(n)}+i~b_2^{(n)}\right)\mathcal{Z}(\mathfrak{B}^\lambda)$ is invariant under $n \leftrightarrow \overline{n}$.  We note that the remaining factors of (\ref{eq:symmetric.function.type.B}) are permuted under $n \leftrightarrow \overline{n}$: terms that don't involve $n$ are clearly fixed, whereas the terms
$\left(a_1^{(j)}a_2^{(n)}+b_1^{(n)}b_2^{(j)}\right)\left(a_1^{(j)}a_1^{(n)}+b_2^{(n)}b_2^{(j)}\right)$ have their factors permuted by this exchange.  Hence our polynomial is fixed by $n \leftrightarrow \overline{n}$, as desired.


To see that the quotient has the appropriate symmetry, notice that it is equal to
\[ \scalebox{.9}{$\frac{\displaystyle \prod_j\left(a_1^{(j)}+i~b_2^{(j)}\right) \prod_{j<k}\left(a_1^{(j)}a_2^{(k)}+b_1^{(k)}b_2^{(j)}\right)\left(a_1^{(j)}a_1^{(k)}+b_2^{(k)}b_2^{(j)}\right) \mathcal{Z}(\mathfrak{B}^\lambda)}{\displaystyle \prod_{j} \left(a_2^{(j)} + i~b_1^{(j)}\right)\left(a_1^{(j)}+i~b_2^{(j)}\right)\prod_{j\neq k} \left(a_1^{(k)}a_2^{(j)}+b_1^{(j)}b_2^{(k)}\right)\prod_{j<k}\left(a_2^{(j)}a_2^{(k)}+b_1^{(k)}b_1^{(j)}\right)\left(a_1^{(j)}a_1^{(j)}+b_2^{(k)}b_2^{(j)}\right)}$}. \]
Both numerator and denominator are invariant under the appropriate spectral index permutations, and hence so too is the quotient.  
\end{proof}

\begin{proposition}\label{prop:partition.functions.for.rho}
Let $\mathfrak{M} \in \{\mathfrak{B},\mathfrak{B}_*,\mathfrak{C},\mathfrak{C}_*,\mathfrak{D},\mathfrak{BC}\}$ be given.  When using the bending weights from Table \ref{tab:summarized.bending.weights}, the divisibility statements for $\mathcal{Z}(\mathfrak{M}^\rho)$ from Theorem \ref{prop:divisibility.for.partition.functions} are equalities.
\end{proposition}

\begin{proof}
We will prove this result for the $\mathfrak{B}^\rho$ model and leave the other proofs to the interested reader.  First, we will prove that each summand in the partition function $\mathcal{Z}(\mathfrak{B}^\rho)$ is a degree $2n^2-n$ polynomial in the variables $a_1^{(j)},a_2^{(j)},b_1^{(j)}$ and $b_2^{(j)}$. To see this, note that for any $A \in \mathfrak{B}^\lambda$ has $n$ more $c_2$ entries than $c_1$ entries (from the perspective of the corresponding alternating sign matrix, the sum of entries in each row must be $+1$, and $c_2$ vertices in our bent ice correspond to $+1$ entries of the ASM, while $c_1$ vertices corresponding to $-1$ entries).  Since any decoration from $\{a_1,a_2,b_1,b_2\}$ has degree $1$ and any vertex decorated by $c_1$ contributes a factor of degree $2$, we have 
$$\deg \left(\text{wt}(A)\right) = \left(\#\mbox{vertices}-(\# c_2 \mbox{ entries}) - (\#c_1 \mbox{ entries})\right) + 2(\# c_1 \mbox{ entries}) = \#\mbox{vertices} - n.$$
Now one simply counts that the $\mathfrak{B}^\rho$ model has $2n^2$ vertices.  


Since the known factor from Theorem \ref{prop:divisibility.for.partition.functions} has the correct degree, it must be that the partition function is a constant multiple of this factor.  To recover this constant, expand the product
$$
\prod_{j\leq n} \left(a_2^{(j)} + i~b_1^{(j)}\right)\prod_{j<k\leq n} \left(a_1^{(k)}a_2^{(j)}+b_1^{(j)}b_2^{(k)}\right)\left(a_2^{(j)}a_2^{(k)}+b_1^{(k)}b_1^{(j)}\right)
$$ and note that the coefficient of $\prod_j \left(a_2^{(j)}\right)^{2n-j} \left(a_1^{(j)}\right)^{j-1}$ is $1$.  We argue that there is only one element of $\mathfrak{B}^\rho$ which has this weight, from which it will follow that the two polynomials are equal.

Consider the configuration whose only $c_2$ decorations are at column $j$ of row $j$ for each $1\leq j \leq n$, and which has no $c_1$ decorations.  (The remaining decorations can be filled in inductively).  This configuration clearly has weight $\prod_j \left(a_2^{(j)}\right)^{2n-j} \left(a_1^{(j)}\right)^{j-1}$.  Now suppose some other configuration has this same weight.  We claim this is enough to determine the vertex decorations (and their placements) in the first row pair.  For this, note that in order for there to be no $a_1^{(1)}$ contribution, there cannot be any $c_1^{(1)}$ or $c_1^{(\bar 1)}$ weights, and so there is only a single $c_2^{(1)}$ or $c_2^{(\bar 1)}$ configuration.  If it were the case that there were a $c_2^{(\bar 1)}$ vertex, then the orientation along the horizontal of row $1$ wouldn't change, and hence we would have a contribution of $\left(b_1^{(1)}\right)^n$ to the weight, a contradiction.  Hence there is a single $c_2^{(1)}$ vertex; if it is anywhere but column $1$, then this would force the top left entry to have weight $b_1^{(1)}$, again a contradiction.  Hence the $c_2^{(1)}$ entry must be in the first column, with remaining vertices in row $1$ decorated $a_2^{(1)}$ and vertices in $\overline{1}$ decorated $a_1^{(\overline{1})} = a_2^{(1)}$.  A similar argument for subsequent rows gives the desired result.
\end{proof}

This is essentially the maximum amount of information about the general partition function which is extractable
from the Yang-Baxter equation. In subsequent sections, we use combinatorial methods to say more about certain
specializations of these weights.

\section{Recovering Okada's and Simpson's Weyl denominator deformations\label{okadamatch}}

In \cite{okada}, Okada defines four families of ASM or ``ASM-like" matrices that he calls $B_n, C_n, C_n'$ and $D_n$; these correspond to admissible states of our ice models $\mathfrak{B}^\rho, \mathfrak{C}^\rho, \mathfrak{C}_*^\rho$ and $\mathfrak{D}^\rho$ under the bijection discussed at the end of Section~\ref{classyice}. Here, as usual, $\rho = [n,n-1,\cdots,1]$.  Later in \cite{simpson}, Simpson introduced another family called $B_n'$, and these matrices correspond to admissible states of the ice model $\mathfrak{B}_*^\rho$. Simpson remarks that matrix transposition gives a bijection between elements of $B_n'$ and $C_n'$.  Likewise, matrix transposition gives a bijection between the half-turn symmetric matrices which correspond to our models $\mathfrak{D}^\rho$ and $\mathfrak{BC}^\rho$.  The matrices corresponding to $\mathfrak{BC}^\rho$ do not appear to have been previously studied in the literature.

In both \cite{okada} and \cite{simpson}, weights are assigned to these matrices based on certain global statistics (i.e., calculated by analyzing relations between entries across the entire matrix). The resulting generating function summed over each family of matrices is shown to be a deformation of the Weyl denominator for the respective classical Lie group as follows:
\begin{equation}\label{eq:okada.deformations}
\begin{split}
\mathcal{Z}(B_n):&=\sum_{\widehat{A} \in B_n} \text{wt}(\widehat{A}) = \prod_{j=1}^n (1-t x_j) \prod_{j<k}(1-t^2x_jx_k^{-1})(1-t^2x_jx_k)\\
\mathcal{Z}(B'_n):&=\sum_{\widehat{A} \in B'_n} \text{wt}(\widehat{A}) = \prod_{j=1}^n (1+ t x_j) \prod_{j<k}(1+t x_jx_k^{-1})(1+t x_jx_k)\\
\mathcal{Z}(C_n):&=\sum_{\widehat{A} \in C_n} \text{wt}(\widehat{A}) = \prod_{j=1}^n (1-t x_j)(1+t^2 x_j) \prod_{j<k}(1-t^2x_jx_k^{-1})(1-t^2x_jx_k)\\
\mathcal{Z}(C'_n):&=\sum_{\widehat{A} \in C'_n} \text{wt}(\widehat{A}) = \prod_{j=1}^n (1+t x_j^2) \prod_{j<k}(1+t x_jx_k^{-1})(1+t x_jx_k)\\
\mathcal{Z}(D_n):&=\sum_{\widehat{A} \in D_n} \text{wt}(\widehat{A}) = \prod_{j<k}(1+t x_jx_k^{-1})(1+t x_jx_k).
\end{split}
\end{equation}
(The factorization of $\mathcal{Z}(C'_n)$ in \cite{okada} contains a small typo that we have corrected here.) 

Note that by specializing $t$ (to $\pm 1$, depending on the family), one recovers the Weyl denominator of the corresponding Lie group.  The function ``$\text{wt}$'' on the left-hand side varies somewhat for each case, and will be recalled shortly for the $B_n$ identity. In this section, we show that this ad-hoc collection of functions has an orderly explanation in the context of lattice models. More precisely, the ``$\text{wt}$'' function in each case is realized by a uniform choice of Boltzmann weights under the natural bijection of ASM-like matrices and ice given at the end of Section~\ref{classyice}. Better still, the natural geometry of lattice models and their boundary conditions leads us to a previously undiscovered generating function identity for the deformed denominator of type $BC$ appearing in the character formula of Proctor. All of these denominator identities will be further generalized to deformed Weyl character formula identities in Section~\ref{specwcf}.

The particular choice of free-fermionic Boltzmann weights that achieves these deformations and a few related weighting schemes are presented in the next definition. For linguistic convenience, we refer to the weights which recover Okada's results as the ``Okada weights'' for the ice models; in the family $\mathfrak{B}_*$ the Okada weights recover Simpson's results, and in the family $\mathfrak{BC}$ they return the aforementioned deformed denominator identity for type $BC$.

\begin{definition*}
Let $\star = 0$ in the $\mathfrak{B}_*$ and $\mathfrak{C}$ families, and $\star = n$ in the $\mathfrak{BC}$ family.  Then the deformation weights are
\begin{gather*}
a_1^{(j)} = a_2^{(j)} = 1, \quad b_1^{(j)} = i t_jx_j, \quad b_2^{(j)}=i t_jx_j^{-1}, \quad c_1^{(j)}=1-t_j^2, \quad c_2^{(j)}=1\\
a^{(\star)}=1, \quad b^{(\star)}=it_\star x_\star, \quad c_1^{(\star)} = 1-t_\star^2x_\star^2, \quad c_2^{(\star)}=1
\end{gather*}
together with the bending weights from Table \ref{tab:summarized.bending.weights}. 
The Okada weights are the specialization of the deformation weights where we assign $x_0=-1$ in the $\mathfrak{C}^\lambda$ model, $x_0=1$ in the $\mathfrak{B}_*^\lambda$ model, and $x_n = 1$ in the $\mathfrak{BC}^\lambda$ model, and we let 
$$t_j = \left\{\begin{array}{ll}
t,&\text{ for models }\mathfrak{B}^\lambda,\mathfrak{C}^\lambda\\
i~\root\of{t},&\text{ for models }\mathfrak{B}_*^\lambda, \mathfrak{C}_*^\lambda,\mathfrak{D}^\lambda, \mathfrak{BC}^\lambda.
\end{array}\right.$$
\end{definition*}

By using Theorem \ref{prop:divisibility.for.partition.functions} and Proposition \ref{prop:partition.functions.for.rho}, the deformation weights give a slightly more general formulation of the identities in (\ref{eq:okada.deformations}), allowing for $t_j$'s to depend on the spectral index $j$ and adding an $x$-parameter to the central row for  $\mathfrak{B}_*^\lambda$ and  $\mathfrak{C}^\lambda$.

\begin{corollary}\label{cor:weyl.denominator}
Using the deformation weights, we have
\begin{align*}
&\prod_{j\leq n} (1-t_jx_j) \prod_{j<k\leq n} \left(1-t_jt_kx_jx_k^{-1}\right)\left(1-t_jt_kx_jx_k\right)  ~\mathlarger{\mathlarger{\mathlarger{\mid}}} ~\mathcal{Z}(\mathfrak{B}^\lambda)\\[10pt]
&\prod_{j\leq n} (1-t_0t_jx_0x_j) \prod_{j<k\leq n} \left(1-t_jt_kx_jx_k^{-1}\right)\left(1-t_jt_kx_jx_k\right)  ~\mathlarger{\mathlarger{\mathlarger{\mid}}} ~\mathcal{Z}(\mathfrak{B}_*^\lambda)\\[10pt]
&\prod_{j\leq n} (1-t_jx_j)\left(1-t_jt_0x_0x_j\right) \prod_{j<k\leq n} \left(1-t_jt_kx_jx_k^{-1}\right)\left(1-t_jt_kx_jx_k\right)  ~\mathlarger{\mathlarger{\mathlarger{\mid}}} ~\mathcal{Z}(\mathfrak{C}^\lambda)\\[10pt]
&\prod_{j\leq n} (1-t_j^2x_j^2) \prod_{j<k\leq n} \left(1-t_jt_kx_jx_k^{-1}\right)\left(1-t_jt_kx_jx_k\right)  ~\mathlarger{\mathlarger{\mathlarger{\mid}}} ~\mathcal{Z}(\mathfrak{C}_*^\lambda)\\[10pt]
&\prod_{j<k\leq n} \left(1-t_jt_kx_jx_k^{-1}\right)\left(1-t_jt_kx_jx_k\right) ~\mathlarger{\mathlarger{\mathlarger{\mid}}} ~ \mathcal{Z}(\mathfrak{D}^\lambda) &\text{(if $1 \in \lambda$)}\\[10pt]
&\prod_{j\leq n} \left(1-t_j^2 x_j^2\right) \prod_{j<k\leq n}\left(1-t_jt_kx_jx_k^{-1}\right)\left(1-t_jt_kx_jx_k\right)  ~\mathlarger{\mathlarger{\mathlarger{\mid}}} ~ \mathcal{Z}(\mathfrak{D}^\lambda) &\text{(if $1 \not\in \lambda$)}\\[10pt]
&\prod_{j< n} (1-t_nt_jx_nx_j)(1-t_j^2x_j^2) \prod_{j<k< n} \left(1-t_jt_kx_jx_k^{-1}\right)\left(1-t_jt_kx_jx_k\right)  ~\mathlarger{\mathlarger{\mathlarger{\mid}}} ~\mathcal{Z}(\mathfrak{BC}^\lambda)\\[10pt]
\end{align*}
Furthermore, these divisibilities are equalities when $\lambda = \rho$.
\end{corollary}

It is clear that when we use the Okada weights, the above expressions recover the partition functions for the classes of ASMs studied by Okada and Simpson.  Moreover, the natural bijection of Section~\ref{classyice} between ASMs and ice is weight preserving, as demonstrated by the following result.

\begin{proposition}\label{prop:okada.bijection}
Let $A \in \{\mathfrak{B}^\rho,\mathfrak{C}^\rho,\mathfrak{C}_*^\rho,\mathfrak{D}^\rho\}$ (resp., $A \in \mathfrak{B}_*^\rho$), and let $\widehat{A}$ be the corresponding half-turn symmetric matrix.  Let $\text{wt}(A)$ be the weight of $A$ as calculated by the Okada weights defined above, and let $\text{wt}(\widehat{A})$ be the weight associated to $\widehat{A}$ in \cite{okada} (resp., in \cite{simpson}).  Then $\text{wt}(\widehat{A}) = \text{wt}(A)$.
\end{proposition}

Before proving this result, we will need some preparatory definitions and lemmas.  For simplicity, we restrict our attention to the case of type $B$; the others follow in a similar way.  As above, $\widehat{A} = (\hat{a}_{ij})$ denotes the $2n \times 2n$ half-turn symmetric alternating sign matrix associated to $A \in \mathfrak{B}^\rho$.  We write $i(\widehat{A})$ for the inversion number of $\widehat{A}$: $i(\widehat{A}) := \sum_{i<k,j>l} \hat a_{ij} \hat a_{kl}$.  The total number of $-1$ entries of $\widehat{A}$ will be denoted $s(\widehat{A})$.  The terms $i_1^+(\widehat{A})$ and $i_1^-(\widehat{A})$ denote the number of $+1$ and $-1$ entries in the top right quartile of $\widehat{A}$.  Let $i_1(\widehat{A}) = i_1^+(\widehat{A}) + i_1^-(\widehat{A})$, and $i_2(\widehat{A}) = i(\widehat{A}) - i_1(\widehat{A})$.  Finally, define the vector $$\delta(B_n) = [n-1/2, n-3/2, \ldots, 1/2, -1/2, \ldots, -(n-3/2), -(n-1/2)], $$
and let $\boldsymbol{x}^\alpha$ for $\alpha = (\alpha_1, \ldots, \alpha_n, -\alpha_n, \ldots, -\alpha_1)$ be defined as $x_1^{\alpha_1} \cdots x_n^{\alpha_n}$.

With this notation in hand, Okada defines
\begin{equation}\label{eq:okadas.weight.for.matrix}
\begin{split}
\text{wt}(\widehat{A}) &= (-1)^{i_1^+(\widehat{A}) + \frac{i_2(\widehat{A})}{2}} t^{i(\widehat{A})} \left(1-\frac{1}{t^2}\right)^{s(\widehat{A})/2} \boldsymbol{x}^{\delta(B_n)-A\delta(B_n)}\\
&= (-1)^{i_1^+(\widehat{A}) + \frac{i_2(\widehat{A})-s(\widehat{A})}{2}} t^{i(\widehat{A})-s(\widehat{A})} \left(1-t^2\right)^{s(\widehat{A})/2} \boldsymbol{x}^{\delta(B_n)-A\delta(B_n)}.
\end{split}
\end{equation}

Before showing $\text{wt}(\widehat{A})=\text{wt(A)}$ we establish a few lemmas.

\begin{lemma}\label{le:minus.one.entries}
The number of $c_1$ vertices in $A$ is $s(\widehat{A})/2$.
\end{lemma}
\begin{proof}
We know that $c_1$ vertices in $A$ correspond precisely to $-1$ entries in the left half of $\widehat{A}$, and that $-1$ entires on the right half of $\widehat{A}$  arise from half-turn symmetry.
\end{proof}

\begin{lemma}\label{le:statistic.on.b.vertices}
The total number of $b_1$ and $b_2$ vertices in $A$ is equal to $i(\widehat{A}) -s(\widehat{A})$.
\end{lemma}

\begin{proof}
Suppose we are at the vertex in row $r$ and column $c$ of $A \in \mathfrak{B}^\rho$ for some $1 \leq c \leq n$, and that this vertex is decorated $b_1$.  In $\widehat{A}$ this means there are as many entries equal to $-1$ as $+1$ above $\hat a_{rc}$, and as many entries equal to $-1$ as $+1$ to the left of $\hat a_{rc}$.  Hence we have $\sum_{1\leq i <r} \sum_{1 \leq l <c}\hat a_{ic} \hat a_{rl} = 0$.  Likewise there is one more entry equal to $+1$ than $-1$ to the right of $a_{rc}$, and one more entry equal to $+1$ than $-1$ below $\hat a_{rc}$.  Hence $\sum_{r < i \leq 2n} \sum_{c< l \leq 2n} \hat a_{ic} \hat a_{rl} = 1$.  Therefore the net contribution to the inversion number that comes from terms $\hat a_{ij} \hat a_{kl}$ where one entry sits in row $r$ and the other sits in column $c$ is $1$. Half turn symmetry tells us that the net contribution to the inversion number from pairs $\hat a_{ij} \hat a_{kl}$ where one entry sits in row $2n-r$ and the other sits in column $2n-c$ is also $1$.

In a similar way, one can show that there is a net contribution of $2$ to the inversion number of $\widehat{A}$ for every $b_1$ vertex in row $\bar r$ and column $c$ of $A$, and from every $b_2$ vertex in rows $r$ or $\bar r$ and column $c$.  Vertices decorated $c_1$ in $A$ give a net contribution of $4$ to the inversion number of $\widehat{A}$.  On the other hand, vertices decorated $a_1, a_2$ or $c_2$ in $A$ yield no contribution to the inversion number of $\widehat{A}$.

Finally, note that if we sum over all such terms then we overcount the inversion number actually by a factor of $2$: a positive contribution coming from $\hat a_{ij} \hat a_{kl}$ will appear when examining contributions from row $i$ and column $l$ as well as from row $j$ and column $k$.  Hence we have
$$2\#\{b_1 \mbox{ vertices}\} + 2\#\{b_2 \mbox { vertices}\} + 4\#\{c_1 \mbox{ vertices}\} = 2i(\widehat{A}).$$  Now we apply Lemma \ref{le:minus.one.entries}.
\end{proof}

\begin{proof}[Proof of Proposition \ref{prop:okada.bijection}]
Factors of $(1-t^2)$ in $\text{wt}(A)$ come exclusively from $c_1$ vertices, so the power of $(1-t^2)$ in $\text{wt}(A)$ is given by the number of $c_1$ vertices in $A$.  This equals the power of $(1-t^2)$ in $\text{wt}(\widehat{A})$ according to Lemma \ref{le:minus.one.entries}.

The only vertex decorations which contribute a factor of $t$ to $\text{wt}(A)$ are $b_1$ and $b_2$, and from Lemma \ref{le:statistic.on.b.vertices} we know that the total number of $b_1$ and $b_2$ vertices in $A$ equals $i(\widehat{A})-s(\widehat{A})$.  Hence the powers of $t$ in $\text{wt}(\widehat{A})$ and $\text{wt}(A)$ agree as well.

To see that the signs of $\text{wt}(\widehat{A})$ and $\text{wt}(A)$ agree, note first that the vertex decorations that contribute to the sign of $\text{wt}(A)$ are precisely the $b_1, b_2$ and $D$ vertices, and that each of these contributes a power of $i$.  Hence the sign of $\text{wt}(A)$ is $-1$ raised to half the sum of the total number of $b_1$, $b_2$ and $D$ vertices in $A$.  Observe further that the vertex $D^{(j)}$ indicates that in the top right quartile of $\widehat{A}$ will have one more $+1$ entry than $-1$ entry in row $j$, whereas the vertex $U^{(j)}$ indicates that the top right quartile of $\widehat{A}$ has an equal number of $+1$ and $-1$ entries in row $j$.  Summing over all $j$, we then have
$$i_1^+(\widehat{A})-i_1^-(\widehat{A}) = \#\{D \mbox{ vertices}\}.$$  Recalling that the total numbr of $b_1$ and $b_2$ vertices is given by $i(\widehat{A})-s(\widehat{A})$ and that $i(\widehat{A})=i_1^+(\widehat{A})+i_1^-(\widehat{A})+i_2(\widehat{A})$, we recover
\begin{equation*}
\begin{split}
\frac{\#\{b_1,b_2,D \mbox{ vertices}\}}{2} 
= i_1^{+}(\widehat{A})+\frac{i_2(\widehat{A})-s(\widehat{A})}{2}.
\end{split}
\end{equation*}

It remains to show that the powers of $\boldsymbol{x} = (x_1, \ldots, x_n)$ match in $\text{wt}(A)$ and $\text{wt}(\hat{A})$ under the bijection. To do this, we ``factor'' the bijection
of ice and ASMs into several intermediate steps, translating the statistics for the exponents of $\boldsymbol{x}$ along the way. First, map an admissible state of ice $A$ to a set of
partitions $\{ \rho^{(i)} \}_{i=1}^{2n+1}$ where $\rho^{(i)}$ records the column indices of all vertical edges between rows $i-1$ and $i$ with an up arrow. Thus it is always
true that $\rho^{(1)} = \rho$ and $\rho^{(2n+1)}=\varnothing$, the empty partition. For example, in $\mathcal{B}^{[2,1]}$, the correspondence is illustrated in the left-most
bijection below:
\begin{figure}
\centering 
\begin{subfigure}[!ht]{.25\textwidth}
\centering
\begin{tikzpicture}[scale=.75]
\node [label=left:$1$] at (0,2) {};
\node [label=left:$2$] at (0,1) {};
\node [label=left:$\overline{2}$] at (0,0) {};
\node [label=left:$\overline{1}$] at (0,-1) {};

\node [label=above:$2$] at (0.5,2.5) {};
\node [label=above:$1$] at (1.5,2.5) {};

\draw [>-] (0,2) -- (1,2);
\draw [>-] (0,1) -- (1,1);
\draw [>-] (0,0) -- (1,0);
\draw [>-] (0,-1) -- (1,-1);

\draw [<-] (0.5,2.5) -- (0.5,1.5);
\draw [<-] (0.5,1.5) -- (0.5,.5);
\draw [>-] (0.5,0.5) -- (0.5,-0.5);
\draw [>->] (0.5,-0.5) -- (0.5,-1.5);

\draw [>-] (1,2) -- (2,2);
\draw [<-] (1,1) -- (2,1);
\draw [>-] (1,0) -- (2,0);
\draw [>-] (1,-1) -- (2,-1);

\draw [<-] (1.5,2.5) -- (1.5,1.5);
\draw [>-] (1.5,1.5) -- (1.5,.5);
\draw [<-] (1.5,0.5) -- (1.5,-0.5);
\draw [>->] (1.5,-0.5) -- (1.5,-1.5);

\draw [<-]
	(2,2) arc (90:0:1.5);
\draw [*-<]
	(3.5,0.5) arc (0:-90:1.5);

\draw [>-]
	(2,1) arc (90:0:.5);
\draw [*->]
	(2.5,0.5) arc (0:-90:.5);
\end{tikzpicture}
\end{subfigure}
$\longleftrightarrow$
\begin{subfigure}[h]{.25\textwidth}
\centering
$$ \begin{array}{rc} \rho^{(1)} = & [2, 1] \\ \rho^{(2)} = & [2] \\ \rho^{(3)} = \rho^{(\overline{3})} = & [1] \\ \rho^{(4)} = \rho^{(\overline{2})} = & \varnothing \\ \rho^{(5)} = \rho^{(\overline{1})} = & \varnothing  \end{array} $$
\end{subfigure}
$\longleftrightarrow B = \begin{pmatrix} 1 & 1 \\ 1 & 0 \\ 0 & 1 \\ 0 & 0 \end{pmatrix} \longleftrightarrow C = \begin{pmatrix} 0 & 1 \\ 1 & -1 \\ 0 & 1 \\ 0 & 0 \end{pmatrix}$
\end{figure}

It is not hard to show that the adjacent partitions $\rho^{(i)}$ and $\rho^{(i+1)}$ must interleave, i.e., their parts $\rho^{(i)}_j$ must satisfy $\rho_j^{(i)} \geq \rho_{j}^{(i+1)} \geq \rho_{j+1}^{(i)}$. Moreover, the number of parts in $\rho^{(i)}$ minus the number of parts in $\rho^{(\overline{i})}$ is $n+1-i$. The sets of partitions with these restrictions, and for which $\rho^{(1)} = \rho$ and $\rho^{(2n+1)}=\varnothing$ are in bijection with states of $\mathcal{B}^\rho$. The parts of the partitions $\rho^{(i)}$ may be recorded in an $n \times 2n$ matrix $B = (b_{ij})$ by placing a 1 in position $b_{ij}$ if $j$ is a part of $\rho^{(i)}$ for $i \in [1, \ldots, 2n]$. Thus $a_{1,j}$ is always 1 for all $j$. Now, if $B(i)$ denotes the $i$-th row of the matrix $B$, form the matrix $C$ whose rows are $(B(1)-B(2), \ldots, B(2n-1) - B(2n), B(2n))$. Its entries are clearly $1,0,$ or $-1$. Finally, $C$ may be completed to a half-turn symmetric $2n \times 2n$ alternating sign matrix, and it is a simple exercise to verify this is $\hat{A}$ under the bijection referred to in Section~\ref{classyice}, taking $+1, -1$ in $\hat{A}$ to NS and EW vertices, respectively, in the state of ice $A$.

Recall that the power of $\boldsymbol{x}$ appearing in $\text{wt}(\hat{A})$ is $\boldsymbol{x}^{\delta(B_n) - \hat{A} \delta(B_n)}$. But if $C(i)$ denotes the $i$-th row of the $C$ matrix corresponding to $\hat{A}$ (and similarly for the matrix $B$ with rows $B(i)$), then 
$$ \boldsymbol{x}^{\hat{A} \delta(B_n)} = \prod_{i=1} x_i^{(C(i) - C(2n+1-i)) \cdot [n-1/2, \ldots, 1/2]^T} = \prod_{i=1} x_i^{(B(i) - B(i+1) - B(2n+1-i ) + B(2n+2-i)) \cdot [n-1/2, \ldots, 1/2]^T}, $$
where we take $C(2n+1) = B(2n+1) = \varnothing$ in the above products.
Let $\ell(\lambda)$ denote the number of parts of a partition $\lambda$ and $|\lambda|$ the sum of the parts. Since $B(i) \cdot [n-1/2, \ldots, 1/2]^T = |\rho^{(i)}| - \ell(\rho^{(i)}) / 2$, then $\boldsymbol{x}^{
\delta(B_n) - \hat{A} \delta(B_n)}$ is equal to
$$ \prod_{i=1}^n x_i^{\frac{2(n-i)+1}{2} +\frac{1}{2} [ \ell (\rho^{(i)}) - \ell (\rho^{(i+1)}) + \ell(\rho^{(\overline{i})}) - \ell(\rho^{(\overline{i+1})}) ] - |\rho^{(i)}| + |\rho^{(i+1)}| - |\rho^{(\overline{i})}| + |\rho^{(\overline{i+1})}|}. $$
Thus it remains to show that the above expression is equal to the power of $\boldsymbol{x}$ in $\text{wt}(A)$ for the admissible state $A$ corresponding to $\hat{A}$. But the exponent of $x_i$ in $\text{wt}(A)$ is just ($\#$ of $b_1^{(i)}$ vertices) $-$ ($\#$ of $b_2^{(i)}$ vertices) $-$ ($\#$ of $b_1^{(\overline{i})}$ vertices) $+$ ($\#$ of $b_2^{(\overline{i})}$ vertices). The reader may easily verify that a $b_1^{(i)}$ vertex occurs precisely in column $\rho^{(i)}_j$ when $\rho^{(i)}_j = \rho^{(i+1)}_j$ and that a $b_2^{(i)}$ vertex occurs along each column strictly between $\rho^{(i)}_j$ and $\rho^{(i+1)}_j$. (When $\ell(\rho^{(i+1)}) = \ell(\rho^{(i)})-1$, then $b_2^{(i)}$'s also appear along each column with index less than the smallest part of $\rho^{(i)}$.) Similar counts apply for their barred counterparts $b_1^{(\overline{i})}$ and $b_2^{(\overline{i})}$. Combining these two facts, then the power of $x_i$ in $\text{wt}(A)$ is computed using the difference of parts and the number of parts in $\rho^{(i)}$'s and comparing to the above displayed expression gives the result.
\end{proof}


\section{Specialization to Character Formulae \label{specwcf}}

In this section we specialize the deformation weights by setting $t_j=1$ for all $j$, and by setting $x_0=-1$ in the $\mathfrak{C}^\lambda$ model,  $x_0=1$ in the $\mathfrak{B}_*^\lambda$ model, and $x_n = 1$ in the $\mathfrak{BC}^\lambda$ model.  We call the resulting weights the \emph{character weights}.  We will show that under these weights, $\mathcal{Z}(\mathfrak{M}^\lambda)$ recovers the character formula for $\mathfrak{M} \in \{\mathfrak{B},\mathfrak{B}_*,\mathfrak{C},\mathfrak{C}_*,\mathfrak{D},\mathfrak{BC}\}$. More precisely, we recover the Weyl character formula for classical types and Proctor's character formula \cite{proctor} for type $BC$.

Notice that using the character weights, a decoration of $c_1$ is weighted $0$, as is an $L$ vertex of $\mathfrak{C}^\lambda$.  Hence by using these weights, the only configurations with nonzero weights are those which have no $c_1$ vertices.  One can then show that there is a single $c_2$ vertex in each row for the models $\mathfrak{B}^\lambda,\mathfrak{C}_*^\lambda,\mathfrak{D}^\lambda, \mathfrak{BC}^\lambda$, and a single $c_2$ vertex in each nonzero row for the models $\mathfrak{B}_*^\lambda$ and $\mathfrak{C}^\lambda$, with no $c_2$ entries in the $0$ row.  

Note that without $c_1$ vertices, the only way to reverse the orientation of arrows along horizontal edges in a row is via a $c_2$ vertex, switching rightward arrows to leftward arrows. If the $j$th bending vertex is decorated $U^{(j)}$, this forces the unique $c_2$ vertex to occur in row $j$; likewise, a u-turn with down arrows forces the unique $c_2$ in row $\overline{j}$. As no other change of orientation in horizontal edges along a row is possible, this uniquely determines the remaining weights in the configuration. Indeed the first four decorations in Figure \ref{fig:six.vertex.model} give the 4 horizontal orientation preserving configurations in the six-vertex model and the column arrows from either above or below are determined inductively, working from the outside pair of rows labeled $(1, \overline{1})$ to the center. 

Recall that the arrows on column edges above row 1 point up in columns whose index is a part in a fixed partition $\mu + \rho = \lambda = [\lambda_1, \cdots, \lambda_n]$ with $\lambda_1 > \cdots > \lambda_n$. Otherwise the arrows in column edges above row 1 point down. The arrows on column edges below row $\overline{1}$ all point down. Thus the only possible locations for the unique $c_2$ vertex in a pair of rows $(j,\overline{j})$ is at a column with index $\lambda_{\sigma(j)}$, for some $\sigma(j) \in \{1,\cdots,n\}$, a part in the partition $\lambda$. We may think of the role of this $c_2$ as removing a part from the partition $\lambda$, and so one part is removed for each of the $n$ pairs of rows $(j, \overline{j})$.

These observations are recorded in the following

\begin{lemma}\label{le:weyl.group.bijection}  When using the character weights in the families $\mathfrak{B}^\lambda,\mathfrak{B}_*^\lambda,\mathfrak{C}^\lambda$ and $\mathfrak{C}_*^\lambda$ (resp., in the family $\mathfrak{D}^\lambda$ when $\lambda_n = 1$ or the family $\mathfrak{BC}^\lambda$), configurations with nonzero weights are in bijection with the Weyl group $S_n \ltimes (\pm 1)^{\oplus n}$ (resp., the Weyl group within $S_n \ltimes (\pm 1)^{\oplus n}$ consisting of those elements with an even number of $-1$ entries in the second factor), as follows.  

Let $A$ be such a configuration, and suppose $A$ is an element of one of $\mathfrak{B}^\lambda,\mathfrak{B}_*^\lambda,\mathfrak{C}^\lambda$ or $\mathfrak{C}_*^\lambda$. If there is a $c_2$ entry is row $j$, column $\lambda_k$, then the corresponding element $(\sigma,\mathbf{v})$ satisfies $\sigma(j)=k$ and $v_j = 1$; if there is a $c_2$ entry in row $\bar j$, column $\lambda_k$, then the corresponding element satisfies $\sigma(j)=k$ and $v_j=0$.  The same correspondence holds when $A$ is an element of $\mathfrak{D}^\lambda$ when $\lambda_n = 1$, with the exception that if there is a $c_2$ entry in row $\bar j$ and column $\lambda_n = 1$, then $v_j$ is chosen so that the total number of $-1$ entries in $\mathbf{v}$ is even.  Likewise if $A$ is an element of $\mathfrak{BC}^\lambda$, the value of $v_n$ is chosen so that the number of $-1$ entries in $\mathbf{v}$ is even.
\end{lemma}

%
%

\begin{remark*}
Our bijection can be thought of as encoding the data of an element $(\sigma,\mathbf{v}) \in S_n \times (\pm 1)^{\oplus n}$ in the rows of an $n \times n$ signed permutation matrix (e.g., the $j$th entry of $\mathbf{v}$ determines the sign of the nonzero entry in the $j$th row of the associated matrix).  Because this data is encoded via the rows of the associated matrix (as opposed to columns), the group operation under our bijection is slightly different than the usual presentation for the semi-direct product: if  $(\sigma_1,\mathbf{v}_1),(\sigma_2,\mathbf{v}_2)$ are given, then $$(\sigma_1,\mathbf{v}_1)(\sigma_2,\mathbf{v}_2) = (\sigma_2\sigma_1,\mathbf{v}_1\mathbf{v}_2^{\sigma_1}).$$
\end{remark*}

We now compute the weight of an admissible state $A_w$ associated to an element $w$ from the Weyl group. In what follows the Weyl vectors for these classical groups of types $B, C$, and $D$, half the sum of positive roots, are denoted
$$ \rho_B = [n-1/2,n-3/2,\ldots,3/2,1/2], \quad \rho_C=[n, n-1, \ldots, 1], \quad \rho_D = [n-1,n-2,\cdots,1,0]. $$  
The vector $\rho$ without a subscript continues to denote $n$-tuple of integers $\rho = [n, n-1, \ldots, 1]$ convenient for labeling columns of ice models.

\begin{lemma}\label{le:weight.of.weyl.group.configuration}
Write $\lambda = \mu+\rho$. Let $w$ be a given element in the Weyl group of type $B$,$C$ or $D$, and let $A_w$ be the corresponding element in $\mathfrak{M}^\lambda$, where $\mathfrak{M} \in \{\mathfrak{B},\mathfrak{B}_*,\mathfrak{C},\mathfrak{C}_*,\mathfrak{D},\mathfrak{BC}\}$.
Using the character weights, we have
$$\text{wt}(A_w) = 
\left\{\begin{array}{ll}i^{|\mu|}(-1)^n \boldsymbol{x}^{\rho_B} (-1)^{\ell(w)} \boldsymbol{x}^{w\left(\mu + \rho_B\right)},&\mbox{ if }\mathfrak{M}^\lambda \in \{\mathfrak{B}^\lambda,\mathfrak{B}_*^\lambda\}\\[5pt]
i^{|\mu|} (-1)^n \boldsymbol{x}^{\rho_C} (-1)^{\ell(w)}\boldsymbol{x}^{w(\mu+\rho_C)},&\mbox{ if }\mathfrak{M}^\lambda \in \{\mathfrak{C}^\lambda,\mathfrak{C}_*^\lambda\}\\[5pt]
 i^{|\mu|} \boldsymbol{x}^{\rho_D} (-1)^{\ell(w)} \boldsymbol{x}^{w(\mu+\rho_D)},&\mbox{ if }\mathfrak{M}^\lambda = \mathfrak{D}^\lambda\\[5pt]
 i^{|\mu|}\mathbf{x}^{\rho_B}(-1)^{\ell(w)}\mathbf{x}^{w(\mu+\rho_B)},&\mbox{ if }\mathfrak{M}^\lambda = \mathfrak{BC}^\lambda
\end{array}\right.
$$
\end{lemma}

\begin{proof}
We give a detailed proof for $\mathfrak{B}^\lambda$, then later explain how to modify this argument for the other models.

Our proof begins by determining the contribution to the Boltzmann weight of an admissible state from any pair of rows $(j, \overline{j})$. The weight will depend on whether the bend along this row points up or down.

\noindent {\bf CASE 1:} The u-turn bend in $(j,\bar j)$ points up.

In this case, the orientation along the row $j$ (not $\overline{j}$) changes exactly once, according to the location of the $c_2^{(j)}$ vertex at the column labeled $\lambda_{\sigma(j)}$. Let $\mathcal{S}_j$ to be the set of parts $\lambda_i$ of the partition $\lambda$ such that the edge above row $j$ in column $\lambda_i$ has an upward arrow. The cardinality of the set $\mathcal{S}_j$, which we denote by $\ell_j$, is equal to $r$ minus the number of rows in $[1, \ldots, j-1]$ such that the u-turn bend has an ``up'' configuration. 

Further, set $\ell_j^+$ to be the number of parts in $\mathcal{S}_j$ with $\lambda_i > \lambda_{\sigma(j)}$ (i.e. column numbers in $\mathcal{S}_j$ to the left of $\lambda_{\sigma(j)}$). Similarly let $\ell_j^-$ be the number of parts in $\mathcal{S}_j$ with $\lambda_i \leqslant \lambda_{\sigma(j)}$ (i.e. equal to or to the right of $\lambda_{\sigma(j)}$). Notice that the number of columns in row $\bar j$ which have an upward pointing arrow is precisely $\ell_j-(n-j)-1$, since the rows $(j+1,\overline{j+1}),\cdots,(n,\bar n)$ each have one $c_2$ entry that changes the corresponding column orientation from upward to downward, and additionally column $\lambda_{\sigma(j)}$ is downward-pointing.  Hence the number of columns decorated by $b_1$ in row $\bar j$ is $\ell_j-(n-j)-1$.  

Having named all this data, one can now precisely compute the contribution of the row $(j, \overline{j})$ to the Boltzmann weight of the configuration.  The figure below is an attempt to draw a sufficiently generic representation of this configuration. (The picture, but not the calculation we perform, assumes that $\lambda_{m_1} = \lambda_1$ for example.)

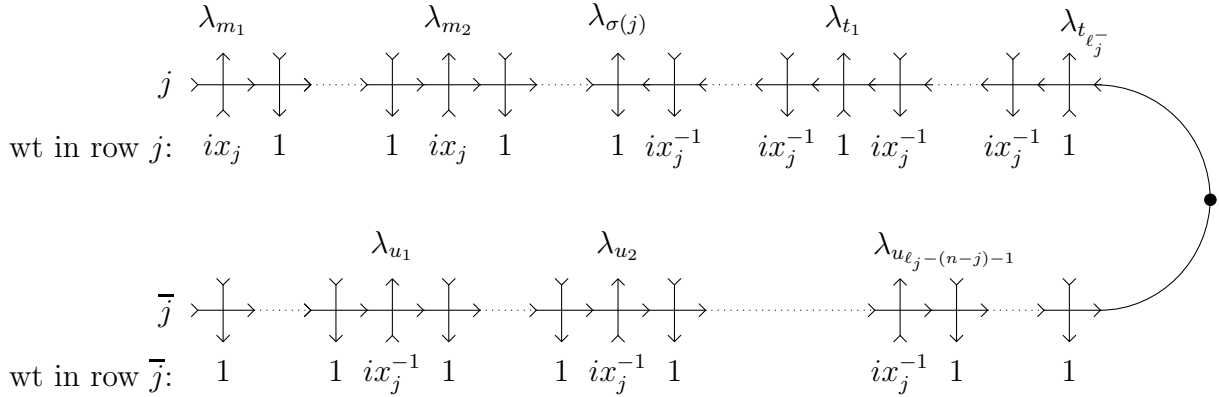
\begin{figure}[!ht]
\centering
  \begin{tikzpicture}[scale=.75]

\node [label=left:$j$] at (1,2) {};
\node [label=left:$\overline{j}$] at (1,-2) {};
\node [label=left:wt in row $j$:] at (1,0.85) {};
\node [label=left:wt in row $\overline{j}$:] at (1,-3.15) {};

\draw [>-] (1,2) -- (2,2);
\draw [>-] (1,-2) -- (2,-2);

\draw [<-<] (1.5,2.5) -- (1.5,1.5);
\node [label=below:$ix_j$] at (1.5,1.5) {};
\draw [>->] (1.5,-1.5) -- (1.5,-2.5);
\node [label=below:$1$] at (1.5,-2.5) {};

\node [label=above:$\lambda_{m_1}$] at (1.5, 2.5) {};

\draw [>->] (2,2) -- (3,2);
\draw [>->,style=dotted] (2,-2) -- (3,-2);

\draw [>->] (2.5,2.5) -- (2.5,1.5);
\node [label=below:$1$] at (2.5,1.5) {};

\draw [>-,style=dotted] (3,2) -- (4,2);
\draw [>-] (3,-2) -- (4,-2);

\draw [>->] (3.5,-1.5) -- (3.5,-2.5);
\node [label=below:$1$] at (3.5,-2.5) {};

\draw[>-] (4,2) -- (5,2);
\draw[>-] (4,-2) -- (5,-2);

\draw[>->] (4.5,2.5) -- (4.5,1.5);
\node [label=below:$1$] at (4.5,1.5) {};
\draw[<-<] (4.5,-1.5) -- (4.5,-2.5);
\node [label=above:$\lambda_{u_1}$] at (4.5, -1.5) {};
\node [label=below:$ix_j^{-1}$] at (4.5,-2.4) {};

\draw[>-] (5,2) -- (6,2);
\draw[>-] (5,-2) -- (6,-2);

\draw[<-<] (5.5,2.5) -- (5.5,1.5);
\node [label=below:$ix_j$] at (5.5,1.5) {};
\draw[>->] (5.5,-1.5) -- (5.5,-2.5);
\node [label=below:$1$] at (5.5,-2.5) {};

\node [label=above:$\lambda_{m_2}$] at (5.5, 2.5) {};

\draw[>-] (6,2) -- (7,2);
\draw[>-,style=dotted] (6,-2) -- (7,-2);

\draw[>->] (6.5,2.5) -- (6.5,1.5);
\node [label=below:$1$] at (6.5,1.5) {};

\draw[>-,style=dotted] (7,2) -- (8,2);
\draw[>-] (7,-2) -- (8,-2);

\draw[>->] (7.5,-1.5) -- (7.5,-2.5);
\node [label=below:$1$] at (7.5,-2.5) {};

\draw[>-] (8,2) -- (9,2);
\draw[>-] (8,-2) -- (9,-2);

\draw[<->] (8.5,2.5) -- (8.5,1.5);
\node [label=above:$\lambda_{\sigma(j)}$] at (8.5,2.5) {};
\node [label=below:$1$] at (8.5,1.5) {};
\draw[<-<] (8.5,-1.5) -- (8.5,-2.5);
\node [label=above:$\lambda_{u_2}$] at (8.5,-1.5) {};
\node [label=below:$ix_j^{-1}$] at (8.5,-2.4) {};

\draw[<-] (9,2) -- (10,2);
\draw[>-] (9,-2) -- (10,-2);

\draw[>->] (9.5,2.5) -- (9.5,1.5);
\node [label=below:$ix_j^{-1}$] at (9.5,1.6) {};
\draw[<-<] (9.5,-2.5) -- (9.5,-1.5);
\node [label=below:$1$] at (9.5,-2.5) {};

\draw[<-,style=dotted] (10,2) -- (11,2);
\draw[>-,style=dotted] (10,-2) -- (13,-2);


\draw[<-] (11,2) -- (12,2);

\draw[>->] (11.5,2.5) -- (11.5,1.5);
\node [label=below:$ix_j^{-1}$] at (11.5,1.6) {};

\draw[<-] (12,2) -- (13,2);

\draw[<-<] (12.5,2.5) -- (12.5,1.5);
\node [label=below:$1$] at (12.5,1.5) {};
\node [label=above:$\lambda_{t_1}$] at (12.5, 2.5) {};

\draw[<-] (13,2) -- (14,2);
\draw[>-] (13,-2) -- (14,-2);

\draw[>->] (13.5, 2.5) -- (13.5, 1.5);
\node [label=below:$ix_j^{-1}$] at (13.5,1.6) {};
\draw[<-<] (13.5, -1.5) -- (13.5, -2.5);
\node [label=above:$\lambda_{u_{\ell_j-(n-j)-1}}$] at (14.3,-1.7) {};
\node [label=below:$ix_j^{-1}$] at (13.5,-2.4) {};

\draw[<-,style=dotted] (14,2) -- (15,2);
\draw[>-] (14,-2) -- (15,-2);

\draw[>->] (14.5, -1.5) -- (14.5, -2.5);
\node [label=below:$1$] at (14.5,-2.5) {};

\draw[<-] (15,2) -- (16,2);
\draw[>-,style=dotted] (15,-2) -- (16,-2);

\draw[>->] (15.5,2.5) -- (15.5,1.5);
\node [label=below:$ix_j^{-1}$] at (15.5,1.6) {};

\draw[<-] (16,2) -- (17,2);
\draw[>-] (16,-2) -- (17,-2);

\draw[<-<] (16.5,2.5) -- (16.5,1.5);
\node [label=below:$1$] at (16.5,1.5) {};
\draw[>->] (16.5,-1.5) -- (16.5,-2.5);
\node [label=below:$1$] at (16.5,-2.5) {};

\node [label=above:$\lambda_{t_{\ell_j^-}}$] at (16.8,2.1) {};

\draw [<-]
	(17,2) arc (90:0:2);
\draw [*-<]
	(19,0) arc (0:-90:2);

\end{tikzpicture}
\end{figure}

Thus the contribution to the Boltzmann weight of the configuration from vertices in row $j$ is:
$$ (ix_j)^{\ell_j^+} (ix_j^{-1})^{\lambda_{\sigma(j)} - \ell_j^{-}} = i^{\ell_j^+ - \ell_j^- + \lambda_{\sigma(j)}} x_j^{-\lambda_{\sigma(j)} + \ell_j} $$
and from vertices from row $\overline{j}$ is
$$ (ix_j^{-1})^{j-n+\ell_j-1}.$$
Putting them together gives:
\begin{equation}\label{eq:up.arrow.row.weight}  i^{\lambda_{\sigma(j)}-1+j-n} (x_j)^{n+1-j} (-1)^{\ell_j^+} x_j^{-\lambda_{\sigma_j}} = i^{\lambda_{\sigma(j)} + n + j -1} x_j^{n-j+1/2} (-1)^{\ell_j^+-n}x_j^{-(\lambda_{\sigma(j)}-1/2)}.
\end{equation}

\noindent {\bf CASE 2:} The u-turn bend in row $(j,\bar j)$ points down.

We first define similar sets and associated cardinalities for $\overline{j}$, as in the previous case. Thus the set $\mathcal{S}_{\overline{j}}$ is the subset of parts $\lambda_i$ of $\lambda$ such that the edge below row $\overline{j}$ in column $\lambda_i$ has an upward pointing arrow. (Note that unlike $\mathcal{S}_j$, the set $\mathcal{S}_{\overline{j}}$ omits $\lambda_{\sigma(j)}$.) Set $\ell_{\overline{j}}$ to be the cardinality of $\mathcal{S}_{\overline{j}}$, the number of rows with index in $[1, \ldots, j-1]$ such that the u-turn bend arrows point down. Thus
\begin{equation} \ell_{\overline{j}} = j -1 - \# \text{(up arrows on u-turns in $[1, \ldots, j-1]$)} = \ell_j + j - n - 1. \label{ellrel} \end{equation}
Again , $\ell_{\overline{j}}^+$ is the cardinality of the subset of $\lambda_i$ in $\mathcal{S}_{\overline{j}}$ with $\lambda_i > \lambda_{\sigma(j)}$, and $\ell_{\overline{j}}^-$ the subset with $\lambda_i < \lambda_{\sigma(j)}$.

By a very similar argument to Case 1, with the roles of $j$ and $\overline{j}$ reversed, the contribution to the Boltzmann weight of the configuration from vertices in row $j$ is:
$$ (i x_j)^{\ell_j} = (i x_j)^{\ell_{\overline{j}}+n - j+1} $$
using (\ref{ellrel}). The contribution to the weight of the configuration from row $\overline{j}$ is:
$$ (i x_j^{-1})^{\ell_{\overline{j}}^+} (i x_j)^{\lambda_{\sigma(j)} - \ell_{\overline{j}}^--1} = i^{\ell_{\overline{j}}^+ - \ell_{\overline{j}}^- + \lambda_{\sigma(j)}-1} x_j^{\lambda_{\sigma(j)} - \ell_{\overline{j}}-1}. $$
Combining the two contributions and remembering that the down bend has weight $i$ in the $\mathfrak{B}$ model, the total contribution from the row $(j, \overline{j})$ is:
\begin{equation}\label{eq:down.arrow.row.weight} 
i^{\lambda_{\sigma(j)} + n - j + 1} x_j^{n-j}  (-1)^{\ell_{\overline{j}}^+} x_j^{\lambda_{\sigma(j)}} = i^{\lambda_{\sigma(j)} + n + j -1} x_j^{n-j+1/2} (-1)^{\ell_{\overline{j}}^+-j+1}x_j^{\lambda_{\sigma(j)}-1/2}. 
\end{equation}

Recall that $\lambda_{\sigma(j)} = \mu_{\sigma(j)} + \rho_{\sigma(j)}$, so that $\lambda_{\sigma(j)} - 1/2 = \mu_{\sigma(j)} + (\rho_B)_{\sigma(j)}$. Then putting the two cases together and taking the product of the weights in pairs $(j, \overline{j})$ for $j = 1, \ldots, n$, the Boltzmann weight of the admissible state is:
$$ \prod_{j=1}^n \left[ i^{\lambda_{\sigma(j)}+n+j-1} x_j^{n-j+1/2} \right] (-1)^{\phi(w)} \boldsymbol{x}^{w(\mu+\rho_B)} =  (-1)^{n} i^{|\mu|} \boldsymbol{x}^{\rho_B} (-1)^{\phi(w)} \boldsymbol{x}^{w(\mu+\rho_B)}, $$
where 
$$ \phi(w) = \sum_{j=1}^n \left\{ \begin{array}{cl} \ell_{j}^+ -n & \text{if u-turn in $(j, \overline{j})$ is up,} \\ \ell_{\overline{j}}^+  -  j + 1 & \text{if u-turn in $(j, \overline{j})$ is down} \end{array} \right\}. $$

Thus to finish the proof for the $\mathfrak{B}^\lambda$ model, it suffices to show that $\phi(w) \equiv \ell(w) \; (\text{mod } 2)$. First, we address the base case $w = e$, the identity element of $W$. According to the bijection above, this occurs when all u-turn bends are {\it down} and $\lambda_j$ is the omitted part in the pair $(j, \overline{j})$. Then $\ell_{\overline{j}}^+ = j-1$ for all $j$, and $\phi(e) = 0 = \ell(e)$ as desired.

Thus it suffices to prove that acting by a simple reflection $w \mapsto s_i \cdot w$, then $\phi(s_i w) \equiv \phi(w) + 1$ (mod 2). This is separated into two cases according to the length of the simple root.

{\bf Case A:} Replace $w$ with $s_jw$, where $s_j= ((j,j+1),\{1,\dots,1\})$.  

If $w = (\sigma,\mathbf{v})$, then $s_j w = \left(\sigma (j,j+1),\mathbf{v}^{(j,j+1)}\right)$.  At the level of admissible states, this swaps the omitted parts $\lambda_{\sigma(j)}$ and $\lambda_{\sigma(j+1)}$ in the pairs of rows $(j, \overline{j})$ and $(j+1, \overline{j+1})$, respectively, as well as exchanging the directions along the $j$ and $j+1$st bends. Let $A_w$ and $A_{s_j w}$ denote the admissible states corresponding to the elements $w$ and $s_j w$, respectively. One proves the result by examining all possible cases, though since there are more than a handful, we will focus only on two representative cases.

First, suppose $\sigma(j) > \sigma(j+1)$ and that the bend along the $j$th row points upwards, whereas the row in the $j+1$st row points downward.  Figure \ref{fig:tracking.transposition.effect} shows a schematic of the configurations $A_w$ and $A_{s_j w}$, omitting rows other than $(j,\bar j)$ and $(j+1,\overline{j+1})$ as well as columns other than $\lambda_{\sigma(j)}$ and $\lambda_{\sigma(j+1)}$.

\begin{figure}[!ht]
\begin{minipage}{2.5in}
$$
\begin{tikzpicture}[scale=.5]
\node [label=above:$\underline{A_w}$] at (3,12) {};
\node [label=left:$j$] at (0,9) {};
\node [label=left:$j+1$] at (0,7) {};
\node [label=left:$\overline{j+1}$:] at (0,3) {};
\node [label=left:$\overline{j}$:] at (0,1) {};
\node [label=above:$\cdots$] at (3,9) {};
\node [label=below:$\cdots$] at (3,1) {};
\node [label=above:$\lambda_{\sigma(j+1)}$] at (1,10) {};
\node [label=above:$\lambda_{\sigma(j)}$] at (5,10) {};

\draw[>->] (0,1) -- (2,1);
\draw[-] (2,1) -- (4,1);
\draw[>->] (4,1) -- (6,1);
\draw [>-<] (0,3) -- (2,3);
\draw[-] (2,3) -- (4,3);
\draw [<-<] (4,3) -- (6,3);
\draw [>->] (0,7) -- (2,7);
\draw[-] (2,7) -- (4,7);
\draw [>->] (4,7) -- (6,7);
\draw [>->] (0,9) -- (2,9);
\draw[-] (2,9) -- (4,9);
\draw [>-<] (4,9) -- (6,9);

\draw [<-] (1,10) -- (1,8);
\draw [<-<] (1,8) -- (1,6);
\draw [<-] (1,4) -- (1,2);
\draw [>->] (1,2) -- (1,0);
\draw [<-] (5,10) -- (5,8);
\draw [>->] (5,8) -- (5,6);
\draw [>-] (5,4) -- (5,2);
\draw [>->] (5,2) -- (5,0);
\end{tikzpicture}$$
\end{minipage}
\quad 
\begin{minipage}{2.5in}
$$
\begin{tikzpicture}[scale=.5]
\node [label=above:$\underline{A_{s_j w}}$] at (3,12) {};
\node [label=left:$j$] at (0,9) {};
\node [label=left:$j+1$] at (0,7) {};
\node [label=left:$\overline{j+1}$:] at (0,3) {};
\node [label=left:$\overline{j}$:] at (0,1) {};
\node [label=above:$\cdots$] at (3,9) {};
\node [label=below:$\cdots$] at (3,1) {};
\node [label=above:$\lambda_{\sigma(j+1)}$] at (1,10) {};
\node [label=above:$\lambda_{\sigma(j)}$] at (5,10) {};

\draw[>-<] (0,1) -- (2,1);
\draw[-] (2,1) -- (4,1);
\draw[<-<] (4,1) -- (6,1);
\draw [>->] (0,3) -- (2,3);
\draw[-] (2,3) -- (4,3);
\draw [>->] (4,3) -- (6,3);
\draw [>->] (0,7) -- (2,7);
\draw[-] (2,7) -- (4,7);
\draw [>-<] (4,7) -- (6,7);
\draw [>->] (0,9) -- (2,9);
\draw[-] (2,9) -- (4,9);
\draw [>->] (4,9) -- (6,9);

\draw [<-] (1,10) -- (1,8);
\draw [<-<] (1,8) -- (1,6);
\draw [<-] (1,4) -- (1,2);
\draw [<->] (1,2) -- (1,0);
\draw [<-] (5,10) -- (5,8);
\draw [<->] (5,8) -- (5,6);
\draw [>-] (5,4) -- (5,2);
\draw [>->] (5,2) -- (5,0);
\end{tikzpicture}$$
\end{minipage}
\caption{Configurations for $A_w$ and $A_{s_j w}$ when $\lambda_{\sigma(j)}>\lambda_{\sigma(j+1)}$: the case where $(j,\overline{j})$ is up and $(j+1,\overline{j+1})$ is down.}\label{fig:tracking.transposition.effect}
\end{figure}
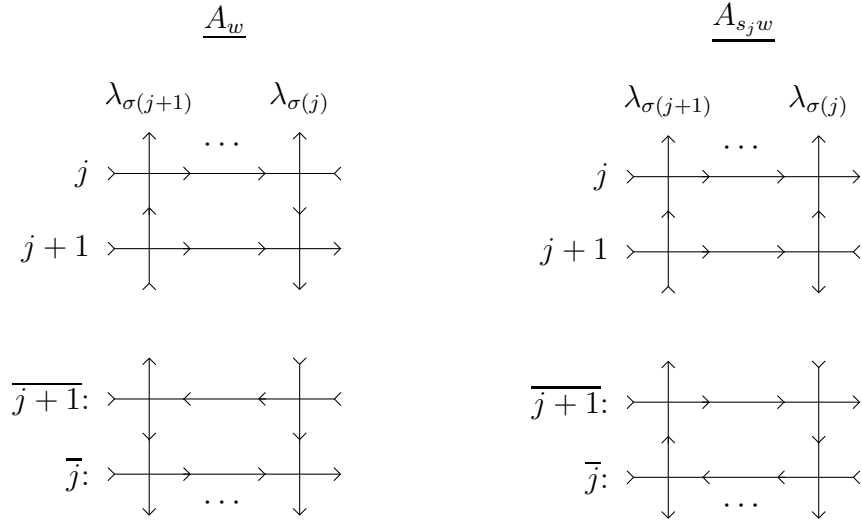

\comment{
$$
\begin{minipage}{2.5in}
\begin{centering}
\underline{u-turn at $(j,\overline{j})$ is up}
\begin{align*}
\ell^+_{j}(A_{s_j w}) &= \ell_{j+1}^+(A_w)\\
\ell_{j+1}^+(A_{ s_jw}) &= \ell_j^+(A_w)-1\\
\ell_{\overline{j}}^+(A_{s_j w}) &= \ell_{\overline{j+1}}^+(A_w)\\
\ell_{\overline{j+1}}^+(A_{s_j w}) &= \ell_{\overline{j}}^+(A_w)
\end{align*}
\end{centering}
\end{minipage}
\quad 
\begin{minipage}{2.5in}
\begin{centering}
\underline{u-turn at $(j,\overline{j})$ is down}
\begin{align*}
\ell^+_{j}(A_{ s_jw}) &= \ell_{j+1}^+(A_w)\\
\ell_{j+1}^+(A_{s_j w}) &= \ell_j^+(A_w)\\
\ell_{\overline{j}}^+(A_{s_j w}) &= \ell_{\overline{j+1}}^+(A_w)\\
\ell_{\overline{j+1}}^+(A_{s_j w}) &= \ell_{\overline{j}}^+(A_w)+1
\end{align*}
\end{centering}
\end{minipage}
$$
}

First, note that the statistics in the rows away from $j$ and $j+1$ do not change, so $\phi(w)$ and $\phi( s_jw)$ only differ in the contribution from statistics related to $(j,\overline{j})$ and $(j+1,\overline{j+1})$.  The relevant terms in the computation of $\phi(w)$ are $$\ell_j^+(A_w) -n + \ell_{\overline{j+1}}^+(A_w)-(j+1)+1,$$ whereas in $\phi(s_jw)$ the relevant terms are  
$$\ell_{\overline{j}}^+(A_{s_j w}) -j+1 + \ell_{j+1}^+(A_{s_j w}) -n.$$  By examining Figure \ref{fig:tracking.transposition.effect}, one sees that these latter terms can be replaced with $$\ell_{\overline{j+1}}^+(A_w) -j+1 + \ell_j^+(A_w)-n,$$ and hence $\phi(s_j w) \equiv \phi(w) + 1$ (mod 2) as desired.  

The other case we examine has $\sigma(j)>\sigma(j+1)$ as before, but this time both bends point upwards.  The schematic is Figure \ref{fig:tracking.transposition.effectII}, though this time we've not drawn the lower rows (since they aren't involved in calculating $\phi$ in this case).

\begin{figure}[!ht]
\begin{minipage}{2.5in}
$$
\begin{tikzpicture}[scale=.5]
\node [label=above:$\underline{A_w}$] at (3,12) {};
\node [label=left:$j$] at (0,9) {};
\node [label=left:$j+1$] at (0,7) {};
\node [label=above:$\cdots$] at (3,9) {};
\node [label=above:$\lambda_{\sigma(j+1)}$] at (1,10) {};
\node [label=above:$\lambda_{\sigma(j)}$] at (5,10) {};

\draw [>-<] (0,7) -- (2,7);
\draw[-] (2,7) -- (4,7);
\draw [<-<] (4,7) -- (6,7);
\draw [>->] (0,9) -- (2,9);
\draw[-] (2,9) -- (4,9);
\draw [>-<] (4,9) -- (6,9);

\draw [<-] (1,10) -- (1,8);
\draw [<->] (1,8) -- (1,6);
\draw [<-] (5,10) -- (5,8);
\draw [>->] (5,8) -- (5,6);
\end{tikzpicture}$$
\end{minipage}
\quad 
\begin{minipage}{2.5in}
$$
\begin{tikzpicture}[scale=.5]
\node [label=above:$\underline{A_{ s_jw}}$] at (3,12) {};
\node [label=left:$j$] at (0,9) {};
\node [label=left:$j+1$] at (0,7) {};
\node [label=above:$\cdots$] at (3,9) {};
\node [label=above:$\lambda_{\sigma(j+1)}$] at (1,10) {};
\node [label=above:$\lambda_{\sigma(j)}$] at (5,10) {};

\draw [>->] (0,7) -- (2,7);
\draw[-] (2,7) -- (4,7);
\draw [>-<] (4,7) -- (6,7);
\draw [>-<] (0,9) -- (2,9);
\draw[-] (2,9) -- (4,9);
\draw [<-<] (4,9) -- (6,9);

\draw [<-] (1,10) -- (1,8);
\draw [>->] (1,8) -- (1,6);
\draw [<-] (5,10) -- (5,8);
\draw [<->] (5,8) -- (5,6);
\end{tikzpicture}$$
\end{minipage}
\caption{Configurations for $A_w$ and $A_{ s_jw}$ when $\lambda_{\sigma(j)}>\lambda_{\sigma(j+1)}$: the case where $(j,\overline{j})$ and $(j+1,\overline{j+1})$ are up.}\label{fig:tracking.transposition.effectII}
\end{figure}

Notice that again the statistics away from rows $(j,\overline{j})$ and $(j+1,\overline{j+1})$ do not change.  The contributions from these rows in $A_w$ is
$$\ell_j^+(A_w)-n+\ell_{j+1}^+(A_w)-n,$$
whereas the contributions from these rows in $A_{s_jw}$ is
$$\ell_{j}^+(A_{s_jw})-n + \ell_{j+1}^+(A_{s_jw})-n = \ell_{j+1}^+(A_w) - n + \ell_{j}^+(A_w)-1-n.$$
The desired result again follows.

The other cases are handled in a similar way. 

{\bf Case B:} Replace $w$ with $s_nw$, where $s_n$ is the simple reflection $s_n=(\text{id},\{1,1,\dots,1,-1\})$. 

At the level of admissible states, this amounts to changing the orientation of the u-turn bend in the pair $(n, \overline{n})$. For example, if the bend is changed from up to down, the omitted part $\lambda_k$ swaps from row $n$ to row $\overline{n}$. If $A_w$ and $A_{s_n w}$ are the corresponding admissible states, this alters $\phi(w)$ by sending
$$ \ell_n^+(A_w)-n \longmapsto \ell_{\overline{n}}^+(A_{s_n w}) -n + 1. $$ 
But $\ell_n^+ = \ell_{\overline{n}}^+$, since the orientation along any column not equal to $\lambda_k$ doesn't change between the top of row $n$ and the bottom of row $\overline{n}$.  
Moreover, all other statistics in pairs of rows $(j, \overline{j})$ with $j \in [1, \ldots, n-1]$ remain unchanged. Thus $\phi(s_{n} w) \equiv \phi(w) + 1$ (mod 2) as desired. The case when the bend is changed from down to up is identical.

This concludes the proof in the case of model $\mathfrak{B}^\lambda$.  The proofs for the other families are very similar, with small changes coming from the total weight along row $(j,\bar j)$ and --- in the case of $\mathfrak{D}^\lambda$ and $\mathfrak{BC}^\lambda$ --- a small change to the function $\phi$.

First, in model $\mathfrak{B}_*^\lambda$, note that the weight along $(j,\bar j)$ when the bend points up is identical to above; when the bend points down, on the other hand, we do not get an additional factor of $i$ coming from the bend, but we do pick up a factor of $i$ from the $b^{(0)}$ configuration that will appear in column $\lambda_{\sigma(j)}$ in the central row.  If we include this contribution from the central row into the total contribution from row $(j,\bar j)$ the total contributions in the up and down bend cases agree with Equations (\ref{eq:up.arrow.row.weight}) and (\ref{eq:down.arrow.row.weight}), and we can proceed as in the previous case.


\comment{ 
A similar phenomenon occurs in the model $\mathfrak{B}^{**}$.  In this case, however, if $(j,\bar j)$ has an upward pointing arrow then the corresponding entry in the central row in column $\lambda_{\sigma(j)}$ has a $b^{(0)}$ configuration, and hence contributes a factor of $i$.  Moreover, the down arrows in this model are weighted as $-1$.  Factoring these into the total weight along row $(j,\bar j)$ gives
\begin{align*}
\mbox{up arrow} &\longrightarrow i^{\lambda_k+j+r} x_j^{r-j+1/2} (-1)^{\ell_j^+-r}x_j^{-(\lambda_{\sigma(j)}-1/2)}\\
\mbox{down arrow} &\longrightarrow i^{\lambda_k+j+r} x_j^{r-j+1/2} (-1)^{\ell_j^+-j+1}x_j^{\lambda_{\sigma(j)}-1/2}.\\
\end{align*}
Now when we take the product along all $r$ rows, we also need to account for the $b^{(0)}$ entries in the central row that we get for each column which isn't a part of $\lambda$, of which there are $\lambda_1-r$.  This contributes an additional factor of $i^{\lambda_1-r}$.
}


In the $\mathfrak{C}^\lambda$ model, an upward bend in row $(j,\bar j)$ gives the same total weight as in Equation (\ref{eq:up.arrow.row.weight}) for the $\mathfrak{B}^\lambda$ model.  When $(j,\bar j)$ points down, one gains an additional factor of $-i$ from row $0$, column $\lambda_{\sigma(j)}$ (since there will be a $b^{(0)}$ entry there which is weighted $ix_0t_0 = -i$), as well as an additional factor of $ix_j$ from a $b_2$ vertex in the half-column.  Hence the total weight along such a row has its power of $i$ changed by $4$ (and therefore, without effect), but also its power of $x_j$ increased by one.  It follows that along row $(j,\bar j)$ we have a total weight of
\begin{align*}
\mbox{up arrow} &\longrightarrow i^{\lambda_k+j+n-1} x_j^{n-j+1} (-1)^{\ell_j^+-n}x_j^{-\lambda_{\sigma(j)}}\\
\mbox{down arrow} &\longrightarrow i^{\lambda_k+j+n-1} x_j^{n-j+1} (-1)^{\ell_{\overline{j}}^+-j+1}x_j^{\lambda_{\sigma(j)}}\\
\end{align*}
(The same is true in the $\mathfrak{C}_*^\lambda$ model, since the only difference in this model is that we lose a factor of $i$ along a downward bend as well as a factor of $-i$ since we don't have a central row.)

The $\mathfrak{D}^\lambda$ model has a few additional peculiarities.  Notice that the $c_2$ entry in column $\lambda_n = 1$ must occur in the bottom portion of its row, even though the corresponding element in the Weyl group might not take value $+1$.  For the sake of bookkeeping, the consistent choice is to include column $\lambda_n$ when computing the size of $\mathcal{S}_j$ when $(j,\bar j)$ points upward; it won't affect the size of $\ell_j^+$, since this column is assuredly to the right of column $\lambda_{\sigma(j)}$.  With this in hand, when row $(j,\bar j)$ has an upward pointing bend its total weight again matches that from Equation (\ref{eq:up.arrow.row.weight}).  When $(j,\bar j)$ points downward, however, we lose a factor of $ix_j$ from the contribution in column $\lambda_n=1$ and row $j$ (since there is no vertex there).  Since the downward bend is weighted as $1$ in this model (instead of $i$ as in the $\mathfrak{B}$ model), one then gets the following expression for the total weight along $(j,\bar j)$:
\begin{align*}
\mbox{up arrow} &\longrightarrow i^{\lambda_{\sigma(j)}+j+n+1} x_j^{n-j} (-1)^{\ell_j^+-n+1}x_j^{-(\lambda_{\sigma(j)}-1)}\\
\mbox{down arrow} &\longrightarrow i^{\lambda_{\sigma(j)}+j+n+1} x_j^{n-j} (-1)^{\ell_{\overline{j}}^+-j+1}x_j^{\lambda_{\sigma(j)}-1}\\
\end{align*}
At this point, most of the argument precedes as in the case of the $\mathfrak{B}^\lambda$ model, except now the function $\phi$ has changed slightly.  Note that it still agrees with the earlier definition of $\phi$ for downward pointing arrows, and hence gives the correct parity for the identity element.  Now note that changes in the $(\pm 1)^{\oplus n}$ component of the Weyl group come in pairs and do not change the parity of the length function; one can show that the analog of Case B in this setting doesn't change the parity of $\phi$, and hence finish the proof in this case.

Finally we settle the $\mathfrak{BC}^\lambda$ model.  In this case we will have to keep track not only contributions to the weight of $A_w$ in rows $(j,\overline{j})$ for $1 \leq j \leq n-1$, but also the weight the central row labeled $n$.  First, if $1 \leq j \leq n-1$ and the bend in row $(j,\overline{j})$ points upward, then the weight coming from row $(j,\overline{j})$ matches the weight computed in Equation \ref{eq:up.arrow.row.weight} for the $\mathfrak{B}^\lambda$ model.  If the arrow points down, the only difference with the $\mathfrak{B}^\lambda$ model is that there is no weight of $i$ coming from the bend.  Nontrivial terms are contributed to the weight of the central row by those columns to the left of $\lambda_{\sigma(n)}$ which point up along row $n$, as well as those columns to the right of $\lambda_{\sigma(n)}$ which point down along row $n$.  In each case, the factor that is contributed is $i$.  Hence we have
\begin{align*}
\mbox{up arrow} &\longrightarrow i^{\lambda_{\sigma(j)}+j+n-1} x_j^{n-j+1/2} (-1)^{\ell_j^+-n}x_j^{-(\lambda_{\sigma(j)}-1/2)}\\
\mbox{down arrow} &\longrightarrow i^{\lambda_{\sigma(j)}+j+n} x_j^{n-j+1/2} (-1)^{\ell_{\overline{j}}^+-j}x_j^{\lambda_{\sigma(j)}-1/2}\\
\mbox{central row} &\longrightarrow i^{\ell_n^+} i^{\lambda_{\sigma(n)}-\ell_n^-} = i^{\lambda_{\sigma(n)}+\ell_n}(-1)^{-\ell_n^-}\\
\end{align*}
Taking a product over all rows and doing a bit of algebra (taking into account that $\ell_n$ equals $n$ minus the number of upward bends), one finds that the total weight is
$$\text{wt}(A_w) = i^{|\mu|} \mathbf{x}^{\rho_B} (-1)^{\phi(w)} \mathbf{x}^{w(\mu+\rho_B)},$$ where $\phi$ is the function defined by
$$\phi(w) = \left(\sum_{j=1}^{n-1} \left\{\begin{array}{cl}
\ell_j^+-n&\text{ if u-turn in }(j,\overline{j}) \text{ is up,}\\
\ell_{\overline{j}}^+-j+1&\text{ if u-turn in }(j,\overline{j}) \text{ is down}
\end{array}\right\}\right)-\ell_n^-+1.
$$
One proves that $\phi(w) \equiv \ell(w) \mod{2}$ as before.  (Note in particular that since $w$ is an element of the Weyl group corresponding to type $D$, one has $\phi(w) \equiv \phi(s_{n-1}w) \mod{2}$ for the element $s_{n-1} = (\text{id},\{1,1,\dots,1,-1,-1\}$.)
\end{proof}

By summing across all elements in a given family, we arrive at the following

\begin{theorem}\label{weylchartheorem} Let $\mathfrak{M} \in \{\mathfrak{B},\mathfrak{B}_*,\mathfrak{C},\mathfrak{C}_*,\mathfrak{D},\mathfrak{BC}\}$ be given. Write $\lambda = \mu + \rho$, where as before $\rho = [n,n-1,\ldots,1]$.  Then using the character weights we have
$$ \mathcal{Z}(\mathfrak{M}^\lambda) =  i^{|\mu|} \mathcal{Z}(\mathfrak{M}^{\rho}) \; \chi_{\mu}(\boldsymbol{x}), $$
where $\chi_{\mu}$ is the highest weight character of a type $\mathfrak{M}$ representation corresponding to the integer partition $\mu$.
\end{theorem}

\begin{proof}
We will prove the result for $\mathfrak{B}^\lambda$, as the other results are essentially identical.  We will use the fact that $$\chi_\mu(\boldsymbol{x}) = \frac{\sum_{w \in W} (-1)^{\ell(w)} \boldsymbol{x}^{w(\mu+\rho_B)}}{\sum_{w \in W} (-1)^{\ell(w)} \boldsymbol{x}^{w(\rho_B)}}.$$  For any $\lambda$, Lemma \ref{le:weight.of.weyl.group.configuration} tells us \begin{align*}
\mathcal{Z}(\mathfrak{B}^\lambda) = \sum_{w \in W} \text{wt}(A_w) &= \sum_{w \in W} i^{|\mu|} (-1)^n \boldsymbol{x}^{\rho_B} (-1)^{\ell(w)} \boldsymbol{x}^{w(\mu+\rho_B)}.
\end{align*} Hence
\begin{align*}
\mathcal{Z}(\mathfrak{B}^\lambda)&=i^{|\mu|}(-1)^n\boldsymbol{x}^{\rho_B}\sum_{w \in W} (-1)^{\ell(w)}\boldsymbol{x}^{w(\mu+\rho_B)}\\
&=i^{|\mu|}(-1)^n\boldsymbol{x}^{\rho_B}\chi_\mu(\boldsymbol{x}) \sum_{w \in W} (-1)^{\ell(w)}\boldsymbol{x}^{w(\rho_B)}\\
&=i^{|\mu|}\left( \sum_{w \in W} i^0(-1)^n\boldsymbol{x}^{\rho_B}(-1)^{\ell(w)}\boldsymbol{x}^{w(0+\rho_B)}\right)\chi_\mu(\boldsymbol{x})\\
&=i^{|\mu|}\mathcal{Z}(\mathfrak{B}^\rho)\chi_\mu(\boldsymbol{x}),
\end{align*} where the last equality comes from applying the previous equation to $\lambda = \rho$.
\end{proof}

\bibliographystyle{acm}
\bibliography{ice}

\def\cprime{$'$}
\begin{thebibliography}{10}

\bibitem{baxter}
{\sc Baxter, R.~J.}
\newblock {\em Exactly solved models in statistical mechanics}.
\newblock Academic Press Inc. [Harcourt Brace Jovanovich Publishers], London,
  1989.
\newblock Reprint of the 1982 original.

\bibitem{bbf-ice}
{\sc Brubaker, B., Bump, D., and Friedberg, S.}
\newblock Schur polynomials and the {Y}ang-{B}axter equation.
\newblock {\em Comm. Math. Phys. 308}, 2 (2011), 281--301.

\bibitem{bbf-annals}
{\sc Brubaker, B., Bump, D., and Friedberg, S.}
\newblock Weyl group multiple {D}irichlet series, {E}isenstein series and
  crystal bases.
\newblock {\em Ann. of Math. (2) 173}, 2 (2011), 1081--1120.

\bibitem{fan-wu}
{\sc Fan, C., and Wu, F.~Y.}
\newblock General lattice model of phase transitions.
\newblock {\em Phys. Rev. B 2\/} (Aug 1970), 723--733.

\bibitem{hk-symplectic}
{\sc Hamel, A.~M., and King, R.~C.}
\newblock Symplectic shifted tableaux and deformations of {W}eyl's denominator
  formula for {${\rm sp}(2n)$}.
\newblock {\em J. Algebraic Combin. 16}, 3 (2002), 269--300.

\bibitem{hk-bijective}
{\sc Hamel, A.~M., and King, R.~C.}
\newblock Bijective proofs of shifted tableau and alternating sign matrix
  identities.
\newblock {\em J. Algebraic Combin. 25}, 4 (2007), 417--458.

\bibitem{ivanov}
{\sc Ivanov, D.}
\newblock Symplectic ice.
\newblock In {\em Multiple {D}irichlet series, {L}-functions and automorphic
  forms}, vol.~300 of {\em Progr. Math.} Birkh\"auser/Springer, New York, 2012,
  pp.~205--222.

\bibitem{korepinetal}
{\sc Korepin, V.~E., Bogoliubov, N.~M., and Izergin, A.~G.}
\newblock {\em Quantum inverse scattering method and correlation functions}.
\newblock Cambridge Monographs on Mathematical Physics. Cambridge University
  Press, Cambridge, 1993.

\bibitem{kuperberg}
{\sc Kuperberg, G.}
\newblock Symmetry classes of alternating-sign matrices under one roof.
\newblock {\em Ann. of Math. (2) 156}, 3 (2002), 835--866.

\bibitem{mcnamara}
{\sc McNamara, P.~J.}
\newblock Metaplectic {W}hittaker functions and crystal bases.
\newblock {\em Duke Math. J. 156}, 1 (2011), 1--31.

\bibitem{mv-satake}
{\sc Mirkovi{\'c}, I., and Vilonen, K.}
\newblock Geometric {L}anglands duality and representations of algebraic groups
  over commutative rings.
\newblock {\em Ann. of Math. (2) 166}, 1 (2007), 95--143.

\bibitem{okada}
{\sc Okada, S.}
\newblock Alternating sign matrices and some deformations of {W}eyl's
  denominator formulas.
\newblock {\em J. Algebraic Combin. 2}, 2 (1993), 155--176.

\bibitem{proctor}
{\sc Proctor, R.~A.}
\newblock Odd symplectic groups.
\newblock {\em Invent. Math. 92}, 2 (1988), 307--332.

\bibitem{razumovstroganov}
{\sc Razumov, A.~V., and Stroganov, Y.~G.}
\newblock Enumeration of odd-order alternating-sign half-turn-symmetric
  matrices.
\newblock {\em Teoret. Mat. Fiz. 148}, 3 (2006), 357--386.

\bibitem{simpson}
{\sc Simpson, T.}
\newblock Another deformation of {W}eyl's denominator formula.
\newblock {\em J. Combin. Theory Ser. A 77}, 2 (1997), 349--356.

\bibitem{sage}
{\sc Stein, W., et~al.}
\newblock {\em {S}age {M}athematics {S}oftware ({V}ersion 5.8)}.
\newblock The Sage~Development Team, 2012.
\newblock {\tt http://www.sagemath.org}.

\bibitem{tabony}
{\sc Tabony, S.~J.}
\newblock {\em Deformations of characters, metaplectic {W}hittaker functions,
  and the {Y}ang-{B}axter equation}.
\newblock PhD thesis, Massachusetts Institute of Technology, 2010.

\bibitem{tokuyama}
{\sc Tokuyama, T.}
\newblock A generating function of strict {G}el\cprime fand patterns and some
  formulas on characters of general linear groups.
\newblock {\em J. Math. Soc. Japan 40}, 4 (1988), 671--685.

\bibitem{yamamoto-tsuchiya}
{\sc Yamamoto, T., and Tsuchiya, O.}
\newblock Integrable {$1/r^2$} spin chain with reflecting end.
\newblock {\em J. Phys. A 29}, 14 (1996), 3977--3984.

\bibitem{zhelobenko}
{\sc {\v{Z}}elobenko, D.~P.}
\newblock Classical groups. {S}pectral analysis of finite-dimensional
  representations.
\newblock {\em Uspehi Mat. Nauk 17}, 1 (103) (1962), 27--120.

\end{thebibliography}

\end{document}